\documentclass[letterpaper,11pt,oneside,reqno, smallextended]{amsart}

\usepackage{bm}
\usepackage{caption}
\usepackage{afterpage}
\usepackage{enumitem}
\usepackage{bmpsize}
\usepackage{hyperref}

\usepackage{amsmath,amssymb,amsthm,amsfonts, dsfont}
\usepackage{relsize}
\usepackage{hyperref}
\usepackage{graphicx,color,subcaption}
\usepackage{upgreek}
\usepackage[mathscr]{euscript}

\allowdisplaybreaks
\numberwithin{equation}{section}

\usepackage{tikz}
\usetikzlibrary{shapes,arrows,positioning,decorations.markings}

\usepackage{float}
\usepackage{multicol}
\usepackage{array}
\usepackage{adjustbox}
\usepackage{cleveref}
\usepackage[DIV12]{typearea}

\newtheorem{proposition}{Proposition}[section]
\newtheorem{lemma}[proposition]{Lemma}
\newtheorem{corollary}[proposition]{Corollary}
\newtheorem{theorem}[proposition]{Theorem}
\newtheorem{maintheorem}{Theorem}
\newtheorem{mainprop}[maintheorem]{Proposition}
\theoremstyle{definition}
\newtheorem{definition}[proposition]{Definition}
\newtheorem{remark}[proposition]{Remark}
\newtheoremstyle{qqq}
  {}   % ABOVESPACE
  {}   % BELOWSPACE
  {\slshape}  % BODYFONT
  {0pt}       % INDENT (empty value is the same as 0pt)
  {\bfseries} % HEADFONT
  {.}         % HEADPUNCT
  {5pt plus 1pt minus 1pt} % HEADSPACE
  {}          % CUSTOM-HEAD-SPEC
\theoremstyle{qqq}
\newtheorem{question}{Question}
\newtheorem{conjecture}[question]{Conjecture}
\begin{document}
\title{Mapping TASEP back in time}

\author[L. Petrov]{Leonid Petrov}
\address{L. Petrov, University of Virginia, Department of Mathematics,
141 Cabell Drive, Kerchof Hall,
P.O. Box 400137,
Charlottesville, VA 22904, USA,
and Institute for Information Transmission Problems,
Bolshoy Karetny per. 19, Moscow, 127994, Russia}
\email{lenia.petrov@gmail.com}

\author[A. Saenz]{Axel Saenz}
\address{A. Saenz, University of Virginia, Department of Mathematics,
141 Cabell Drive, Kerchof Hall,
P.O. Box 400137,
Charlottesville, VA 22904, USA
}
\email{saenzaxel@gmail.com}

\date{}

\begin{abstract}
	We obtain a new relation between the distributions $\upmu_t$ at
	different times $t\ge 0$ of the continuous-time TASEP (Totally Asymmetric Simple
	Exclusion Process) started from the step initial configuration.
	Namely, we present a continuous-time Markov process with local
	interactions and particle-dependent rates which maps the TASEP
	distributions $\upmu_t$ backwards in time. 
	Under the backwards process, particles jump to the left, and 
	the dynamics
	can be
	viewed as a version of the discrete-space Hammersley process.
	Combined with the forward TASEP evolution, this leads to a
	stationary Markov dynamics preserving $\upmu_t$
	which in turn brings new identities for expectations with respect to $\upmu_t$.
	
	The construction
	of the backwards dynamics is based on Markov maps interchanging
	parameters of Schur processes, and is motivated by bijectivizations
	of the Yang-Baxter equation. 
	We also present a number of corollaries, 
	extensions, and open questions 
	arising from our constructions.
\end{abstract}

\maketitle

\setcounter{tocdepth}{3}

%----------------------------------------------------------------------------

\section{Introduction}

\subsection{TASEP}

The Totally Asymmetric Simple Exclusion Process
(TASEP)
is a prototypical stochastic model 
of transport in one dimension. Introduced around 
50 years ago in parallel in biology
\cite{macdonald1968bioASEP},
\cite{MacdonaldGibbsASEP1969}
and probability theory \cite{Spitzer1970},
it has been extensively studied
by a variety of methods.

TASEP is a continuous-time Markov process on 
the space of 
particle configurations in $\mathbb{Z}$ in which at most one
particle per site is allowed.
Each particle has an independent exponential 
clock of rate $1$ (that is, the random time~$T$ after 
which the clock rings is distributed
as 
$\mathrm{Prob}(T>s)=e^{-\uplambda s}$, $s>0$, 
where $\uplambda=1$ is the rate). When the 
clock rings, the particle jumps to the 
right by one if the destination is free of a particle.
Otherwise, the jump is blocked and nothing happens. 
See \Cref{fig:intro_TASEP} for an illustration.

\begin{figure}[htpb]
	\centering
	\includegraphics{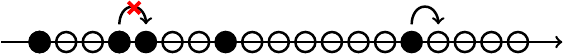}
	\caption{A forbidden jump (on the left) and a jump (on the right)
		in TASEP.}
	\label{fig:intro_TASEP}
\end{figure}

In this work we focus on
the process with 
the most well-studied initial condition
--- the step initial condition.
Under it,
the particles
initially occupy $\mathbb{Z}_{<0}$,
while $\mathbb{Z}_{\ge0}$ is free of particles.
Denote by $h(t,x)$ the TASEP interface 
(where $t\in \mathbb{R}_{\ge0}$,
$x\in \mathbb{Z}$), which is obtained
by placing a slope $+1$ or a slope $-1$ segment
over a hole or particle, respectively,
with the agreement that 
the step initial configuration corresponds to $h(0,x)=|x|$.
See \Cref{fig:TASEP_BHP} for an illustration.
We also denote the TASEP distribution at time $t$
(with step initial condition)
by $\upmu_t$.

It was shown by \cite{Rost1981}
(see also, e.g., 
\cite{johansson2000shape},
\cite[Chapter 4]{romik2015surprising}
for an alternative approach based on symmetric functions) 
that the interface grows linearly with time and tends to the limit shape,
under the \emph{hydrodynamic scaling} (i.e.~linear space and time scaling),
which is a parabola:
\begin{equation}
	\label{eq:TASEP_limit_shape}
	\frac{1}{L}\, h(\tau L,\varkappa L)
	\to
	\frac{\varkappa^2+\tau^2}{2\tau},
	\qquad L\to+\infty,
\end{equation}
where $\varkappa$ and $\tau$ are scaled space and time,
and $|\varkappa|\le \tau$.

In the past 20 years, starting with \cite{johansson2000shape},
much finer results about asymptotic behavior
of TASEP have become available through the tools of  
Integrable Probability
(cf. 
\cite{BorodinGorinSPB12},
\cite{BorodinPetrov2013Lect}).
This asymptotic analysis revealed that 
TASEP belongs to the 
(one-dimensional) Kardar-Parisi-Zhang 
(KPZ)
universality class
\cite{CorwinKPZ},
\cite{QuastelSpohnKPZ2015}.
In particular, 
the TASEP interface at time $L$, on the
horizontal $L^{2/3}$ and vertical $L^{1/3}$ scales,
converges to the Airy$_2$ process,
which is the top line of the 
Airy$_2$ line ensemble 
(about the latter see, e.g., \cite{corwin2014brownian}).
Furthermore, computations with TASEP allow 
to formulate general predictions for all one-dimensional
systems in the KPZ class (e.g., see \cite{Ferrari_Airy_Survey}, \cite{Spohn2012}).
The progress in understanding 
multitime 
asymptotics of the TASEP interfaces
is rapidly 
advancing at present
(see \Cref{rmk:finer_scaling} for 
references to recent results).

\subsection{The backwards dynamics}
\label{sub:intro_BHP}

The goal of our work is to present a 
new surprising property of the 
family of TASEP distributions $\left\{ \upmu_t \right\}_{t\ge0}$.
We show that the distributions $\upmu_t$
are coupled 
in the reverse time direction 
by a time-homogeneous Markov process with 
local interactions 
(the interaction strength 
depends on the location in the system).
Let us now describe this backwards dynamics.

Denote by $\mathcal{C}$ 
the (countable) space of 
configurations on $\mathbb{Z}$ 
which differ from the step 
configuration 
by finitely many TASEP 
jumps.\footnote{In 
other words, $\mathcal{C}$ consists of 
configurations 
$\{x_1>x_2>x_3>\ldots \}\subset\mathbb{Z}$
which possess a rightmost particle $x_1$,
and such that 
$x_N=-N$
for all $N$ large enough.}
Consider the continuous-time Markov chain on $\mathcal{C}$
which evolves as follows. At each hole
there is an independent exponential
clock whose rate is equal to the number of particles 
to the right of this hole.
When the clock at a hole rings, 
the leftmost of the particles that are to the right of the hole
instantaneously jumps into this hole
(in particular, the particles almost surely jump to the left).
See \Cref{fig:intro_BHP} for an illustration 
or \eqref{eq:BHP_gen} for a description of the generator.
Note that, for configurations in $\mathcal{C}$,
almost surely at most one particle can move
at any time moment 
because there are only finitely many holes with 
nonzero rate.

\begin{figure}[htpb]
	\centering
	\includegraphics{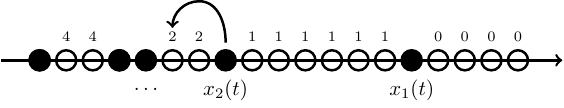}
	\caption{An illustration of the backwards 
	process. Jump rates attached to holes
	and a possible jump are indicated.}
	\label{fig:intro_BHP}
\end{figure}

The jumping mechanism described above 
has the following features:
\begin{itemize}
	\item 
		gaps attract neighboring particles 
		from the right;
	\item 
		the rate of attraction is proportional to the
		size of the gap;
	\item 
		the jumping particle
		lands inside the gap uniformly at random.
\end{itemize}
The same features 
of the jumping mechanism
appear in the well-known 
continuous-space
Hammersley process
\cite{hammersley1972few}, \cite{aldous1995hammersley},
and the discrete-space Hammersley process
\cite{ferrari1996}, \cite{FerrariMartin2005}.
For this reason
we call our Markov process 
(which evolves in the discrete space)
the \emph{backwards Hammersley-type
process}, or BHP, for short.
Note that compared to the well-known continuous-space 
Hammersley
process, our BHP is \emph{space-inhomogeneous}: the 
jump rate at a hole also depends on the number
of particles to the right of it.
The evolutions of the interface under TASEP and the BHP 
are given in \Cref{fig:TASEP_BHP}.

Let $\{\mathbf{L}_\tau\}_{\tau\in \mathbb{R}_{\ge0}}$ be
the Markov semigroup of the BHP defined
in \Cref{sub:intro_BHP}. That is, 
$\mathbf{L}_\tau(\mathbf{x},\mathbf{y})$,
$\mathbf{x}, \mathbf{y}\in \mathcal{C}$,
is the probability that the particle configuration
is $\mathbf{y}$ at time $\tau$ 
given that it started at $\mathbf{x}$ at time $0$
(here we use the fact that BHP is time-homogeneous).

\begin{remark}
	The backwards process is well-defined. Indeed, 
	for each initial condition $\mathbf{x}\in \mathcal{C}$
	of the backwards process, the set of its possible further states 
	is finite. Therefore, the probability
	$\mathbf{L}_{\tau}(\mathbf{x},\mathbf{y})$
	for any 
	$\mathbf{x},\mathbf{y}\in \mathcal{C}$
	is well-defined (and can be obtained by exponentiating the corresponding
		finite-size
	piece of the BHP jump matrix).
\end{remark}

\subsection{Main result}

Recall that $\upmu_t$ is the 
distribution of the TASEP configuration
at time $t$ (with the step initial condition).
The measure $\upmu_t$ is supported on the space
$\mathcal{C}$ for all $t\ge0$.

\begin{maintheorem}
	\label{thm:intro_action_on_TASEP}
	The BHP maps the TASEP distributions backwards in time.
	That is,
	for any $t,\tau\in \mathbb{R}_{\ge0}$, we have
	\begin{equation}
		\label{eq:intro_TASEP_short}
		\upmu_t \,\mathbf{L}_\tau=\upmu_{\,e^{-\tau}t}.
	\end{equation}
	In detail, this identity means that
	for any $\mathbf{x}\in \mathcal{C}$ 
	we have
	\begin{equation*}
		\sum_{\mathbf{y}\in\mathcal{C}}
		\upmu_{t}(\mathbf{y})
		\,
		\mathbf{L}_\tau(\mathbf{y},\mathbf{x})
		=
		\upmu_{\,e^{-\tau}t}(\mathbf{x}).
	\end{equation*}
\end{maintheorem}
\begin{figure}[p]
	\centering
	\vspace{10pt}
	\boxed{\includegraphics[width=\textwidth]{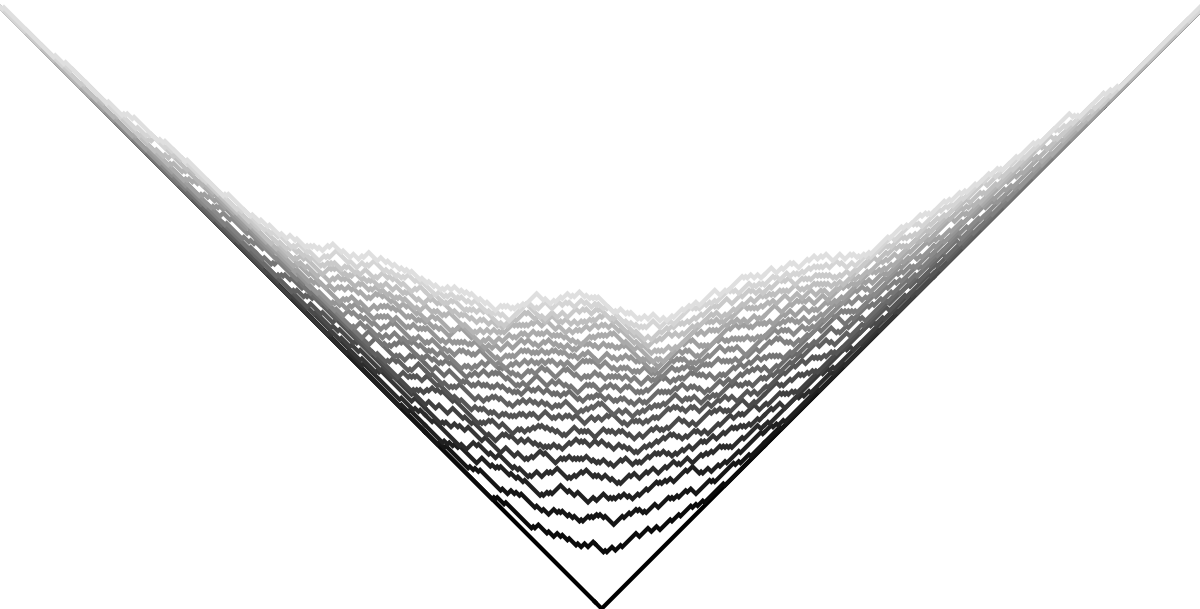}}

	\vspace{20pt}
	\boxed{\includegraphics[width=\textwidth]{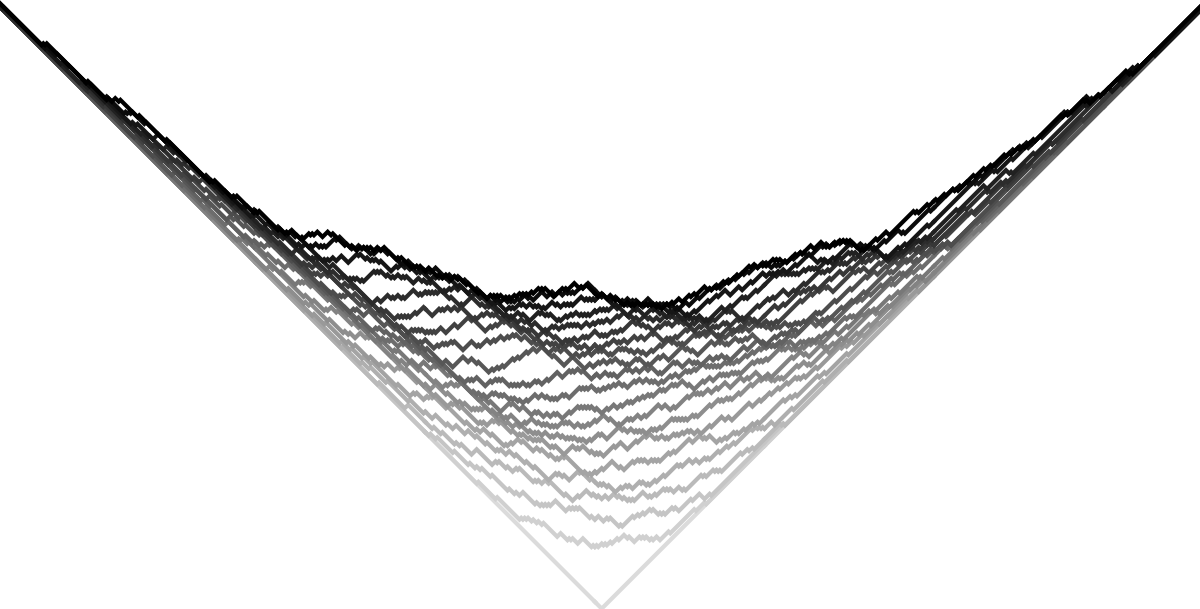}}
	\vspace{10pt}
	\caption{An illustration of the TASEP interface
	growth (top) and the interface decay under the backwards dynamics 
	(bottom). In both pictures, lighter curves are
	the interfaces at later times. One can see that the 
	TASEP evolution is symmetric about 
	the vertical axis, while the 
	backwards dynamics is not symmetric.
	Because of this asymmetry, there are 
	in fact two backwards processes --- one focusing on 
	holes and the other focusing on particles. We
	only consider one of them in the present work.}
	\label{fig:TASEP_BHP}
\end{figure}
As $\tau\to+\infty$, the right-hand side of 
\eqref{eq:intro_TASEP_short} becomes $\upmu_0$, which is the 
delta measure on the step configuration.
This agrees with the observation 
that for any $\mathbf{x}\in \mathcal{C}$ we 
have\footnote{Throughout
the paper
$\mathds{1}_{E}$ stands for the indicator function
if the event $E$.}
\begin{equation*}
	\lim_{\tau\to+\infty}\mathbf{L}_\tau(\mathbf{x},\mathbf{y})=
	\mathds{1}_{\mathbf{y}=\textnormal{step configuration}}.
\end{equation*}

\Cref{thm:intro_action_on_TASEP} leads to a stationary
Markov dynamics on the TASEP measure $\upmu_t$
(it is discussed in
\Cref{sub:intro_equil} below). In particular, this stationary
dynamics brings new identities for expectations with respect
to $\upmu_t$. One of these identities
is given in \Cref{cor:microscopic_equation}.

The simulation depicting the TASEP evolution
from the step initial configuration to $t=350$, 
and then the action of the BHP on this interface is 
available online 
\cite{PetrovLi2019simul}.
The interfaces in 
\Cref{fig:TASEP_BHP}
are snapshots of this simulation.

\subsection{Remark. Reversal of Markov processes}
\label{sub:intro_Markov_reversal}

Before discussing the strategy of the proof
of \Cref{thm:intro_action_on_TASEP}
let us mention that TASEP, 
like any Markov chain (under certain technical
assumptions), can be 
reversed in time, and its reversal is again a Markov chain
---
but usually
time-inhomogeneous and quite complicated.

For TASEP,
let $\{\mathbf{T}_t\}_{t\in \mathbb{R}_{\ge0}}$
be its Markov semigroup.
Defining 
\begin{equation*}
	\mathbf{T}^{rev}_{t,s}(\mathbf{x},\mathbf{y})
	=
	\frac{\upmu_s(\mathbf{y})}{\upmu_t(\mathbf{x})}
	\,
	\mathbf{T}_{t-s}(\mathbf{y},\mathbf{x})
	,\qquad 
	t>s,
\end{equation*}
we see that $\mathbf{T}^{rev}$ also maps
the TASEP distributions back in time:
$\upmu_t \mathbf{T}^{rev}_{t,s}=\upmu_s$, $s<t$.
In other words, the probabilities $\mathbf{T}^{rev}$
come from the time-reversal of the TASEP conditional distributions.
The Markov process
corresponding to $\{\mathbf{T}^{rev}_{t,s}\}$
is time-inhomogeneous, and its interactions are substantially
nonlocal.
\Cref{thm:intro_action_on_TASEP} 
implies that
the BHP $\{\mathbf{L}_\tau\}$
is a \emph{different}, much more natural, Markov process
which maps the TASEP distributions back in time.

By a different
mapping of the distributions
we mean the following. 
One can check that the joint distribution
of the TASEP configuration at two times $e^{-\tau}t$ and $t$
differs from the joint distribution of $(\mathbf{x},\mathbf{y})$,
where $\mathbf{y}$ is distributed as $\upmu_t$, 
and $\mathbf{x}$ is obtained from $\mathbf{y}$ by running
the BHP process $\mathbf{L}_{\tau}$.

\subsection{Idea of proof of \Cref{thm:intro_action_on_TASEP}}
\label{sub:intro_idea_of_proof}

We prove 
\Cref{thm:intro_action_on_TASEP} in 
\Cref{sec:Markov_maps,sec:action_on_q_Gibbs,sec:limit_q_1}.
Here let us outline the main steps.

First, we modify the problem by introducing an 
extra parameter $q\in(0,1)$,
and consider the TASEP
in which the $k$-th particle from the right, 
$k\in \mathbb{Z}_{\ge1}$, 
has the jump rate $q^{k-1}$.\footnote{We
emphasize that this 
$q$-version of the TASEP
should not be confused with
the 
$q$-TASEP of \cite{SasamotoWadati1998}, \cite{BorodinCorwin2011Macdonald}.} 
Let the distribution at time $t$ of this TASEP
(with step initial configuration)
be denoted by $\upmu_t^{(q)}$.

\begin{figure}[ht]
	\centering
	\includegraphics{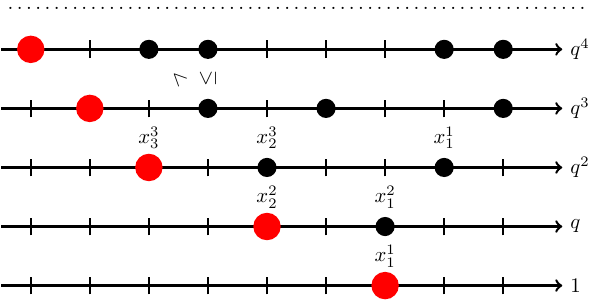}
	\caption{A configuration $\{x^j_i\}$ 
		in $\mathbb{Z}\times \mathbb{Z}_{\ge1}$.
		The leftmost (marked) particles are identified
		with TASEP. The interlacing condition 
		$x^{j+1}_{i+1}< x^{j}_{i}\le x^{j+1}_{i}$
		holds throughout the configuration.}
	\label{fig:intro_interlacing}
\end{figure}

Second, we use 
the well-known mapping of the TASEP to 
\emph{Schur processes}.
Schur processes 
\cite{okounkov2003correlation}
(and their various generalizations
including the Macdonald processes
\cite{BorodinCorwin2011Macdonald})
are one of the central 
tools in Integrable Probability.
The particular Schur 
processes we employ are probability distributions
on particle configurations
$\{x^j_i\}_{1\le i\le j}$
in $\mathbb{Z}\times \mathbb{Z}_{\ge1}$
which satisfy
an interlacing condition,
see \Cref{fig:intro_interlacing}.

There exists a Schur process
(depending on $q$ 
and the time parameter $t\in \mathbb{R}_{\ge0}$)
under which the joint distribution
of
the leftmost particles 
$\{x^N_N\}_{N\in \mathbb{Z}_{\ge1}}$
in each horizontal row
is the same as of the $q$-dependent
TASEP particles $x_1(t)>x_2(t)>\ldots$
(i.e., this is 
$\upmu_t^{(q)}$).
This mapping between TASEP and Schur processes
is described in \cite{BorFerr2008DF}, 
but also follows from earlier constructions
involving the Robinson-Schensted-Knuth correspondence.
We recall the details in \Cref{sec:Schur_and_TASEP}.

This Schur process corresponding to 
$\upmu_t^{(q)}$ depends on $q$ via the
\emph{spectral parameters} $1,q,q^2,\ldots $ attached to the 
horizontal lines
(as indicated in \Cref{fig:intro_interlacing}).
The new ingredients we
bring to Schur processes 
are \emph{Markov maps} interchanging two neighboring
spectral parameters (say, the $j$-th and the $(j+1)$-th). 
By a Markov map we mean 
a way to randomly modify the interlacing particle configuration
in $\mathbb{Z}\times \mathbb{Z}_{\ge1}$
such that:
\begin{itemize}
	\item At the $j$-th horizontal level the particles almost
		surely jump to the left;
	\item All other levels are untouched;
	\item The interlacing conditions
		are preserved;
	\item If the starting configuration
		was distributed as a Schur process,
		then the resulting configuration
		is distributed as a modified Schur process with 
		the $j$-th and the $(j+1)$-th
		spectral parameters interchanged.
\end{itemize}
We refer to this as the ``L Markov map''
since it moves particles to the left 
(it has a counterpart, the ``R Markov map'', but 
we do not need it for the main result).
The L Markov 
map at each $j$-th level
depends only on the ratio of the spectral parameters
being interchanged.

Combining the L Markov maps in such a way that they interchange
the bottommost spectral parameter~$1$ 
with $q$, then with $q^2$, then with $q^3$, and so on, 
we can move this parameter $1$ to infinity, where it 
``disappears'' (see \Cref{fig:braid_q} below
for an illustration). 
The 
resulting distribution of the configuration
will again
be a Schur process with the same 
spectral parameters $(1,q,q^2,\ldots )$,
but with the modified time parameter, $t\mapsto qt$.
Here we use the fact that the measure
does not change under the simultaneous 
rescaling of the spectral parameters.

Considering the action of this combination of the L Markov 
maps on the leftmost particles
$\{x_N^N\}$, we arrive at an explicit Markov transition
kernel on $\mathcal{C}$, denoted by $\mathbf{L}^{(q)}$,
with the property that 
(this is \Cref{thm:action_on_q_geom_TASEP} below)
\begin{equation*}
	\upmu_t^{(q)}\,\mathbf{L}^{(q)}=
	\upmu_{qt}^{(q)}
	\qquad \textnormal{for all $t\in \mathbb{R}_{\ge0}$}.
\end{equation*}
Finally,
iterating the action of $\mathbf{L}^{(q)}$
and taking the limit as $q\to1$, we arrive at 
\Cref{thm:intro_action_on_TASEP}.

\subsection{\texorpdfstring{``}{"}Toy\texorpdfstring{''}{"} example. Coupling of Bernoulli random walks}
\label{sub:intro_Bernoulli_coupling}

The Schur process computations 
leading to \Cref{thm:intro_action_on_TASEP}
have an
elementary consequence
which we now describe.
% This statement may be proven independently
% by an interested reader
% using only basic probability theory.
Its connection to Schur processes is detailed in 
\Cref{sub:back_TASEP_and_branching_graph}.

\begin{figure}[htpb]
	\centering
	\includegraphics{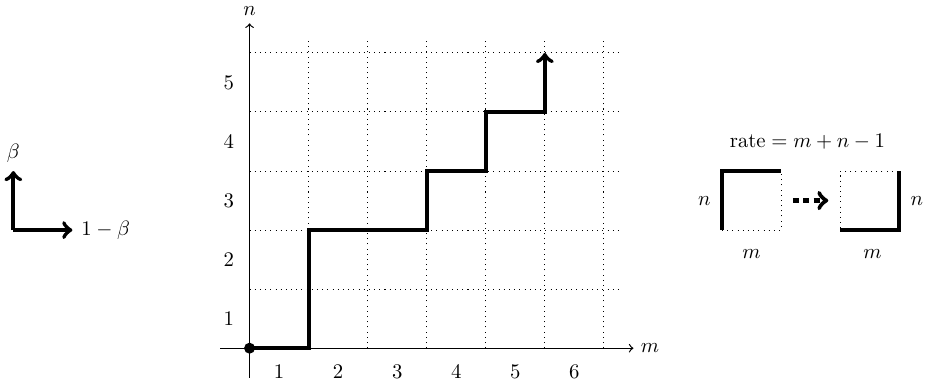}
	\caption{Left: Probabilities in the Bernoulli
	random walk. Center: A sample trajectory
	of the Bernoulli random walk. 
	Right: Local step of the process 
	$\mathbf{D}_\tau$.}
	\label{fig:intro_Bernoulli}
\end{figure}

Fix $\beta\in(0,1)$, and let 
$\mathsf{b}_\beta$ be the distribution of the simple random
walk in the quadrant $\mathbb{Z}_{\ge0}^2$, 
under which the walker 
starts at $(0,0)$ and
goes up with probability $\beta$ 
and to the right with probability $1-\beta$, independently
at each step. 

Consider the continuous-time Markov process on the space
of random walk trajectories under which each 
$(up, right)$ local piece 
is independently
replaced by the $(right,up)$ piece
at rate $m+n-1$,
where $(m,n)\in \mathbb{Z}_{\ge1}$ 
are the coordinates of the local piece.
See \Cref{fig:intro_Bernoulli} for an illustration.
Clearly, in each triangle $\left\{ m+n\le K \right\}$,
almost surely 
at each time moment there is at most one change of the trajectory.
Moreover, for different $K$ these processes are compatible, so 
by the Kolmogorov extension theorem
they indeed define a 
continuous-time Markov process on the full space
of random walk trajectories.
Denote the resulting Markov 
semigroup by 
$\{\mathbf{D}_\tau\}_{\tau\in\mathbb{R}_{\ge0}}$.

\begin{mainprop}
	\label{prop:Bernoulli_random_walk}
	For any $\beta\in(0,1)$ and $\tau\ge0$ we have
	\begin{equation*}
		\mathsf{b}_\beta\,\mathbf{D}_\tau=\mathsf{b}_{\beta(\tau)},
		\quad \text{where}\quad
		\beta(\tau)=\frac{\beta e^{-\tau}}{1-\beta+\beta e^{-\tau}}.
	\end{equation*}
\end{mainprop}

The action of $\mathbf{D}_{\tau}$ decreases the parameter
$\beta$ and almost surely moves the 
trajectory closer to the $m$ (horizontal) axis.
By symmetry, one can also define a continuous-time 
Markov chain which moves the vertical pieces
of the trajectory to the left, and 
increases the parameter $\beta$.
It could be interesting to look at the
stationary dynamics ---
a combination of the two processes
running in parallel which
does not change $\beta$ ---
and understand its large-scale
asymptotic behavior. 
We do not focus on this question in the present work.

\subsection{Stationary dynamics on the TASEP measure}
\label{sub:intro_equil}

Fix $t\in \mathbb{R}_{>0}$.
The backwards Hammersley-type process 
slowed down by a factor of $t$ 
compensates the time change 
of the
forward TASEP 
evolution.
Running these two processes in parallel
thus amounts to 
a
continuous-time Markov process which 
\emph{preserves} the TASEP distribution~$\upmu_t$.

One can say that the TASEP distributions $\upmu_t$
are the ``blocking measures''
for the stationary dynamics
\cite{Liggett1985}
(see also
\cite{balazs2018product}).

The presence of the stationary dynamics
on $\upmu_t$ allows to obtain new 
properties of the TASEP measure. 
In particular, we
write down an exact
evolution equation 
for $\mathbb{E}\,G(N_t^0)$,
where 
$N_t^0$ is the number of 
particles to the right of zero at time $t$, and 
$G$ is an arbitrary function.
This equation contains
one more random quantity --- the 
number of holes immediately
to the left of zero.
See 
\Cref{cor:microscopic_equation}
for details.

Moreover, in \Cref{sec:equil_dyn}
we rederive the limit shape parabola for the TASEP 
by looking at the hydrodynamics
of the process preserving~$\upmu_t$.
Indeed, recall that the TASEP local equilibria --- the ergodic
translation invariant measures
on configurations on the full line $\mathbb{Z}$
which are also invariant
under the TASEP evolution --- 
are precisely the product Bernoulli measures
\cite{Liggett1985}. In the bulk of the BHP,
the difference between jump rates of consecutive particles 
is inessential. Thus, the product
Bernoulli measures also serve as local 
equilibria for the BHP.\footnote{In fact, they are the only 
(extreme)
local equilibria because the particle-hole 
involution turns the homogeneous BHP into the 
PushTASEP (=~long-range TASEP), and local equilibria for the 
latter are classified 
\cite{guiol1997resultat},
\cite{andjel2005long}.}
By looking at the local equilibria,
one can write down two hydrodynamic 
PDEs for the 
TASEP limit shape: first is the well-known
Burgers' equation, and the second is a PDE
coming from the BHP, which is specific to the 
step initial condition.
After simplifications, these PDEs
lead to the parabola \eqref{eq:TASEP_limit_shape}.

Beyond hydrodynamics, 
the asymptotic \emph{fluctuation}
behavior 
of the TASEP measures 
$\upmu_t$ as $t\to+\infty$
is understood very well by now, starting from
\cite{johansson2000shape}. 
It would be very interesting to extend these results
to the combination $\textnormal{TASEP}+t^{-1}\textnormal{BHP}$ which preserves
$\upmu_t$.

\subsection{Further extensions}
\label{sub:intro_extensions_questions}

The Markov maps 
on Schur processes
we introduce to prove 
our main result,
\Cref{thm:intro_action_on_TASEP},
offer a variety of other applications and open problems.
We discuss them in more detail \Cref{sec:extensions}.
Here let us briefly outline the main directions:
\begin{itemize}
	\item 
		The one-dimensional statement
		(mapping the TASEP distributions back in time)
		has an extension to two dimensions.
		Namely, there is a continuous-time Markov
		process on interlacing particle configurations
		(as in \Cref{fig:intro_interlacing})
		which maps back in time the 
		distributions of the 
		anisotropic KPZ growth process
		on interlacing arrays 
		studied in
		\cite{BorFerr2008DF}.
	\item 
		Instead of Schur processes, one can consider 
		interlacing configurations of finite depth.
		This includes
		probability
		distributions on boxed plane partitions
		with weight proportional to $q^{\mathrm{vol}}$
		(where $\mathrm{vol}$ is the volume under the boxed
		plane partition).
		In this setting our constructions produce Markov chains
		mapping the measure $q^{\mathrm{vol}}$ to 
		the measure
		$q^{-\mathrm{vol}}$, and vice versa.
		(A
		simulation is available online
		\cite{PetrovZhang2019simul}.)
		Applying this procedure twice 
		leads to a new sampling algorithm for the 
		measures $q^{\pm\mathrm{vol}}$.
	\item
		A certain bulk limit of our two-dimensional Markov
		maps essentially leads to the 
		growth processes preserving ergodic
		Gibbs measures on 
		two-dimensional interlacing configurations
		introduced and 
		studied in \cite{Toninelli2015-Gibbs}.
		Thus, one can view our Markov maps as 
		exact
		``pre-bulk'' 
		stationary dynamics on two-dimensional
		interlacing configurations.
	\item 
		\Cref{thm:intro_action_on_TASEP} 
		may be interpreted as the 
		statement that the family of measures $\{\upmu_t\}$
		is coherent with respect to a \emph{projective
		system} determined by the process $\{\mathbf{L}_\tau\}$.
		Projective systems \cite{BorodinOlsh2011Bouquet}
		generalize the notion of branching graphs, and the 
		latter play a fundamental role in 
		Asymptotic Representation Theory
		\cite{VK81AsymptoticTheory},
		\cite{borodin2016representations}.
		(Even further, the distributions of the anisotropic
		KPZ growth are also coherent, on a 
		projective system whose ``levels'' are spaces of
		two-dimensional interlacing configurations.)
		The framework of projective systems / branching graphs
		provides many natural questions 
		in this setting.
	\item 
		Structurally, our Markov maps
		are inspired
		by the 
		study of stochastic vertex 
		models and bijectivization of the Yang-Baxter equation
		\cite{BufetovPetrovYB2017}, 
		\cite{BufetovMucciconiPetrov2018}.
		Compared with the Schur case, the 
		full Yang-Baxter equation 
		for the quantum $\mathfrak{sl}_2$
		contains more parameters.
		In this setting, Schur polynomials
		should be replaced 
		by the spin Hall-Littlewood or spin 
		$q$-Whittaker symmetric functions 
		\cite{Borodin2014vertex},
		\cite{BorodinWheelerSpinq}.
		It is interesting to see how far
		\Cref{thm:intro_action_on_TASEP} can be generalized
		to other particle systems
		arising in this framework, 
		including ASEP, various stochastic
		six vertex models, and 
		random matrix models.
	\item There exists a backwards dynamics
		for the ASEP started from a
		family of shock measures
		\cite{belitsky2016self}.
		This ASEP backwards dynamics is obtained via a
		duality.
		While the shock measures
		are very different
		from the step initial configuration,
		it would be interesting to 
		find connections of 
		\Cref{thm:intro_action_on_TASEP}
		to Markov duality.
\end{itemize}

Concrete open questions 
along these directions 
are formulated and discussed
in 
\Cref{sec:extensions}.

\subsection*{Outline}

In \Cref{sec:ascending_Schur,sec:Schur_and_TASEP}
we recall the necessary facts about Schur processes,
TASEP, and their connection.
In \Cref{sec:Markov_maps} we introduce the L and R Markov
maps at the level of interlacing arrays. The action of 
each such map
swaps two neighboring spectral
parameters.
In \Cref{sec:action_on_q_Gibbs}
we combine the L Markov maps in such a way that their combination
$\mathbb{L}^{(q)}$ preserves the class of $q$-Gibbs measures
on interlacing arrays
(which includes the Schur
processes related to the $q$-dependent TASEP).
We compute the action of $\mathbb{L}^{(q)}$
on $q$-Gibbs measures and the corresponding Schur processes.
In \Cref{sec:limit_q_1} we take a limit $q\to1$, which
leads to our main result, \Cref{thm:intro_action_on_TASEP}.
In \Cref{sec:equil_dyn} 
we illustrate the relation between the TASEP and the backwards evolutions
at the hydrodynamic level by looking at the stationary 
dynamics on the TASEP distribution $\upmu_t$.
Finally, in \Cref{sec:extensions} we discuss 
possible extensions of our
constructions 
indicated in 
\Cref{sub:intro_extensions_questions} above,
and formulate a number of open questions.

\subsection*{Data Availability}

Data sharing not applicable to this article as no datasets were generated or analysed during the current study.

\subsection*{Acknowledgments}

We are grateful to 
Alexei Borodin, 
Evgeni Dimitrov,
Patrik Ferrari,
Vadim Gorin, 
Matthew Nicoletti,
Grigori Olshanski,
Dan Romik, 
Tomohiro Sasamoto, 
Mykhaylo Shkolnikov,
and 
Fabio Toninelli
for helpful remarks.
LP is grateful to the organizers of the workshop ``Asymptotic
Algebraic Combinatorics'' and the support of the Banff International
Research Station where a part of this work was done.
Both authors were partially supported by the National Science Foundation grant
DMS-1664617.

%----------------------------------------------------------------------------

\section{Ascending Schur processes}
\label{sec:ascending_Schur}

This section is a brief review of ascending Schur processes introduced
in \cite{okounkov2003correlation} and their relation to TASEP.
More details may be found in, e.g., \cite{BorodinGorinSPB12}.

\subsection{Partitions} 
	A partition 
	$\lambda = (\lambda_1\ge \ldots\ge \lambda_\ell(\lambda)>0 )$,
	where $\lambda_i\in \mathbb{Z}$,
	is a weakly decreasing sequence of nonnegative integers. 
	We denote $|\lambda|:=\sum_{i=1}^{N}\lambda_i$.
	We call $\ell(\lambda)$ the length of a partition. 
	By convention we do not distinguish partitions
	if they differ by trailing zeroes.
	In this way $\ell(\lambda)$ always denotes the number
	of strictly positive parts in $\lambda$.
	Denote by $\mathbb{Y}$ the set of all partitions
	including the empty one $\varnothing$ (by convention,
	$\ell(\varnothing)=|\varnothing|=0$).
	
% we will add more things here if needed

\subsection{Schur polynomials}

Fix $N\in \mathbb{Z}_{\ge1}$.
The Schur symmetric polynomials 
in $N$ variables are
indexed $\lambda\in \mathbb{Y}$
and are defined as
\begin{equation*}
	s_\lambda(x_1,\ldots,x_N ):=
	\frac{\det \bigl[x_i^{\lambda_j + N -j} \bigr]_{i,j=1}^N}{\prod_{1\le i<j\le N} (x_i -x_j)},\qquad 
	N\ge \ell(\lambda).
\end{equation*}
If $N<\ell(\lambda)$, we set $s_\lambda(x_1,\ldots,x_N )=0$,
by definition.

The Schur polynomials $s_\lambda$ indexed by all $\lambda\in \mathbb{Y}$
with $\ell(\lambda)\le N$
form a linear basis in
the space $\mathbb{C}[x_1,\ldots,x_N ]^{\mathfrak{S}_N}$
of symmetric polynomials in $N$ variables.
Each $s_\lambda$ is a homogeneous polynomial
of degree $|\lambda|$.

The Schur polynomials are stable in the following sense:
\begin{equation}
	\label{eq:schur_stability}
	s_{\lambda}
	(x_1,\ldots,x_N,0)=
	s_{\lambda}
	(x_1,\ldots,x_N).
\end{equation}
This stability allows to define Schur symmetric \emph{functions}
$s_\lambda$, $\lambda\in \mathbb{Y}$,
in infinitely many variables.
These objects form a linear basis of the algebra of
symmetric functions $\Lambda$. We refer to 
\cite[Ch. I.2]{Macdonald1995} for the precise definition and details 
on the algebra $\Lambda$. 

\subsection{Skew Schur polynomials}
The skew Schur polynomials
$s_{\lambda/\varkappa}$, $\lambda,\varkappa\in \mathbb{Y}$
are defined through the branching rule as follows:
\begin{equation}
	\label{eq:branching}
	s_\lambda(x_1,\ldots,x_N )
	=
	\sum_{\varkappa\in \mathbb{Y}}
	s_\varkappa(x_1,\ldots,x_K )
	s_{\lambda/\varkappa}(x_{K+1},\ldots,x_N ).
\end{equation}
Indeed, 
$s_\lambda(x_1,\ldots,x_N )$ is a symmetric polynomial in 
$x_1,\ldots,x_K$, and so 
the skew Schur polynomials in
\eqref{eq:branching} are the coefficients of the linear expansion.
These skew Schur polynomials are symmetric in $x_{K+1},\ldots,x_N$
and satisfy the stability property similar to \eqref{eq:schur_stability}.
We have $s_{\lambda/\varnothing}=s_\lambda$.

Let $\lambda,\varkappa\in\mathbb{Y}$.
Plugging in just one variable into $s_{\lambda/\varkappa}$
simplifies this symmetric function. Namely,
$s_{\lambda/\varkappa}(x)$ vanishes unless
$\varkappa$ and $\lambda$ \emph{interlace}
(notation $\varkappa\prec \lambda$; equivalently, $\lambda/\varkappa$ is a horizontal strip):
\begin{equation}
	\label{eq:interlacing}
		\lambda_1\ge \varkappa_1\ge \lambda_2\ge \varkappa_2
		\ge \ldots .
\end{equation}
Moreover, 
\begin{equation}
	\label{eq:skew_one_var}
	s_{\lambda/\varkappa}(x)=x^{|\lambda|-|\varkappa|}\mathds{1}_{\varkappa\prec\lambda}.
\end{equation}
For any $\lambda\in \mathbb{Y}$, the set $\left\{ \varkappa\colon \varkappa\prec \lambda \right\}$
is finite.

Iterating \eqref{eq:branching} and breaking down all skew
Schur polynomials into single-variable ones,
we see that each Schur polynomial has the following form:
\begin{equation}\label{eq:schur_array_expansion}
	s_\lambda(x_1,\ldots,x_N )=
	\sum_{\lambda^{(1)}\prec \ldots\prec \lambda^{(N)}=\lambda }
	x_1^{|\lambda^{(1)}|}
	x_2^{|\lambda^{(2)}|-|\lambda^{(1)}|}
	\ldots
	x_{N-1}^{|\lambda^{(N-1)}|-|\lambda^{(N-2)}|}
	x_{N}^{|\lambda^{(N)}|-|\lambda^{(N-1)}|},
\end{equation}
where the sum is taken over all interlacing arrays of partitions 
of depth $N$ in which the top row coincides with $\lambda$
(see \Cref{fig:array} for an illustration).
In combinatorial language, \eqref{eq:schur_array_expansion} is the 
representation of a Schur polynomial as a generating function
of semistandard Young tableaux, cf. \cite{fulton1997young}.

\begin{remark}\label{rmk:interlacing_number_of_rows}
	If $N<\ell(\lambda)$, then there are no interlacing arrays
	of depth $N$ whose top row is $\lambda$ because
	at each level one can add at most one nonzero component. 
	Thus, the right-hand side of \eqref{eq:schur_array_expansion}
	automatically vanishes if $N<\ell(\lambda)$.
	This agrees with the fact that
	$s_\lambda(x_1,\ldots,x_N )=0$ if $N<\ell(\lambda)$.
\end{remark}

The following two identities for skew 
Schur polynomials play a fundamental role in our work.
The first identity is a straightforward consequence of the symmetry 
of the Schur polynomials.

\begin{proposition}	
	\label{prop:Schur_symmetry}
	For any $\lambda,\mu\in \mathbb{Y}$ and variables $x,y$ we have
	\begin{equation*}
		\sum_{\varkappa\in \mathbb{Y}}
		s_{\lambda/\varkappa}(x)s_{\varkappa/\mu}(y)
		=
		\sum_{\hat \varkappa\in \mathbb{Y}}
		s_{\lambda/\hat\varkappa}(y)
		s_{\hat\varkappa/\mu}(x).
	\end{equation*}
	The sums in both sides are finite.
\end{proposition}

The second is the 
skew Cauchy identity, see \cite[Ch. I.5]{Macdonald1995}.

\begin{proposition}\label{prop:skew_Cauchy}
	For any $\lambda, \mu \in \mathbb{Y}$
	and variables $x_1,\ldots,x_N,y_1,\ldots,y_M $ we have
	\begin{equation}\label{eq:cauchy}
		\begin{split}
			&
			\sum_{\nu\in \mathbb{Y}} 
			s_{\nu/\mu}(x_1,\ldots,x_N )
			s_{\nu/\lambda}(y_1,\ldots,y_M )
			\\&\hspace{70pt}=
			\prod_{i=1}^{N}\prod_{j=1}^{M}
			\frac{1}{1 - x_i y_j} 
			\sum_{\varkappa \in \mathbb{Y}} 
			s_{\lambda / \varkappa }(x_1,\ldots,x_N)
			s_{\mu/\varkappa}(y_1,\ldots,y_M ).
		\end{split}
	\end{equation}
	This is an identity of generating series in $x_i,y_j$
	under the standard geometric series expansion
	$\frac{1}{1-x_iy_j}=1+x_iy_j+(x_iy_j)^2+\ldots $.
	Moreover, \eqref{eq:cauchy} holds as a numerical identity if 
	$x_i,y_j\in \mathbb{C}$ are such that $|x_iy_j|<1$ for all $i,j$.
\end{proposition}
\begin{remark}
	If we set $\lambda=\mu=\varnothing$ in \eqref{eq:cauchy}, the 
	sum in the right-hand side disappears 
	(because $s_{\varnothing/\varkappa}=\mathds{1}_{\varkappa=\varnothing}$),
	and we obtain 
	\begin{equation}
		\label{eq:usual_cauchy}
		\sum_{\nu\in \mathbb{Y}} 
			s_{\nu}(x_1,\ldots,x_N )
			s_{\nu}(y_1,\ldots,y_M )
			=
			\prod_{i=1}^{N}\prod_{j=1}^{M}
			\frac{1}{1 - x_i y_j}.
	\end{equation}
	Again, this is a numerical identity provided that 
	$|x_iy_j|<1$ for all $i,j$.
\end{remark}

\subsection{Specializations}
\label{sub:specializations}

When $x\ge0$, 
we have $s_{\lambda/\varkappa}(x)\ge0$
from \eqref{eq:skew_one_var}.
More generally,
the Schur polynomials
$s_\lambda(x_1,\ldots,x_N )$
are nonnegative 
for real nonnegative $x_1,\ldots,x_N$.

We will also need the
\emph{Plancherel} specializations of Schur functions $s_\lambda$.
These specializations,
indexed by $t\ge0$,
may be defined through the limit
\begin{equation}
	\label{eq:Plancherel_limit}
	s_{\lambda}(\rho_t)
	:=
	\lim_{K\to+\infty}
	s_\lambda\left( \frac{t}{K},\ldots,\frac{t}{K} \right),
	\qquad 
	\lambda\in \mathbb{Y},
\end{equation}
where $\frac{t}{K}$ is repeated $K$ times.

\begin{remark}
	We also have 
	$s_\lambda(\rho_t)=t^{|\lambda|}\,\dfrac{\dim \lambda}{|\lambda|!}$,
	where $\dim \lambda$ is the dimension of the irreducible
	representation of the symmetric
	group $\mathfrak{S}_{|\lambda|}$,
	or, equivalently, the number of standard Young tableaux of 
	shape $\lambda$.
\end{remark}

Generic nonnegative specializations will be denoted
as $\rho\colon \Lambda\to \mathbb{R}$, and we will also use the notation $s_\lambda(\rho)$
for $\rho(s_\lambda)$. For the purposes of the present paper, 
$\rho$ would be either a Plancherel specialization, or 
a substitution of a finitely many nonnegative variables
into the symmetric function.

\begin{remark}
	A classification of 
	Schur-positive
	specializations
	(that is, algebra homomorphisms $\Lambda\to\mathbb{R}$ which are 
	nonnegative on Schur functions)
	is known and is
	equivalent to the celebrated Edrei--Thoma theorem.
	See, for example, \cite{borodin2016representations} 
	for a modern account 
	discussing various equivalent formulations.
\end{remark}

\subsection{Schur processes}

Schur measures and processes 
are probability distributions on 
partitions 
or sequences of partitions
whose probability weights are expressed through Schur polynomials
in a certain way. They
were introduced in
\cite{okounkov2001infinite}, 
\cite{okounkov2003correlation}. 

A \emph{Schur measure} is a probability
measure on $\mathbb{Y}$
with probability weights depending on two nonnegative specializations
$\rho_1,\rho_2$:
\begin{equation}
	\label{eq:Schur_measure}
	\mathbb{P}[\rho_1\mid  \rho_2](\lambda) =
	\frac{1}{Z}\,
	s_{\lambda}(\rho_1) s_{\lambda}(\rho_2),
	\qquad 
	Z=\sum_{\lambda\in\mathbb{Y}}s_\lambda(\rho_1)s_\lambda(\rho_2).
\end{equation}
The normalizing constant $Z$ can be computed 
using the Cauchy identity \eqref{eq:usual_cauchy}
(provided that the infinite sum converges).

Schur processes are probability measures on 
sequences of partitions generalizing the 
Schur measures. 
We will only need the particular case of \emph{ascending Schur processes}. 
These are probability measures on interlacing arrays
\begin{equation*}
	\lambda^{(1)}\prec \lambda^{(2)}\prec \ldots\prec \lambda^{(N)}, 
	\qquad \lambda^{(j)}\in\mathbb{Y}
\end{equation*}
(for some fixed $N$) depending on a nonnegative specialization $\rho$
and $c_1,\ldots,c_N\ge0$:
\begin{equation}
	\label{eq:Schur_process}
	\mathbb{P}[\vec c\mid \rho](\lambda^{(1)},\ldots,\lambda^{(N)} ):=
	\frac{1}{Z}\,
	s_{\lambda^{(1)}}(c_1)s_{\lambda^{(2)}/\lambda^{(1)}}(c_2)\ldots
	s_{\lambda^{(N)}/\lambda^{(N-1)}}(c_N)s_{\lambda^{(N)}}(\rho).
\end{equation}
The normalizing constant has the form
(this follows from \eqref{eq:branching} and \eqref{eq:usual_cauchy}):
\begin{equation}\label{eq:Schur_process_normalization}
	Z=\sum_{\lambda\in \mathbb{Y}}
	s_\lambda(c_1,\ldots,c_N )s_\lambda(\rho)
\end{equation}
(provided that the series converges).
We call $N$ the depth of a Schur process.
We will sometimes call the $c_i$'s the \emph{spectral parameters} 
of Schur process $\mathbb{P}[\vec c\mid \rho]$.

The next statement immediately follows from \eqref{eq:branching}
and the skew Cauchy identity:
\begin{proposition}
	\label{prop:marginal_Schur}
	Under the Schur process \eqref{eq:Schur_process},
	the marginal distribution of each 
	$\lambda^{(K)}$, $1\le K\le N$,
	is given by the Schur measure $\mathbb{P}[(c_1,\ldots,c_K )\mid\rho]$.
\end{proposition}

\subsection{Schur processes of infinite depth}
\label{sub:infinite_depth}

Let us denote by $\mathcal{S}$ the set of interlacing arrays of infinite depth
$\{\lambda^{(j)}\}_{j\in \mathbb{Z}_{\ge0}}$, where
$\lambda^{(j)}\in \mathbb{Y}$ and
$\lambda^{(j-1)}\prec \lambda^{(j)}$ (cf. \Cref{fig:array} for an illustration).

\begin{remark}
	The interlacing array 
	in \Cref{fig:array}
	and the one in \Cref{fig:intro_interlacing}
	in the Introduction are related by 
	$x^N_k=\lambda^{(N)}_k-N+k$.
	We work with the $\{\lambda^{(N)}_k\}$
	notation throughout the paper.
\end{remark}

By the Kolmogorov extension theorem, a measure on $\mathcal{S}$
is uniquely determined by a collection of compatible joint distributions
of $\{\lambda^{(1)}\prec\ldots\prec\lambda^{(N)}\}_{N\ge1}$.
If these joint distributions satisfy the 
$\vec c$-Gibbs property,
then the resulting measure on $\mathcal{S}$ is $\vec c$-Gibbs.

Thus, \Cref{prop:marginal_Schur} implies the following extension of the definition of a 
Schur process.
Given an infinite sequence $c_1,c_2,\ldots, $ of nonnegative reals
such that the sums like \eqref{eq:Schur_process_normalization} converge for all 
$N$, one can define the Schur process $\mathbb{P}[\vec c\mid \rho]$ of infinite depth, i.e.,
a probability measure on $\mathcal{S}$.
Indeed, this is because the distributions \eqref{eq:Schur_process}
for different $N$ are compatible with each other
by \Cref{prop:marginal_Schur}, so the measure on 
$\mathcal{S}$ with the desired finite-dimensional distributions
exists.

\begin{figure}[htpb]
	\centering
	\begin{tikzpicture}[scale=1.2]
		\def\h{0.8}
		\def\x{2}
		\node at (-5,5) {$\lambda_N^{(N)}$};
		\node at (-3,5) {$\lambda_{N-1}^{(N)}$};
		\node at (0-\x/2,5) {$\ldots\ldots\ldots\ldots$};
		\node at (0-\x/2,5.7) {$\ldots\ldots\ldots\ldots\ldots\ldots\ldots\ldots\ldots\ldots\ldots\ldots\ldots\ldots$};
		\node at (3-\x,5) {$\lambda_{2}^{(N)}$};
		\node at (5-\x,5) {$\lambda_{1}^{(N)}$};
		\node at (-4,5-\h) {$\lambda^{(N-1)}_{N-1}$};
		\node at (-2,5-\h) {$\lambda^{(N-1)}_{N-2}$};
		\node at (0-\x/2+.1,5-\h) {$\ldots$};
		\node at (2-\x,5-\h) {$\lambda^{(N-1)}_{2}$};
		\node at (4-\x,5-\h) {$\lambda^{(N-1)}_{1}$};
		\node at (-3,5-2*\h) {$\lambda^{(N-2)}_{N-2}$};
		\node at (3-\x,5-2*\h) {$\lambda^{(N-2)}_{1}$};
		\foreach \LePoint in {(4.5-\x,5-\h/2),(2.5-\x,5-\h/2),(-3.5,5-\h/2),
		(-2.5,5-3*\h/2),(3.5-\x,5-3*\h/2),(2.5-\x,5-5*\h/2),(1.7-\x,5-7*\h/2)} {
    		\node [rotate=45] at \LePoint {$\le$};
    	};
    	\foreach \GePoint in {(3.5-\x,5-\h/2),(-4.5,5-\h/2),(-2.5,5-\h/2),
    	(-3.5,5-3*\h/2),(-2.5,5-5*\h/2),(.5,5-3*\h/2),(-1.7,5-7*\h/2)} {
    		\node [rotate=135] at \GePoint {$\ge$};
    	};
			\node at (0-\x/2,5-3*\h+.1) {$\ldots\ldots\ldots\ldots$};
    	\node at (0-\x/2+.08,5-4*\h) {$\lambda^{(1)}_1$};
	\end{tikzpicture}
	\caption{An interlacing array.}
	\label{fig:array}
\end{figure}
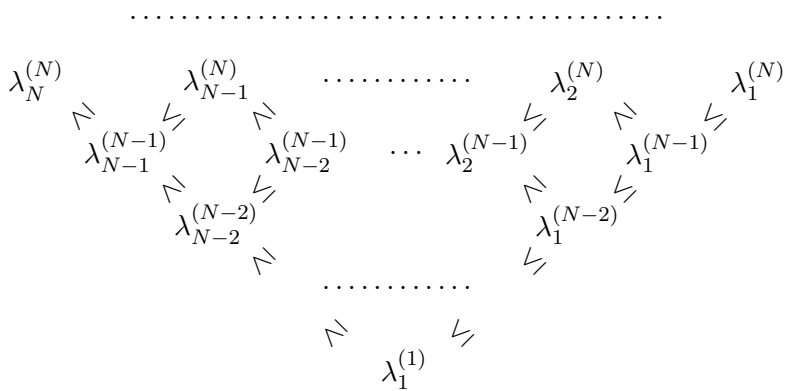

\subsection{\texorpdfstring{$\vec c$}{c}-Gibbs measures}
\label{sub:c_Gibbs_def}

Fix nonnegative reals $c_1,c_2,\ldots $.
A probability distribution on 
$\mathcal{S}$ is called \emph{$\vec c$-Gibbs} if
for any $N$, given $\lambda^{(N)}=\lambda$, the conditional distribution of the 
bottom part $\lambda^{(1)}\prec \ldots\prec\lambda^{(N-1)}\prec\lambda $
of the interlacing array has the form
\begin{equation}
	\label{eq:c_Gibbs}
	\mathrm{Prob}(\lambda^{(1)},\ldots,\lambda^{(N-1)}\mid \lambda^{(N)}=\lambda )
	=
	\frac{s_{\lambda^{(1)}}(c_1)s_{\lambda^{(2)}/\lambda^{(1)}}(c_2)\ldots
	s_{\lambda^{(N-1)}/\lambda^{(N-2)}}(c_{N-1})s_{\lambda/\lambda^{(N-1)}}(c_N)}
	{s_{\lambda}(c_1,\ldots,c_N )}.
\end{equation}
The expression in the denominator is simply the normalizing constant.
One can say that each interlacing array in \eqref{eq:c_Gibbs} is weighted 
proportional to the corresponding term in the expansion \eqref{eq:schur_array_expansion}.
Note that the $\vec c$-Gibbs property depends on the order of the $c_i$'s,
but the normalizing constant in \eqref{eq:c_Gibbs} does not.

The next lemma is straightforward consequence of \eqref{eq:c_Gibbs}.
\begin{lemma}
	\label{lemma:c_Gibbs_one_level}
	Fix any $j\ge2$.
	Under a $\vec c$-Gibbs measure, the conditional probability 
	of $\lambda^{(j)}$ given all $\lambda^{(i)}$, with $i\ne j$, 
	is proportional to 
	$s_{\lambda^{(j+1)}/\lambda^{(j)}}(c_{j+1})\,
	s_{\lambda^{(j)}/\lambda^{(j-1)}}(c_j)$.
\end{lemma}

Denote the space of all $\vec c$-Gibbs measures on $\mathcal{S}$ by $\mathfrak{G}_{\vec c}$.
Note that this space does not change
if we multiply all the parameters 
by the same positive number:
$\mathfrak{G}_{\vec c}=\mathfrak{G}_{a\cdot\vec c}$, $a>0$.
Indeed, this follows from \eqref{eq:c_Gibbs} and 
the homogeneity of the 
Schur polynomials. 

\begin{remark}
	\label{rmk:uniform_Gibbs}
	When all $c_i\equiv 1$, the conditional distribution \eqref{eq:c_Gibbs}
	becomes uniform (on the set of all interlacing arrays of depth $N$ 
	with top row $\lambda$).
	This uniform Gibbs case 
	justifies the name $\vec c$-Gibbs in the general situation.
\end{remark}

The Schur process $\mathbb{P}[\vec c\mid \rho]$ is a particular example of a
$\vec c$-Gibbs measure. The full classification of 
$\vec c$-Gibbs measures is known only in several particular cases.
In the uniform case $c_i\equiv 1$ this is the celebrated Edrei--Voiculescu theorem
(see \Cref{sub:general_schur} below and also, e.g., \cite{BorodinOlsh2011GT} for a modern account discussing various equivalent formulations).
When the $c_i$'s form a geometric sequence, the classification was obtained
much more recently in \cite{Gorin2010q}
(see also \cite{gorin2016quantization} for a generalization).

\section{Schur processes and TASEP}
\label{sec:Schur_and_TASEP}

In this section we recall a coupling between TASEP (with step initial configuration and 
particle-dependent speeds)
and a marginal of an ascending Schur process.
This mapping can be seen as a consequence of the column Robinson-Schensted-Knuth
insertion \cite{Vershik1986}, \cite{OConnell2003Trans}, \cite{OConnell2003}.
One can also define a continuous-time Markov dynamics
on interlacing arrays whose marginal is TASEP
\cite{BorFerr2008DF} (see also \cite{BorodinPetrov2013Lect}).

\subsection{TASEP}
\label{sub:TASEP_def}

Let $c_1,\ldots,c_N,\ldots$ be positive reals. 
The continuous-time TASEP (Totally Asymmetric Simple
Exclusion Process) with step initial condition and speeds $\vec c$ is defined as follows.
It is a Markov process on particle configurations
$\mathbf{x}(t)=(x_1(t)>x_2(t)>\ldots )$ on the integer lattice,
such that 
\begin{itemize}
	\item The initial particles' locations are $x_i(0)=-i$, $i=1,2,\ldots $
		(this is the step initial configuration);
	\item The configuration has the rightmost particle $x_1$;
	\item The configuration is densely packed far to the left,
		that is, for all large enough $M$ (where the bound on $M$ depends on $t$)
		we have
		$x_{M}(t)=-M$;
	\item There is at most one particle per site.
\end{itemize}
Denote the space of such left-packed and right-finite 
particle configurations on $\mathbb{Z}$
by $\mathcal{C}$.

The continuous-time Markov evolution of TASEP proceeds as follows.
Each particle $x_i$ has an independent exponential
clock with rate $c_i$. 
That is, the time before $x_i$ attempts to jump is an exponential random
variable: $\mathrm{Prob}(\textnormal{time}>t)=e^{-c_i t}$, $t\ge0$.
(We will refer to $c_i$'s as to the particle speeds.)
When the clock of $x_i$ rings, the particle jumps to the right by one
if the destination is not occupied. If the destination of the jumping particle
is occupied, the jump is forbidden and the particle configuration does not 
change. 
Because the process starts from the step initial configuration,
only finitely many particles are free to jump at any particular time. 
Therefore
at any time almost surely at most one jump happens.
See \Cref{fig:TASEP} for an illustration.

\begin{figure}[htpb]
	\centering
	\includegraphics{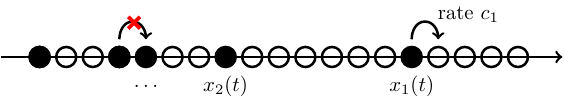}
	\caption{An example of a jump and a forbidden jump in TASEP.}
	\label{fig:TASEP}
\end{figure}

% Denote by $\upmu_{t}=\upmu_{\vec c\,;t}$ the distribution of
% the TASEP at time $t$.

\subsection{Coupling to a Schur process}

Fix $N\in \mathbb{Z}_{\ge1}$, positive reals $c_1,\ldots,c_N $, and $t\ge0$.
Consider the Schur process
$\mathbb{P}[\vec c\mid\rho_t]$
defined by \eqref{eq:Schur_process}, where $\rho_t$ is the 
Plancherel specialization. Note that the series for the normalizing constant
\eqref{eq:Schur_process_normalization} always converges because
\begin{equation*}
	\begin{split}
		Z=\sum_{\lambda\in \mathbb{Y}}
	s_\lambda(c_1,\ldots,c_N)
	s_\lambda(\rho_t)
	&=
	\lim_{K\to\infty}
	\sum_{\lambda\in \mathbb{Y}}
	s_\lambda(c_1,\ldots,c_N)
	s_\lambda\left( \frac{t}{K},\ldots,\frac{t}{K}  \right)
	\\&=
	\lim_{K\to\infty}\prod_{i=1}^{N}\frac{1}{(1-c_i t/K)^K}
	=
	e^{(c_1+\ldots+c_N)t},
	\end{split}
\end{equation*}
and the last expression is an entire function in $t$ and $c_i$.
Since this procedure works for all $N$, we can 
view $\mathbb{P}[\vec c\mid \rho_t]$ as a Schur process of infinite depth,
i.e., a probability measure on $\mathcal{S}$.

When $t=0$, $\mathbb{P}[\vec c\mid \rho_0]$
concentrated on 
the single 
interlacing array 
densely packed
at zero, that is,
with each $\lambda^{(j)}=(0,\ldots,0 )$
($j$ times).

The next result is present in
\cite{BorFerr2008DF}, but alternatively follows from much
earlier constructions involving Robinson-Schensted-Knuth 
correspondences \cite{Vershik1986}, \cite{OConnell2003Trans}, \cite{OConnell2003}.

\begin{theorem}
	\label{thm:BorFerr}
	Fix $t\ge0$ and particle speeds $c_1,c_2,\ldots $, and consider the TASEP 
	as in \Cref{sub:TASEP_def} at time $t$. Then we have equality of joint
	distributions at the fixed time $t$:
	\begin{equation}\label{eq:TASEP_Schur_equality}
		x_i(t)\stackrel{d}{=}\lambda_i^{(i)}-i,\qquad 
		i=1,2,\ldots,
	\end{equation}
	where $\lambda^{(i)}$ are the random partitions coming from the Schur process
	$\mathbb{P}[\vec c\mid \rho_t]$ described above.
\end{theorem}

\begin{remark}
	\label{rmk:anisotropic_KPZ}
	A dynamical version of this result is also proven in 
	\cite{BorFerr2008DF}:
	there exists a continuous-time Markov chain
	on interlacing arrays 
	(even a whole family of them, cf. \cite{BorodinPetrov2013NN}, \cite{BorodinPetrov2013Lect})
	whose action 
	on a Schur process
	$\mathbb{P}[\vec c\mid \rho_t]$
	continuously increases the parameter $t$. 
	We will refer to the dynamics from \cite{BorFerr2008DF}
	as the \emph{push-block process} (see \Cref{def:kpz_growth} below for details).
	For the push-block process on interlacing arrays,
	\eqref{eq:TASEP_Schur_equality}
	holds as equality of joint distributions of Markov processes.
	In other words, \eqref{eq:TASEP_Schur_equality} is also true for multitime
	joint distributions of these processes.
	However, we do not need this dynamical result for most of our constructions.
\end{remark}

%----------------------------------------------------------------------------

\section{Markov maps}
\label{sec:Markov_maps}

This section introduces our main objects --- the Markov maps
$L_\alpha^{(j)}$ and $R_{\alpha}^{(j)}$ which randomly change
the $j$-th row
$\lambda^{(j)}$ in
an interlacing
array while keeping all other rows intact.
These maps
act on $\vec c$-Gibbs measures by permuting spectral parameters.

% \begin{figure}[htbp]
%   \centering
%   \includegraphics[height=140pt]{MarkovMap}
%   \caption{The Markov map $L_{\alpha}^{(t)} $ acts locally on the
%   $t^{th}$ level of sequence of interlacing partitions.}
%   \label{fig:markov_map}
% \end{figure}

\subsection{First level}
\label{sub:toy_case}

Let us first describe the maps for $j=1$ (the simplest nontrivial case)
to illustrate their structure and properties.
We use the shorthand notation $\lambda^{(2)}=(\lambda_1,\lambda_2)$ and $\lambda^{(1)}=\varkappa_1$.
The interlacing means
that $\lambda_2\le \varkappa_1\le \lambda_1$.

\begin{definition}[Truncated geometric distribution]
	\label{def:truncated_geometric}
	Let $A\in \mathbb{Z}_{\ge0}$ and 
	$\alpha\in[0,1]$.
	A discrete random variable $Y=Y_\alpha(A)$ on $\left\{ 0,1,\ldots,A  \right\}$ is called
	\emph{truncated geometric} if it has the distribution
	\begin{equation*}
		\mathrm{Prob}(Y=k)=
		\begin{cases}
			(1-\alpha)\,\alpha^k,& 0\le k\le A-1;\\
			\alpha^{A},& k=A.
		\end{cases}
	\end{equation*}
\end{definition}

\begin{definition}[The L and R maps, first level]
	\label{def:LR_maps_one_level}
	For $\alpha\in[0,1]$,
	let $L_\alpha^{(1)}$ be the Markov map\footnote{A Markov
	map is the same as a stochastic matrix or a one-step transition operator of a Markov chain 
	(it is also sometimes called ``link'' in the literature).
	An application of a Markov map is a random 
	update of the underlying configuration.
	At the same time, each 
	Markov map is a deterministic linear operator in the space of 
	probability distributions on configurations.
	When applying a map $M$ to a probability 
	measure $\pi$, we write this as $\pi\mapsto \pi M$.}
	whose action 
	on the pair $\varkappa\prec\lambda$
	does not change $\lambda$, and replaces $\varkappa_1$ as follows:
	\begin{equation*}
		L_\alpha^{(1)}\colon \varkappa_1\mapsto 
		\lambda_2+Y_\alpha(\varkappa_1-\lambda_2).
	\end{equation*}
	The action of $R_\alpha^{(1)}$
	is simply the reflection of $L_\alpha^{(1)}$:
	\begin{equation*}
		R_{\alpha}^{(1)}\colon \varkappa_1\mapsto 
		\lambda_1-Y_\alpha(\lambda_1-\varkappa_1).
	\end{equation*}
	The notation for the L and R operators 
	is suggested by the directions in which they move $\varkappa_1$.
	See \Cref{fig:LR_map} for an illustration.
\end{definition}

\begin{remark}
	\label{rmk:alpha_0_or_1}
	If $\alpha=1$, both $L_1^{(1)}$ and $R_1^{(1)}$ are identity operators.
	If $\alpha=0$, then $Y_0(A)=0$ almost surely, and so the actions of 
	both $L_0^{(1)}$ or $R_0^{(1)}$ lead to the maximal possible displacement of $\varkappa_1$,
	respectively, to the left or to the right.
\end{remark}

\begin{figure}[htpb]
	\centering
	\includegraphics[width=.8\textwidth]{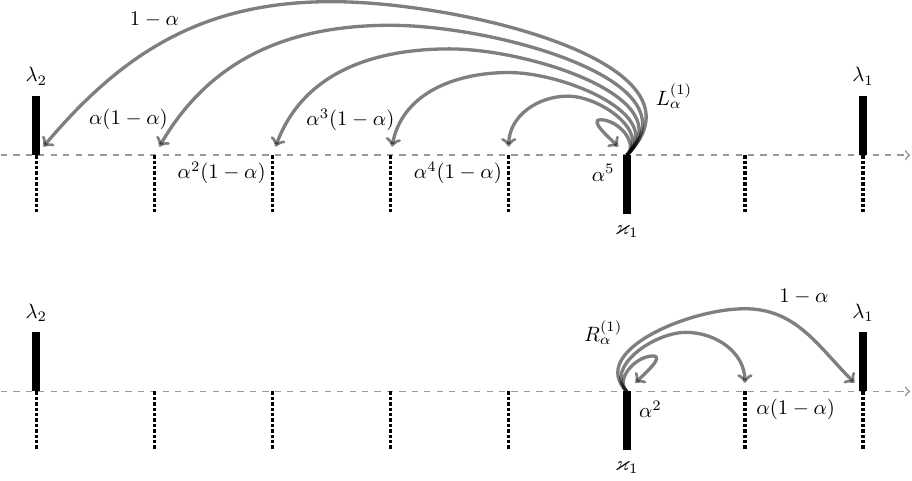}
	\caption{Probabilities of all possible moves in the maps $L_\alpha^{(1)}$ (top) and $R_\alpha^{(1)}$ (bottom).
	The parts of the partitions are represented by bold vertical bars.}
	\label{fig:LR_map}
\end{figure}

The next lemma plays a key role and 
will later generalize to other rows of the interlacing
array.
Denote by $s_i$, $i=1,2,\ldots $ the $i$-th elementary 
permutation of the spectral parameters,
\begin{equation}\label{eq:elementary_permutation}
	s_i\vec c:=(\ldots,c_{i-1},c_{i+1},c_i,c_{i+2},\ldots  ).
\end{equation}

\begin{lemma}
	\label{lemma:main_lemma_level_one}
	If $c_1\ge c_2$ and $c_1\ne 0$, then the Markov operator 
	$L_{c_2/c_1}^{(1)}$
	maps $\mathfrak{G}_{\vec c}$ to $\mathfrak{G}_{s_1\vec c}$.
	If $c_1\le c_2$ and $c_2\ne 0$, then the Markov operator
	$R_{c_1/c_2}^{(1)}$
	maps $\mathfrak{G}_{\vec c}$ to $\mathfrak{G}_{s_1\vec c}$.
\end{lemma}
\begin{proof}
	Let us consider only $L_\alpha^{(1)}$, the case of $R_\alpha^{(1)}$ is analogous.
	By \Cref{rmk:alpha_0_or_1}, when $c_1=c_2$, $L_1^{(1)}$
	is the identity. But in this case $s_1 \vec c=\vec c$, so there is nothing to prove.
	
	We can assume that $c_1>c_2$. Denote $\alpha=c_2/c_1$.
	Using the $\vec c$-Gibbs property, we see that given $\lambda=(\lambda_1,\lambda_2)$, the 
	conditional probability weight of $\varkappa=(\varkappa_1)$ is 
	proportional to $s_\varkappa(c_1)s_{\lambda/\varkappa}(c_2)$, which by \eqref{eq:skew_one_var}
	leads to 
	\begin{equation*}
		\mathrm{Prob}(\varkappa_1\mid \lambda)
		=
		\frac{\alpha^{-\varkappa_1}}{\sum_{k=\lambda_2}^{\lambda_1}\alpha^{-k}}.
	\end{equation*}
	The action of the operator $L_{\alpha}^{(1)}$ on this distribution is 
	readily computed:
	\begin{align*}
		&\sum_{\hat\varkappa_1=\lambda_2}^{\lambda_1}
		\mathrm{Prob}(\hat\varkappa_1\mid \lambda)\cdot L_\alpha^{(1)}(\hat\varkappa_1\to \varkappa_1)
		=
		\frac{1}{\sum_{k=\lambda_2}^{\lambda_1}\alpha^{-k}}
		\left( \alpha^{-\varkappa_1}\cdot\alpha^{\varkappa_1 - \lambda_2}+
		\sum_{\hat\varkappa_1=\varkappa_1+1}^{\lambda_1}\alpha^{-\hat\varkappa_1}\cdot
		(1-\alpha)\alpha^{\varkappa_1- \lambda_2} \right)
		\\&\hspace{30pt}=
		\frac{1}{\sum_{k=\lambda_2}^{\lambda_1}\alpha^{-k}}
		\left( 
			\alpha^{-\lambda_2}
			+
			\frac{\alpha^{-\lambda_1}-\alpha^{-\varkappa_1}}{1-\alpha}
			\cdot (1-\alpha)\alpha^{\varkappa_1-\lambda_2}
		\right)
		=
		\frac{\alpha^{\varkappa_1-\lambda_1-\lambda_2}}{\sum_{k=\lambda_2}^{\lambda_1}\alpha^{-k}}
		=
		\frac{\alpha^{\varkappa_1}}{\sum_{k=\lambda_2}^{\lambda_1}\alpha^{k}}.
	\end{align*}
	The final expression is the conditional probability weight
	of $\varkappa_1$ given $\lambda$ under the $s_1\vec c$-Gibbs property. 
	This completes the proof.
\end{proof}
\begin{remark}
	{\bf1.\/}
	In words, \Cref{lemma:main_lemma_level_one} states that
	the action of the L or R operators
	reverses the geometric distribution
	on the segment
	$[\lambda_2,\lambda_1]$.

	\noindent
	{\bf2.\/}
	Note also that we apply
	$L_{c_2/c_1}^{(1)}$ 
	only if $c_2\le c_1$ (and the opposite ordering restriction for $R_{c_1/c_2}^{(1)}$).
	If $c_1>c_2$ in $L_{c_2/c_1}^{(1)}$, then the algebraic computations in the
	proof of \Cref{lemma:main_lemma_level_one} are still valid. But the 
	operator 
	itself loses probabilistic meaning as
	some of its matrix elements become negative.
\end{remark}

\subsection{Remark. Relation to bijectivization}
\label{sub:bijectivisation_mention}

The Markov maps
of 
\Cref{def:LR_maps_one_level} 
which interchange the spectral parameters 
were suggested by the idea of bijectivization
of the Yang-Baxter equation 
first employed in \cite{BufetovPetrovYB2017} 
(see also \cite{ABB2018stochasticization},
\cite{BufetovMucciconiPetrov2018}). 

First, note that
one can deduce the symmetry of the 
skew Schur polynomials (\Cref{prop:Schur_symmetry})
from the Yang-Baxter equation. 
This argument is present, for example, in
\cite[Theorem 3.5]{Borodin2014vertex}
in a 
$U_q(\widehat{\mathfrak{sl}_2})$ setting
with additional parameters $q,s$ (the Schur case corresponds to 
$q=s=0$).

Next, bijectivization refines the Yang-Baxter equation into a pair of
forward and backward
local Markov moves which randomly update the configuration.
Here the locality means the following. Encode $\varkappa_1$ 
using the occupation variables 
$\{\eta_x\}_{x\in \mathbb{Z}}$, where $\eta_{\varkappa_1}=1$ 
and all other $\eta_x\equiv 0$.
The application of a single local Markov move (forward or backward)
would change one of the occupation variables. 

Then, considering a sequence of forward or backward moves
leads, respectively, to 
the L and R operators.
This can be seen by setting $t=s=0$ in \cite[Figure 4]{BufetovPetrovYB2017},
taking a sequence of these moves, and
passing from the occupation variables (equivalently, vertical arrows
in the notation of that paper)
to the elements of the interlacing
array. For brevity, we do not explain the details of derivation
of the L and R Markov operators from the bijectivization
as an independent proof of the key \Cref{lemma:main_lemma_level_one}
is rather straightforward.

\subsection{General case}

Let us now describe the Markov maps $L_{\alpha}^{(j)}$ and
$R_{\alpha}^{(j)}$ for general $j$. This is an extension of 
Definition \ref{def:LR_maps_one_level}. For the next definition we use the convention $\lambda^{(j)}_0 =
\infty$ and $\lambda^{(j)}_{j+1} = 0$ for all $j \in
\mathbb{Z}_{\geq 0}$ (recall that by \Cref{rmk:interlacing_number_of_rows}
in the $j$-th row of the interlacing array there cannot be more than $j$ nonzero entries).

\begin{definition}[The L and R maps, general case]
	\label{def:LR_maps}
	Fix $\alpha \in [0,1]$ and $j \geq 2$. Let $L_{\alpha}^{(j)}$ be the
	Markov map whose action on interlacing arrays of infinite depth
	$\{\lambda^{(i)}\}_{i\ge1}$
	does not change
	$\lambda^{(i)}$ for $i \neq j$, and replaces $\lambda^{(j)}$ as
	follows:
	\begin{equation*}
		L_{\alpha}^{(j)} : \lambda^{(j)}_{k} \mapsto \max\{\lambda^{(j-1)}_k , \lambda^{(j+1)}_{k+1} \} + Y_\alpha^{(k)}, \qquad k =1, \dots, j,
	\end{equation*}
	where $\{ Y_\alpha^{(k)} \}_{k=1}^j$ is a collection of independent
	truncated geometric random variables with $Y_\alpha^{(k)}$  distributed as
	$Y_{\alpha}\bigl(\lambda^{(j)}_{k} - \max\{\lambda^{(j-1)}_k ,
	\lambda^{(j+1)}_{k+1} \} \bigr)$.
	
	The action of $R_{\alpha}^{(j)}$ is simply the reflection of $L_{\alpha}^{(j)}$:
	\begin{equation*}
		R_{\alpha}^{(j)} : \lambda^{(j)}_{k} \mapsto \min\{ \lambda^{(j-1)}_{k-1}, \lambda^{(j+1)}_{k}\} - Y_\alpha^{(k)}, \qquad k =1, \dots, j,
	\end{equation*}
	where $\{ Y_\alpha^{(k)} \}_{k=1}^j$ is a collection of independent
	truncated geometric random variables with $Y_\alpha^{(k)}$  distributed as
	$Y_{\alpha}\bigl(\min\{  \lambda^{(j-1)}_{k-1}, \lambda^{(j+1)}_{k} \}
	- \lambda^{(j)}_{k} \bigr)$.

	In words, under both $L_\alpha^{(j)}$ and $R_\alpha^{(j)}$
	each $\lambda_k^{(j)}$, $k=1,\ldots,j $,
	is randomly independently moved to the left (resp., to the right)
	within the segment 
	\begin{equation}\label{eq:max_min_interlacing_segment}
		\Bigl[\max\{\lambda^{(j-1)}_k,\lambda^{(j+1)}_{k+1}\},
		\min\{\lambda_{k-1}^{(j-1)},\lambda^{(j+1)}_k \}\Bigr]
	\end{equation}
	to which $\lambda^{(j)}_k$ is constrained by interlacing. 
	The moves of each $\lambda_{k}^{(j)}$
	are exactly the same as on the first level
	and are governed by the truncated geometric random variables.
\end{definition}

The next statement is a generalization of
Lemma \ref{lemma:main_lemma_level_one}. 
Recall that $s_i$
denotes the $i$-th elementary permutation
of the spectral parameters $\vec{c}$.

\begin{proposition}
\label{prop:main_lemma_level_general}
	Fix $j\ge1$.
	If $c_j\ge c_{j+1}$ and $c_j \ne 0$, then the Markov operator
	$L_{c_{j+1}/c_j}^{(j)}$ maps $\mathfrak{G}_{\vec c}$ to
	$\mathfrak{G}_{s_j\vec c}$. If $c_j\le c_{j+1}$ and $c_{j+1}\ne 0$,
	then the Markov operator $R_{c_j/c_{j+1}}^{(j)}$ maps
	$\mathfrak{G}_{\vec c}$ to $\mathfrak{G}_{s_j\vec c}$.
\end{proposition}

\begin{proof}
	Let us consider $L^{(j)}_{\alpha}$ only; the case of
	$R_{\alpha}^{(j)}$ is analogous. Denote $\alpha = c_{j+1} / c_j$.
	We may assume that $\alpha\ne 1$ 
	as otherwise there is nothing to prove.
	Let us also take $j\ge2$ as the case $j=1$ is \Cref{lemma:main_lemma_level_one}.

	Using the $\vec{c}$-Gibbs property, we see that given
	all $\lambda^{(i)}$ with $i\ne j$,
	the conditional probability weight
	of $\lambda^{(j)}$ is proportional to
	$s_{\lambda^{(j)}} /
	s_{\lambda^{(j-1)}}(c_j) \,s_{\lambda^{(j+1)}/
	\lambda^{(j)}}(c_{j+1})$ (cf. \Cref{lemma:c_Gibbs_one_level}).
	By \eqref{eq:skew_one_var}, this implies
	\begin{equation}\label{eq:c_Gibbs_conditional}
		\mathrm{Prob}
		\left(\lambda^{(j)} \mid \lambda^{(i)}, i\ne j \right) 
		= 
		\prod_{k=1}^j \mathrm{P}
		\left(
		\lambda^{(j)}_k \mid \lambda^{(j-1)}_{k-1}, 
		\lambda^{(j-1)}_{k}, \lambda^{(j+1)}_{k}, 
		\lambda^{(j+1)}_{k+1} 
		\right),
	\end{equation}
	where 
	\begin{equation*}
	\mathrm{P}
	( m 
		\mid 
		a,b,c,d
	)
	=
	\alpha^{-m}\biggl(\,
		\sum_{
			r = \max\{b,d \}
		}^{
			\min\{a,c\}
		} 
	\alpha^{-r} \biggr)^{-1}.
	\end{equation*}
	For $\lambda^{(j)}= (\lambda^{(j)}_1, \dots , \lambda^{(j)}_j)$, the
	operator $L^{(j)}_{\alpha}$ acts on  each $\lambda^{(j)}_k$
	independently. Thus we may write $L^{(j)}_{\alpha}$ as a product of
	local Markov maps
	which act on each segment 
	\eqref{eq:max_min_interlacing_segment}
	in the same manner as in \Cref{sub:toy_case}.
	Similarly to 
	\Cref{lemma:main_lemma_level_one}
	we 
	conclude that the action of $L_\alpha^{(j)}$
	reverses each local geometric distribution 
	$\mathrm{P}(m\mid a,b,c,d)$.
	Therefore, $L_\alpha^{(j)}$ turns \eqref{eq:c_Gibbs_conditional}
	into the conditional probability weight of 
	$\lambda^{(j)}$ under a $s_j\vec c$-Gibbs measure. 
	This completes the proof.
\end{proof}

\section{Action on $q$-Gibbs measures}
\label{sec:action_on_q_Gibbs}

This section shows that suitably composed L maps 
preserve the class of $q$-Gibbs measures on interlacing arrays,
and describes how a $q$-Gibbs measure changes under this action.

\subsection{$q$-Gibbs property} 

Fix $q \in (0,1]$. A $\vec{c}$-Gibbs measure on
the set
$\mathcal{S}$ 
of infinite interlacing arrays
is called \emph{$q$-Gibbs} if  $c_i =q^{i-1}$ for all
$i \in \mathbb{Z}_{\geq 1}$. We denote the set of $q$-Gibbs measures
by~$\mathfrak{G}_q$. 

\begin{remark}
	One can define the volume of an interlacing
	array of finite depth $N$
	by 
	\begin{equation}
		\label{eq:volume}
		\mathrm{vol}(\lambda^{(1)}\prec\ldots\prec\lambda^{(N)}):
		=
		\sum_{i=1}^{N-1}|\lambda^{(i)}|.
	\end{equation}
	Then the $q$-Gibbs property is equivalent to saying that 
	conditioned on $\lambda^{(N)}$, the 
	probability weight of the
	interlacing
	array $\lambda^{(1)}\prec \ldots\prec \lambda^{(N)}$
	is proportional to $q^{-\mathrm{vol}(\lambda^{(1)}\prec\ldots\prec\lambda^{(N)} )}$
	(e.g., see \cite{mkrtchyan2017gue}).
	Note that sometimes (in particular, in \cite{Gorin2010q})
	the term ``$q$-Gibbs measures''
	refers to the elements of $\mathfrak{G}_{q^{-1}}$
	in our notation.
\end{remark}

When $q=1$, the $q$-Gibbs measures correspond to the uniform conditioning property
(cf. \Cref{rmk:uniform_Gibbs}).
Throughout this section we work under the assumption $0<q<1$.

\subsection{Iterated L map}
\label{sub:iterated_L_maps}

When $c_i=q^{i-1}$, we have $c_{i+1}<c_i$ for all $i$.
By \Cref{prop:main_lemma_level_general},
it means that the action of 
$L^{(i)}_{c_{i+1}/c_i}$ permutes the spectral parameters 
$q^{i}$ and $q^{i-1}$.
Iterating such $L^{(i)}$ from $i=1$ to infinity
and keeping track of the permutations of the spectral parameters,
we arrive at the following definition:
\begin{definition}[Iterated L map]
	\label{def:qGibbs_map}
	Let $\mathsf{M}$ be a probability measure on $\mathcal{S}$ and set
	$\mathsf{M}^{(0)} := \mathsf{M}$. Denote, inductively, 
	$\mathsf{M}^{(j)}:=\mathsf{M}^{(j-1)}L^{(j)}_{q^j}$ (see \Cref{fig:braid_q} for an illustration).
	Let $\mathbb{L}^{(q)}$ be the Markov map 
	which acts on probability measures
	on $\mathcal{S}$ by
	\begin{equation*}
		\mathbb{L}^{(q)}:
		\{ \mathsf{M}(\lambda^{(1)} , \dots, \lambda^{(N)}) \}_{N\ge1}
		\mapsto 
		\{ \mathsf{M}^{(N+1)}(\lambda^{(1)} , \dots, \lambda^{(N)}) \}_{N\ge1}.
	\end{equation*}
	Let us explain why $\mathbb{L}^{(q)}$ is well-defined.
	Recall that a
	probability measure on $\mathcal{S}$
	is uniquely determined by 
	a family of compatible joint distributions 
	of $(\lambda^{(1)},\ldots,\lambda^{(N)})$ 
	(cf. \Cref{sub:infinite_depth}).
	Next, for all $K > N$ we have $\mathsf{M}^{(K)}(\lambda^{(1)} , \dots, \lambda^{(N)}) 
	= 
	\mathsf{M}^{(N+1)}(\lambda^{(1)}, \dots, \lambda^{(N)})$.
	This guarantees that the collection of 
	measures 
	$\{ \mathsf{M}^{(N+1)}(\lambda^{(1)} , \dots, \lambda^{(N)}) \}_{N\ge1}$
	is indeed compatible, and thus defines a measure on $\mathcal{S}$
	which we denote by $\mathsf{M}\,\mathbb{L}^{(q)}$.
\end{definition}

\begin{figure}[htpb]
	\centering
	\includegraphics{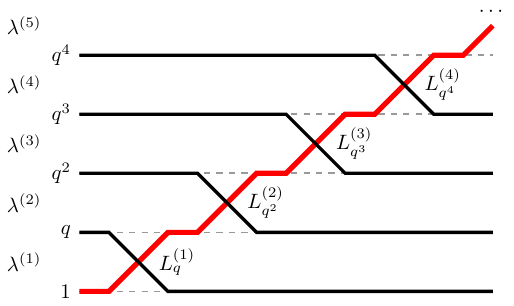}
	\caption{Construction of the map  $\mathbb{L}^{(q)}$. 
	The spectral parameters $q^j$ correspond to 
	the action on 
	$q$-Gibbs measures
	considered in 
	\Cref{sub:action_of_L_on_q_Gibbs},
	and the lines indicate the swapping of the 
	spectral parameters after each $j$-th map
	$L^{(j)}_{q^j}$.}
	\label{fig:braid_q}
\end{figure}

\subsection{$q$-Gibbs harmonic families}

Let $\mathsf{M}$ be 
a $q$-Gibbs measure on $\mathcal{S}$. By the $q$-Gibbs property, 
for each $N\ge1$
the probability weight of
$\lambda^{(1)},\ldots,\lambda^{(N)}$
is represented as a product of the marginal probability
weight
of $\lambda^{(N)}$ and a $q$-Gibbs factor
corresponding to the conditional distribution of
$\lambda^{(1)}\prec \ldots\prec\lambda^{(N-1)} $ given $\lambda^{(N)}$.
This allows to write
\begin{equation*}
	\mathsf{M}(\lambda^{(1)}\prec\ldots\prec\lambda^{(N)})=
	s_{\lambda^{(1)} } (1) 
	s_{\lambda^{(2)}/ \lambda^{(1)}}(q)
	\ldots 
	s_{\lambda^{(N)} / \lambda^{(N-1)}}(q^{N-1})
	\cdot \varphi_N(\lambda^{(N)}),
\end{equation*}
where $\varphi_N$ is a function on the $N$-th level of the array defined as
\begin{equation}\label{eq:phi_N_def}
	\varphi_N(\nu)=\frac{\mathsf{M}(\lambda^{(N)}=\nu)}{s_{\nu}(1,q,\ldots,q^{N-1} )}.
\end{equation}
Because the functions $\varphi_N$ for different $N$ come from the same 
$q$-Gibbs measure $\mathsf{M}$, they must be compatible. 
This compatibility relation reads
\begin{equation}
	\label{eq:q_Gibbs_harmonic}
	\sum_{\lambda\colon \mu\prec \lambda} 
	\varphi_{N}(\lambda) \,
	s_{\lambda / \mu}(q^{N-1})
	= 
	\varphi_{N-1}(\mu)
\end{equation}
for all $N\ge1$ and all $\mu=(\mu_1,\ldots,\mu_{N-1})$ on the $(N-1)$-st level of the array
(at the zeroth level we set $\varphi_0(\varnothing)=1$, 
by agreement).
We call a family of functions $\{\varphi_N\}$
satisfying \eqref{eq:q_Gibbs_harmonic} and
$\varphi_0(\varnothing)=1$
a 
\emph{$q$-Gibbs harmonic family}.
The term ``harmonic'' comes from the Vershik--Kerov theory of the boundary 
of branching graphs (e.g., see 
\cite{Kerov1998}).
Clearly, a $q$-Gibbs measure on $\mathcal{S}$ 
is uniquely determined by 
its associated $q$-Gibbs harmonic family $\{\varphi_N\}$.

\subsection{Action of the iterated L map on $q$-Gibbs measures}
\label{sub:action_of_L_on_q_Gibbs}
	
If $\mathsf{M}\in\mathfrak{G}_q$,
then the action of $\mathbb{L}^{(q)}$ (that is, the 
sequence of the Markov maps $L^{(j)}_{q^j}$)
on $\mathsf{M}$ swaps the spectral parameters
as in \Cref{fig:braid_q},
moving $c_1=1$ all the way up to infinity
where it ``disappears''.
The resulting spectral parameters
$(q,q^2,q^3,\ldots )$ are proportional
to the original ones. This suggests that
$\mathbb{L}^{(q)}$ should preserve the class
of $q$-Gibbs measures. The next result 
shows that this is 
indeed the case, 
and also describes the action of $\mathbb{L}^{(q)}$
on $\mathfrak{G}_q$ in the language of harmonic families.

\begin{theorem}
	\label{thm:main_prop_qGibbs}
	The Markov map $\mathbb{L}^{(q)}$ maps preserves 
	$\mathfrak{G}_q$, the set of $q$-Gibbs measures on $\mathcal{S}$.
	More precisely, 
	$\mathbb{L}^{(q)}$ maps each 
	$q$-Gibbs harmonic family 
	$\{ \varphi_N\}_{N \in \mathbb{Z}_{\geq 1}}$ to a new 
	$q$-Gibbs harmonic family 
	$\{\hat{\varphi}_N \}_{N \in \mathbb{Z}_{\geq 1}}$ as follows:
	\begin{equation}
		\label{eq:action_q_Gibbs_main_thm}
		\hat{\varphi}_N (\mu) 
		=
		q^{|\mu|} 
		\sum_{\lambda} 
		\varphi_{N+1}(\lambda) s_{\lambda / \mu}(1)
		=q^{|\mu|}
		\sum_{\lambda\colon \mu\prec \lambda} 
		\varphi_{N+1}(\lambda) 
		.
	\end{equation}
\end{theorem}
\begin{proof}
	The second equality in \eqref{eq:action_q_Gibbs_main_thm} immediately
	follows from \eqref{eq:skew_one_var}. 
	Let us first explain why the sum in \eqref{eq:action_q_Gibbs_main_thm} is finite.
	We have by the definition 
	\eqref{eq:phi_N_def}
	of $\varphi_N$:
	\begin{equation}
		\label{eq:action_q_Gibbs_main_thm_proof1}
		1=
		\sum_{\lambda}
		\varphi_{N+1}(\lambda)\,s_\lambda(1,q,\ldots,q^{N})
		\ge
		\sum_{\lambda}
		\varphi_{N+1}(\lambda)\,
		1^{\lambda_1}q^{\lambda_2}\ldots(q^N)^{\lambda_{N+1}},
	\end{equation}
	where we bounded the Schur polynomial from below 
	by taking 
	one of its monomials (since all the monomials are nonnegative).
	The condition $\mu\prec\lambda$ in 
	\eqref{eq:action_q_Gibbs_main_thm}
	means that only the sum over $\lambda_1$ 
	in 
	\eqref{eq:action_q_Gibbs_main_thm} is over an infinite set,
	and it thus converges thanks to 
	\eqref{eq:action_q_Gibbs_main_thm_proof1}.
	
	Now let $\{\lambda^{(i)}\}$ be a random interlacing
	array distributed according to the $q$-Gibbs measure
	coming from $\{\varphi_N\}$.
	Let the random array $\{\theta^{(i)}\}$ be the image
	of $\{\lambda^{(i)}\}$
	under $\mathbb{L}^{(q)}$.
	Fix $N\ge1$.
	The distribution of $\theta_N$
	(described by the function $\hat\varphi_N$ which we aim to compute)
	is a result of applying the 
	sequence of Markov maps
	$L_q^{(1)},\ldots,L_{q^{N}}^{(N)}$ (in this order).
	Because the last of these operators depends on $\lambda^{(N+1)}$,
	we see that the distribution of $\theta^{(N)}$
	is not determined only by the joint distribution of $\lambda^{(1)},\ldots,\lambda^{(N)}$.
	In other words, to compute $\hat\varphi_N$ 
	we need to first extend $\varphi_N$ to $\varphi_{N+1}$, and utilize the $q$-Gibbs property. 
	
	Let us apply this idea.
	Fix $\lambda^{(N+1)}=\lambda$. This condition
	completely determines
	the conditional joint distribution of $\lambda^{(1)},\ldots,\lambda^{(N)}$ via the $q$-Gibbs property.
	By iterating \Cref{prop:main_lemma_level_general}, we see that
	after applying the Markov maps $L_q^{(1)},\ldots,L_{q^{N-1}}^{(N-1)} $, the joint distribution of 
	$\lambda^{(N)}$ and $\theta^{(N-1)}$, conditioned on $\lambda^{(N+1)}=\lambda$
	comes from the $(q,q^2,\ldots,q^{N-1},1,q^{N})$-Gibbs property:
	\begin{equation*}
		\mathrm{Prob}\bigl( \lambda^{(N)}=\varkappa,\;\theta^{(N-1)}=\nu\mid \lambda^{(N+1)}=\lambda \bigr)
		=
		\frac{s_\nu(q,q^2,\ldots,q^{N-1} ) s_{\varkappa/\nu}(1) s_{\lambda/\varkappa}(q^{N})}{s_\lambda(1,q,\ldots,q^{N} )}.
	\end{equation*}
	After the application of $L^{(N)}_{q^N}$, the partition
	$\lambda^{(N)}$ turns into $\theta^{(N)}$,
	and we similarly have
	\begin{equation*}
		\mathrm{Prob}\bigl( \theta^{(N)}=\mu,\;\theta^{(N-1)}=\nu\mid \lambda^{(N+1)}=\lambda \bigr)
		=
		\frac{s_\nu(q,q^2,\ldots,q^{N-1})s_{\mu/\nu}(q^N)s_{\lambda/\mu}(1)}{s_\lambda(1,q,\ldots,q^N )}.
	\end{equation*}
	Let us rewrite the last expression to
	compare it to the $q$-Gibbs conditional distribution.
	In the numerator, due to the homogeneity 
	of Schur and skew Schur polynomials, we have:
	\begin{equation}
		\label{eq:power_computation}
		\begin{split}
			s_{\lambda/\mu}(1)s_{\mu/\nu}(q^N)s_\nu(q,\ldots,q^{N-1} )
			&= 
			q^{|\nu|} s_\nu(1,q,\ldots,q^{N-2} ) q^{|\mu|-|\nu|}s_{\mu/\nu}(q^{N-1})  s_{\lambda/\mu}(1)
			\\
			&= 
			q^{|\mu|} s_\nu(1,q,\ldots,q^{N-2} ) s_{\mu/\nu}(q^{N-1}) s_{\lambda/\mu}(1).
		\end{split}
	\end{equation}
	To extract from this the marginal distribution of $\theta^{(N)}$
	(that is, to get to $\hat \varphi_{N}$),
	we need to multiply
	\eqref{eq:power_computation}
	by $\mathrm{Prob}(\lambda^{(N+1)}=\lambda)/s_\lambda(1,\ldots,q^N)$ (which is exactly $\varphi_{N+1}(\lambda)$)
	and sum the resulting expression over both $\lambda$ and $\nu$. We have
	\begin{equation*}
		\mathrm{Prob}(\theta^{(N)}=\mu)=
		\sum_{\nu,\lambda\colon \nu\prec \mu\prec\lambda}
		q^{|\mu|} s_\nu(1,\ldots,q^{N-2} ) s_{\mu/\nu}(q^{N-1}) s_{\lambda/\mu}(1)
		\,
		\varphi_{N+1}(\lambda).
	\end{equation*}
	The sum over $\nu$ is simplified using
	the branching rule \eqref{eq:branching}, and so 
	\begin{equation*}
		\frac{\mathrm{Prob}(\theta^{(N)}=\mu)}{s_\mu(1,q,\ldots,q^{N-1} )}=
		q^{|\mu|}\sum_{\lambda\colon \mu\prec \lambda}
		\varphi_{N+1}(\lambda)\,s_{\lambda/\mu}(1).
	\end{equation*}
	We see that at the level of marginal distributions, the family 
	$\left\{ \varphi_N \right\}$ turns into $\left\{ \hat \varphi_N \right\}$,
	where
	$\hat \varphi_N$ is defined by 
	\eqref{eq:action_q_Gibbs_main_thm}.
	
	It remains to show that the new family $\{\hat \varphi_N\}$
	satisfies the $q$-Gibbs harmonicity. That is, we want to show for all $N$ that
	\begin{equation*}
		\sum_\mu\hat \varphi_N(\mu)s_{\mu/\nu}(q^{N-1})=
		\hat \varphi_{N-1}(\nu)
		=
		q^{|\nu|}
		\sum_{\varkappa}\varphi_N(\varkappa)s_{\varkappa/\nu}(1)
	\end{equation*}
	(the second equality is simply the definition of $\hat \varphi_{N-1}$).
	We have
	\begin{align*}
		\sum_\mu\hat \varphi_N(\mu)s_{\mu/\nu}(q^{N-1})
		&=
		\sum_{\mu,\lambda}
		\varphi_{N+1}(\lambda)s_{\lambda/\mu}(1)q^{|\mu|}
		s_{\mu/\nu}(q^{N-1})
		\\
		&=
		q^{|\nu|}
		\sum_{\mu,\lambda}
		\varphi_{N+1}(\lambda)s_{\lambda/\mu}(1)
		s_{\mu/\nu}(q^{N})
		\\
		&=
		q^{|\nu|}
		\sum_{\lambda}
		\varphi_{N+1}(\lambda)s_{\lambda/\nu}(1,q^N)
		\\
		&=
		q^{|\nu|}
		\sum_{\lambda,\varkappa}
		\varphi_{N+1}(\lambda)s_{\lambda/\varkappa}(q^N)s_{\varkappa/\nu}(1)
		\\
		&=
		q^{|\nu|}
		\sum_{\varkappa}
		\varphi_{N}(\varkappa)s_{\varkappa/\nu}(1),
	\end{align*}
	as desired.
	In the last step we used the harmonicity of the original family 
	$\left\{ \varphi_N \right\}$. This completes the proof.
\end{proof}

\begin{remark}
	Note that \Cref{thm:main_prop_qGibbs}
	fundamentally relies on the fact that the $q$-Gibbs measure lives
	on an infinite array. 
	Indeed, for an array of finite depth it is not possible to 
	move the spectral parameter $1$ all the way up to infinity.
	In the proof of \Cref{thm:main_prop_qGibbs}
	we use the fact that the array has infinite depth
	when we extend $\varphi_N$ to $\varphi_{N+1}$.
	The case of arrays of finite depth is discussed in 
	\Cref{sec:extensions} below.
\end{remark}

\subsection{Application to Schur processes and TASEP with geometric speeds}
\label{sub:q_schur_TASEP}

Schur processes $\mathbb{P}[\vec c\mid \rho_t]$
with $c_i=q^{i-1}$
are particular cases of $q$-Gibbs measures
with
\begin{equation*}
	\varphi_N(\lambda)=
	e^{-t(1+q+\ldots+q^{N-1} )}\,
	\frac{s_\lambda(\rho_t)s_\lambda(1,q,\ldots,q^{N-1} )}
	{s_\lambda(1,q,\ldots,q^{N-1} )}
	=
	e^{-t(1+q+\ldots+q^{N-1} )}
	s_\lambda(\rho_t),
\end{equation*}
where we took into account the normalization of the Schur measures.
The Markov map $\mathbb{L}^{(q)}$ acts on these Schur processes
as follows:
\begin{align*}
	\hat \varphi_N(\mu)
	&=
	q^{|\mu|}
	e^{-t(1+q+\ldots+q^{N} )}
	\sum_{\lambda}s_\lambda(\rho_t)\,s_{\lambda/\mu}(1)
	\\&
	=
	e^{t}
	e^{-t(1+q+\ldots+q^{N} )}
	q^{|\mu|}
	s_\mu(\rho_t)
	\\&=
	e^{-qt(1+q+\ldots+q^{N-1} )}
	s_\mu(\rho_{q\cdot t}),
\end{align*}
where we used the skew Cauchy identity
(\eqref{eq:cauchy} with $\varkappa=\varnothing$)
and the homogeneity of the Schur polynomials 
(both properties
clearly survive the Plancherel limit \eqref{eq:Plancherel_limit}).
Therefore, we have
\begin{equation}
	\label{eq:action_q_t}
	\mathbb{P}[(1,q,q^2,\ldots )\mid \rho_t]
	\,\mathbb{L}^{(q)}
	=
	\mathbb{P}[(1,q,q^2,\ldots )\mid \rho_{qt}].
\end{equation}

Recall that by \Cref{thm:BorFerr},
the joint distribution of the quantities $\{\lambda^{(N)}_N-N \}_{N\ge1}$
under the Schur process
$\mathbb{P}[(1,q,q^2,\ldots )\mid \rho_t]$
is the same as the joint distribution of the particle 
locations $\{x_N(t)\}_{N\ge1}$ at time $t$
of the TASEP with particle
speeds $c_i=q^{i-1}$
and the step initial configuration.
Denote this joint distribution of particles
$\{x_N(t)\}$
by $\upmu_t^{(q)}$.

Our next observation is that the action of 
the Markov map $\mathbb{L}^{(q)}$ 
on the random interlacing array $\{\lambda^{(N)}\}_{N\ge1}$
can be projected to the leftmost components 
$\{\lambda^{(N)}_N \}_{N\ge1}$,
and the result is still a Markov map. In more detail,
let $\{\theta^{(N)}\}_{N\ge1}$
be the random interlacing array which is the image of $\{\lambda^{(N)}\}_{N\ge1}$
under $\mathbb{L}^{(q)}$. 
From the very definition of $\mathbb{L}^{(q)}$,
we see that 
conditioned on $\{\lambda^{(N)}\}_{N\ge1}$,
the distribution of 
$\{\theta^{(N)}_N \}_{N\ge1}$ depends only on the leftmost components
$\{\lambda^{(N)}_N \}_{N\ge1}$, and not on the rest of the array $\lambda$.
Let us describe this projection of $\mathbb{L}^{(q)}$ explicitly
in terms of locations of the TASEP particles
$x_N$ (via the identification $x_N=\lambda^{(N)}_N-N$).
Recall from \Cref{sub:TASEP_def}
that $\mathcal{C}$ stands for the space of 
left-packed, right-finite
particle configurations on $\mathbb{Z}$.

\begin{definition}
	\label{def:L_projection}
	Let $0<q<1$. We aim to define a Markov map $\mathbf{L}^{(q)}$ on $\mathcal{C}$.
	Fix
	a configuration $x_1>x_2>\ldots $ in $\mathcal{C}$.
	By definition, its 
	random image
	$\hat x_1>\hat x_2>\ldots $ under the action of $\mathbf{L}^{(q)}$
	is
	\begin{equation*}
		\hat x_i=x_{i+1}+1+Y_{q^{\scriptstyle i}}(x_i-x_{i+1}-1),\qquad 
		i=1,2,\ldots, 
	\end{equation*}
where the $Y_{q^{\scriptstyle i}}$'s are independent truncated
	geometric random variables (see \Cref{def:truncated_geometric}).
\end{definition}

\begin{remark}
	A homogeneous version of $\mathbf{L}^{(q)}$
	appeared in
	\cite{Brankov_1996},
	it is solvable through coordinate Bethe Ansatz
	\cite{rakos2005current}.
\end{remark}

\Cref{thm:main_prop_qGibbs}, identity \eqref{eq:action_q_t}, 
and the fact that 
$\mathbf{L}^{(q)}$ is a projection of
$\mathbb{L}^{(q)}$ 
immediately imply the following result:
\begin{theorem}
	\label{thm:action_on_q_geom_TASEP}
	For any $t\ge 0$, we have
	\begin{equation*}
		\upmu_t^{(q)}\,\mathbf{L}^{(q)}=\upmu_{qt}^{(q)},
	\end{equation*}
	where $\upmu_t^{(q)}$ is the distribution of the TASEP
	with geometric rates (with ratio $q$) and step initial configuration, 
	and $\mathbf{L}^{(q)}$ is the Markov map 
	from
	\Cref{def:L_projection}.
\end{theorem}

%----------------------------------------------------------------------------

\section{Limit \texorpdfstring{$q\to1$}{q->1} and proof of the main result}
\label{sec:limit_q_1}

Here we take the limit as $q\to1$
of the results of the previous section, and 
arrive at a continuous-time Markov chain 
mapping the TASEP distributions backwards in time.
This proves our main result, \Cref{thm:intro_action_on_TASEP}.

\medskip

Iterate \Cref{thm:action_on_q_geom_TASEP} to observe that
for any $T\in \mathbb{Z}_{\ge1}$:
\begin{equation}
	\label{eq:q_action_iterated}
	\upmu_t^{(q)}
	(\mathbf{L}^{(q)})^T
	=
	\upmu^{(q)}_{q^T t},
\end{equation}
where $(\mathbf{L}^{(q)})^T$ simply denotes the $T$-th power.
Next, introduce the scaling:
\begin{equation}
	\label{eq:q_scaling}
	q=e^{-\varepsilon},\qquad 
	T=\lfloor \tau/\varepsilon \rfloor ,
\end{equation}
where $\varepsilon>0$ will go to zero,
and $\tau\in \mathbb{R}_{\ge0}$ is the scaled continuous time.
Clearly, we have $q^T=e^{-\tau}(1+O(\varepsilon))$.
We aim to take the limit as $q\to1$ in
\eqref{eq:q_action_iterated}.

Recall that by $\upmu_t$, $t\in \mathbb{R}_{\ge0}$, we denote
the distribution of the TASEP with constant speeds $c_i\equiv 1$
at time $t$, started from the step initial configuration.
Also recall that $\mathcal{C}$ is the space 
of 
left-packed, right-finite
particle configurations on $\mathbb{Z}$.
The space $\mathcal{C}$ has a natural partial order:
$\mathbf{x}$ precedes $\mathbf{y}$ if 
$x_i\le y_i$ for all $i$.

\begin{lemma}
	\label{lemma:finite_part_of_C}
	For any fixed $\tau,t\in \mathbb{R}_{\ge0}$ and
	any $\delta>0$ 
	there exists a finite set $\mathcal{C}^\delta
	=\mathcal{C}^\delta(t,\tau)\subset \mathcal{C}$
	such that 
	\begin{equation*}
		\upmu_t(\mathcal{C}^\delta)>1-\delta,
		\qquad 
		\upmu_{e^{-\tau}t}(\mathcal{C}^\delta)>1-\delta,
		\qquad 
		\upmu_t^{(q)}(\mathcal{C}^\delta)>1-\delta,
		\qquad 
		\upmu_{q^T t}^{(q)}(\mathcal{C}^\delta)>1-\delta
	\end{equation*}
	for
	all sufficiently small $\varepsilon>0$.
\end{lemma}
\begin{proof}
	Take finite $\mathcal{C}^\delta\subset\mathcal{C}$
	such that
	$\upmu_{t}(\mathcal{C}^\delta)>1-\delta$,
	and, moreover, $\mathcal{C}^\delta$
	is closed with respect to the partial order
	(i.e., if $\mathbf{x}$ precedes $\mathbf{y}$ and 
	$\mathbf{y}\in \mathcal{C}^\delta$, 
	then $\mathbf{x}\in \mathcal{C}^\delta$).
	This is possible because 
	$\upmu_t$ is a probability measure on $\mathcal{C}$,
	and 
	closing a finite set with respect
	to our partial order keeps it finite.
	(One can even estimate the size of $\mathcal{C}^\delta$
	because the first particle $x_1(t)$ performs speed 1
	directed
	a random walk.)

	Next, $\upmu_{e^{-\tau}t}(\mathcal{C}^\delta)>1-\delta$
	because the TASEP dynamics almost surely increases
	the configuration with respect to the order.
	The rest of the claim 
	follows by monotonically coupling the TASEP $\upmu_{\bullet}$
	with constant speeds to
	the TASEP $\upmu_{\bullet}^{(q)}$
	with the $q$-geometric speeds. 
	Here monotonicity means that the TASEP 
	with the $q$-geometric speeds is always behind
	(in our partial order)
	the $q=1$ TASEP; this monotone 
	coupling exists since $q<1$.
\end{proof}

By \Cref{lemma:finite_part_of_C}, 
it suffices to consider the limit of identity
\eqref{eq:q_action_iterated}
as $q\to1$ on finite subsets of~$\mathcal{C}$.
In the right-hand side we immediately get 
$\upmu^{(q)}_{q^T t}\to \upmu_{e^{-\tau}t}$.
In the left-hand side we have $\upmu_t^{(q)}\to \upmu_t$.
It remains to take the limit of the $T$-th power
of the Markov map $\mathbf{L}^{(q)}$.

The limit transition in $(\mathbf{L}^{(q)})^{T}$
is in the spirit of the classical Poisson approximation 
to the binomial distribution --- 
the probability of jumps gets smaller, but the number of trials (i.e., 
the discrete time) scales accordingly. 
More precisely, 
we have for the random variables 
$Y_{q^{\scriptstyle k}}$ in 
\Cref{def:L_projection}:
\begin{equation}
	\label{eq:rates_computation}
	\begin{split}
		\mathrm{Prob}(Y_{q^{\scriptstyle k}}(A)=m)
		&=
		\begin{cases}
			(1-q^{k})q^{mk},&0\le m<A;\\
			q^{Ak},&m=A
		\end{cases}
		\\
		&=
		\begin{cases}
			k\varepsilon+O(\varepsilon^2),& 0\le m\le A;\\
			1-Ak\varepsilon+O(\varepsilon^2),&m=A.
		\end{cases}
	\end{split}
\end{equation}
This leads to the following definition
of the continuous-time backwards dynamics:
\begin{definition}[Backwards Hammersley-type 
	process $\mathbf{L}_\tau$]
	\label{def:text_def_L_tau}
	Consider the continuous-time dynamics
	on $\mathcal{C}$ defined as follows.
	Each particle $x_k$, $k=1,2,\ldots $ independently jumps to the left 
	to one of the holes $\{x_{k+1}+1,x_{k+1}+2,\ldots,x_k-1 \}$
	at rate $k$ \emph{per hole}. Equivalently, 
	each particle $x_k$ has an independent exponential clock
	of rate $k(x_k-x_{k+1}-1)$; when the clock rings,
	$x_k$ selects a hole between $x_{k+1}$ and $x_k$
	uniformly at random and instantaneously moves there.\footnote{The mechanism 
	of jumping into a hole selected uniformly at random
	is similar to the Hammersley process \cite{hammersley1972few}, \cite{aldous1995hammersley}.
	Therefore, we will sometimes refer to $\mathbf{L}_\tau$ (as well as its two-dimensional
	version $\mathbb{L}_\tau$ discussed in \Cref{sub:general_schur}
	below) as \emph{Hammersley-type} process (BHP, for short).}
	
	Note that for configurations in $\mathcal{C}$, 
	the total jump rate of all particles is always finite.
	Therefore, the dynamics on $\mathcal{C}$
	is well-defined. 
	Denote by $\mathbf{L}_\tau$, $\tau\in \mathbb{R}_{\ge0}$,
	the Markov transition operator of this dynamics 
	from time $0$ to time $\tau$ (note that 
	the dynamics is time-homogeneous).
	Observe that the step configuration
	($x_i=-i$ for all $i=1,2,\ldots $)
	is absorbing for the backwards dynamics $\mathbf{L}_\tau$.
\end{definition}

Thanks to \Cref{lemma:finite_part_of_C} and 
\eqref{eq:rates_computation}, 
we have the convergence
$(\mathbf{L}^{(q)})^{T}\to \mathbf{L}_\tau$. 
This completes the proof of the main theorem
$\upmu_t\,\mathbf{L}_\tau=\upmu_{\,e^{-\tau}t}$.

\section{Stationary dynamics on the TASEP measure}
\label{sec:equil_dyn}

Here we illustrate the relation between the
TASEP and the backwards Hammersley-type process
by constructing a Markov dynamics preserving the TASEP measure $\upmu_t$.
We also discuss hydrodynamics of these two processes.

In this section we
denote particle configurations by occupation
variables $\eta:\mathbb{Z} \rightarrow \{0,1\}$, with $\eta(x) =1$ if
there is a particle at location $x\in \mathbb{Z}$, and $\eta(x) =0$
otherwise. The step initial configuration is $\eta(x)=1$ iff $x<0$.
Recall that by $\mathcal{C}$ we denote the space of
left-packed, right-finite configurations. 
Denote by $\overline{\mathcal{C}}=\left\{ 0,1 \right\}^{\mathbb{Z}}$
the space of all particle configurations in $\mathbb{Z}$.

\subsection{Definition of the stationary dynamics}
\label{sub:equil_dyn_defninition}

Let $A^{T}:= A^{\mathrm{TASEP}}$ be the infinitesimal generator for the TASEP with
homogeneous particle speeds $c_i=1$ (\Cref{sub:TASEP_def}), and $\{\mathbf{T}_t\}_{t\ge0}$
be the corresponding Markov semigroup.
Let $A^{L}: = A^{\mathrm{BHP}}$ the infinitesimal generator of the backwards
Hammersley-type process (BHP), see \Cref{def:text_def_L_tau},
and $\{\mathbf{\mathbf{L}_\tau}\}_{\tau\ge0}$ denote the BHP semigroup.
For a fixed configuration $\eta \in
\mathcal{C}$, we denote by $\eta^{x,y}$, $x\ne y$, the configuration
\begin{equation*}
	\eta^{x, y}(z) 
	= 
	\begin{cases} 
		\eta(z) + 1, & z =y; 
		\\ 
		\eta(z),  
		&	z \neq x, y; 
		\\ 
		\eta(z) - 1, & z = x.
	\end{cases}
\end{equation*}
In words, $\eta^{x, y}$
corresponds to a particle
jumping from location $x \in \mathbb{Z}$ to location $y \in
\mathbb{Z}$. Note that $\eta^{x,y}$ may not be in $\mathcal{C}$ even if $\eta\in \mathcal{C}$.

The
infinitesimal generator for the TASEP acts as follows:
\begin{equation}
	\label{eq:TASEP_gen}
	(A^T f) (\eta) = \sum_{x \in \mathbb{Z}}\eta(x) (1 - \eta(x+1)) \bigl(f(\eta^{x, x+1}) - f(\eta)\bigr),
\end{equation}
for $f$ a cylindrical function on $\eta \in \mathcal{C}$ (i.e.~a function that depends on finitely many coordinates of $\eta$). 
The factor $\eta(x) (1 - \eta(x+1))$
takes care of the TASEP exclusion rule.
The infinitesimal generator of the BHP acts as follows:
\begin{equation}
	\label{eq:BHP_gen}
	(A^{L} f)(\eta) 
	= 
	\sum_{x \in \mathbb{Z}} 
	\eta(x)
	\Bigl(\sum_{y =x}^{\infty} \eta(y)\Bigr)
	\sum_{m=1}^{\infty}
	\Bigl(   \prod_{k=1}^{m}(1 - \eta(x- k)) \Bigr) 
	\bigl(f(\eta^{x, x-m}) -f(\eta)\bigr),
\end{equation}
for $f$ a cylindrical function on $\eta$.
Note that summations in the action of $A^L$ are well defined since 
for $\eta\in \mathcal{C}$ we have
$\eta(x) = 0$ for 
$x\gg0$ and $\eta(x) =1$ for $x\ll0$.

Recall that $\upmu_t$ is the distribution of the TASEP configuration
at time $t$ started from the step initial configuration.
Denote the corresponding
random particle configuration by $\eta_t$.
We have $\eta_t\in \mathcal{C}$ almost surely.

For any $t\in \mathbb{R}_{>0}$, define the operator
\begin{equation}
	\label{eq:equil_gen}
	A:= tA^T+A^L.
\end{equation}
This is the generator of the continuous-time Markov
process
which is a combination
of the BHP and the TASEP sped up by the factor of $t$.
By a ``combination'' we mean that both processes run in parallel.

\begin{proposition}
	\label{prop:invariance}
	The TASEP distribution $\{\eta_t\}$ is invariant under the continuous-time Markov process with generator 
	$A$, that is,
	\begin{equation*}
	 \mathbb{E}\left[ (A f) (\eta_t) \right] = 0
	\end{equation*}
	for all cylinder functions $f$.
\end{proposition}
\begin{proof}
	By Theorem \ref{thm:intro_action_on_TASEP}, we have 
	\begin{equation}
		\upmu_t\,
		\mathbf{L}_\tau \,
		\mathbf{T}_{t (1 - e^{-\tau})}
		=
		\upmu_{t}
	\end{equation}
	for any $t, \tau \geq 0$. Fixing $t \geq 0$, differentiating the
	above identity in $\tau$, 
	and sending $\tau$ to zero, we get $\upmu_t\,(t A^{T} + A^{L}) =0$. This
	establishes the result. 
\end{proof}

\begin{remark}
	It should be possible to show that the process 
	with the generator \eqref{eq:equil_gen},
	started from any configuration $\mathbf{x}\in \mathcal{C}$,
	converges (as time goes to infinity)
	to its stationary distribution $\upmu_t$. 
	However, we do not focus on this question in the present paper.
\end{remark}

A local version of \Cref{prop:invariance}
holds, too.
That is, the Bernoulli measures of any given density
$\rho\in[0,1]$
on particle configurations on $\mathbb{Z}$
are invariant under both the TASEP and the 
homogeneous version of the BHP.
(Locally the rates under BHP are constant, so
the invariance should be considered under the homogeneous BHP.)
The remarkable content of
\Cref{prop:invariance}
is that the invariance is global
on ``out-of-equilibrium'' random configurations
with
the distribution $\upmu_t$,
if the speeds of the TASEP and the inhomogeneous 
BHP are related in as in \eqref{eq:equil_gen}.

As a consequence of \Cref{prop:invariance}, 
let us take a specific function of the configuration:
\begin{equation}
	\label{eq:F_of_height}
	N^0:=\eta(0)+\eta(1)+\eta(2)+\ldots,\qquad  
	f(\eta):=G(N^0),
\end{equation}
where $G(\cdot)$ is a function $\mathbb{Z}_{\ge0}\to \mathbb{R}$.
Note that $2 N^0$ is the height function
at zero.
Let $\eta_t$ be the random configuration
of the TASEP at time $t$ with the step initial configuration,
and $N^0_t:=\eta_t(0)+\eta_t(1)+\ldots $.

\begin{corollary}
	\label{cor:microscopic_equation}
	With the above notation, we have
	\begin{equation*}
		\frac{\partial}{\partial t}\,\mathbb{E}\,
		G(N^0_t)=
		-\frac{1}{t}\,
		\mathbb{E}\left( 
		N^0_t\left( G(N^0_t-1)-G(N^0_t) \right)
		\sum_{x=1}^{\infty}
		x\,\eta_t(-x-1)
		\prod_{k=1}^{x}[1-\eta_t(-k)]
	\right).
	\end{equation*}
	In the sum over $x$ in the right-hand side 
	almost surely only one term is nonzero, and the whole sum is equal to
	the distance of the rightmost particle in $\mathbb{Z}_{<0}$
	to zero.
\end{corollary}
\begin{proof}
	The left-hand side is equal to $\mathbb{E}\left( A^T f(\eta_t) \right)$,
	which by \Cref{prop:invariance}
	is the same as $-t^{-1}\mathbb{E}(A^L f(\eta_t))$.
	The rest 
	follows from the computation of $A^L f(\eta_t)$ for the particular function
	\eqref{eq:F_of_height},
	which is straightforward.
\end{proof}

\subsection{Hydrodynamics}

The hydrodynamic limit for the TASEP is well known, with early
results by \cite{Liggett1975} on the convergence
to a local equilibrium and by \cite{Rost1981}
on the connection of the density function to the Burgers'
equation. 
The latter means that 
under linear space and time scaling, 
the limiting density 
density
function $\rho(t,z)$
of the TASEP
is the entropic solution of the following
initial-value problem for the one-dimensional Burgers' equation:
\begin{equation}\label{eq:burgers}
	\begin{split}
	\frac{\partial \rho}{\partial t} 
	&=
	- \frac{\partial [\rho (1-\rho)]}{\partial z}\,;\\
	\rho( 0, z)
	&= 
	\begin{cases} 1, \quad z \leq 0 \\ 0;\quad z > 0. \end{cases}
	\end{split}
\end{equation}
We refer to \cite{BenassiFouque1987} for further details,
see also 
\cite{ferrari_PA2018tasep} for a recent review.
The solution to \eqref{eq:burgers}
is given by
\begin{equation}\
	\label{eq:rho_TASEP_limit_shape}
	\rho(t, z)
	=
	\begin{cases} 
		1, &z < -t; 
		\\
		(t-z)/2t, & -t \leq z \leq t; 
		\\
		0, & z> t.
	\end{cases}
\end{equation}
The limiting density $\rho(t,z)$ describes the law of
large numbers type behavior of the TASEP.

\begin{remark}[Asymptotic analysis of TASEP]
	\label{rmk:finer_scaling}
	More recently, in the last 20 years, much finer scaling limits for the
	TASEP have become available, beginning with the work of Johansson
	\cite{johansson2000shape} on the Tracy-Widom fluctuations of
	the position of the particles in the TASEP. 
	More generally, the TASEP with 
	various other 
	examples of initial data
	has been shown to converge to the 
	top lines of the 
	$\text{Airy}_1$ or
	$\text{Airy}_2$ 
	line ensembles
	under the appropriate scalings, 
	see, e.g., the
	survey \cite{Ferrari_Airy_Survey} and references therein for
	details.
	The progress in understanding the 
	TASEP asymptotics
	with general initial data,
	and also the asymptotics of the space-time structure in 
	TASEP is currently ongoing 
	\cite{ferrari2016time},
	\cite{chhita2018limit},
	\cite{matetski2017kpz},
	\cite{johansson2018two},
	\cite{baik2019multipoint},
	\cite{ferrari2019time},
	\cite{basu2018time},
	\cite{directed_landscape},
	\cite{Johansson_two_time_shorter},
	\cite{basu_ganguly_hammond_2019},
	\cite{JohanssonRahman2019}.
	
	While we expect the BHP and 
	the stationary 
	dynamics from \Cref{sub:equil_dyn_defninition}
	to have applications for all these
	types of scaling limits, we begin by considering 
	the hydrodynamic limit
	of the BHP in this section.
\end{remark}

Let $\eta_t \in \mathcal{C}$ be the random configuration
at time $t \geq0$ of the TASEP with step initial conditions. 
For any $\epsilon>0$, the \emph{($\epsilon$-scaled) random empirical measure} on $\mathbb{R}$ 
associated to $\eta_t \in \mathcal{C}$ is given as follows:
\begin{equation}
	\label{eq:random_empirical_measures}
	\pi^{\epsilon}_t := \epsilon \sum_{x \in \mathbb{Z}} \eta_{t}(x)\, \delta_{\epsilon x}.
\end{equation}
In particular, we have scaled the mass of each point by $\epsilon$, 
scaled the lattice distance by $\epsilon$, but the time remains unscaled. 
Denote the set of compactly supported continuous
functions on the line by
$C_0(\mathbb{R})$. The integral of a
function $f \in C_0(\mathbb{R})$ against the measure $\pi^{\epsilon}$
is denoted by $\langle \pi^{\epsilon} , f\rangle$.
Clearly, 
$\langle \pi_t^\epsilon , f\rangle = \epsilon\sum_{x \in \mathbb{Z}} f(\epsilon x)\, \eta_{t} (x)$.

The next statement can be found in, e.g.,
\cite{seppalainen1999existence}. 
The sequence of measures $\{ \pi_{t/\epsilon}^{\epsilon}\}_{\epsilon \in
\mathbb{R}_{>0} }$ converges as $\epsilon\to0$ in probability to $\rho(t, z) dz$ so that
the density function $\rho(t,z)$ is the entropic solution of the
initial value problem for the Burgers equation \eqref{eq:burgers}.
That is, for each $t\ge0$, given any $\delta >0$,
\begin{equation}
	\label{eq:convergence_of_random_empirical_measures}
	\lim_{\epsilon \rightarrow 0}\, \mathrm{Prob} \left(
	\Bigl|\epsilon \sum_{x \in \mathbb{Z}}\, f(\epsilon x) \eta_{t/\epsilon} (x)  - \int_{-\infty}^{\infty} f(z) \rho( t, z) dz \Bigr| \geq \delta \right) =0
\end{equation}
for any $f \in C_0(\mathbb{R})$. Note that now we have scaled time by $\epsilon^{-1}$ in the empirical measure.

This result for TASEP generalizes to
a large class of initial conditions.
For instance, given a continuous
density profile $\rho_0: \mathbb{R} \rightarrow [0,1]$,
a sequence $\{\nu^{\epsilon} \}_{\epsilon  \in \mathbb{R}_{>0}}$
of probability
measures on $\overline{\mathcal{C}} = \{ 0, 1\}^{\mathbb{Z}}$
is said to be \emph{associated to the profile $\rho_0$} 
if for every
$f\in C_0(\mathbb{R})$ 
and every $\delta >0$, we have
\begin{equation*}
	\lim_{\epsilon \rightarrow 0} \,
	\nu^{\epsilon} 
	\left[ 
		\Bigl|\epsilon 
		\sum_{x \in \mathbb{Z}} 
		f(\epsilon x)\, \eta(x)
		-
		\int_{- \infty}^{\infty} f(w)\, \rho_0(w) dw \Bigr| 
		> \delta
	\right] 
	= 0.
\end{equation*}
Then, the empirical measure $\pi_{t/\epsilon}^{\epsilon}$ for the TASEP, with
initial conditions now given by $\nu^{\epsilon}$ converges in
probability to an absolutely continuous measure $\rho( t, z) dz$ so
that the density function is the entropic solution to the Burgers'
equation with the initial value given by the density profile
$\rho(t,z)$, see \cite{seppalainen1999existence}. We expect a similar
hydrodynamic result to hold for the BHP with some modifications: (1) a different PDE
arising from the infinitesimal generator of the BHP, and (2) no time scaling for the empirical measure since lattice scaling also scales the particle numbers and, consequently, the speed of the particles.

\begin{conjecture}\label{conj:hydrodynamics}
	Let $\rho_0: \mathbb{R} \rightarrow [0,1]$ be an initial density
	profile and let $\{ \nu^{\epsilon}\}_{\epsilon \in \mathbb{R}_{>0}}$
	be a sequence of probability measures on 
	$\mathcal{C}$
	associated to
	$\rho_0$.\footnote{We need to make sure that
	the BHP evolution is well-defined, so the initial 
	configuration must be in 
	$\mathcal{C}\subset\overline{\mathcal{C}}$.} 
	Also, for a fixed $\epsilon>0$, take
	$\eta_t^{\epsilon} \in \mathcal{C}$ to be the random configuration
	at time $t>0$ 
	of the BHP,
	with the initial configuration 
	$\eta_0^{\epsilon}$ determined by the measure $\nu^{\epsilon}$.
	Then, for every $t>0$, the sequence of random empirical
	measures $\pi^{\epsilon}_t$ defined as in 
	\eqref{eq:random_empirical_measures}
	converges in probability to the absolutely continuous measure
	$\pi_t(dz) = \rho(z, t) dz$ 
	in the sense of \eqref{eq:convergence_of_random_empirical_measures}.
	The density $\rho(t,z)$
	is a solution of the initial
	value problem
	\begin{equation}\label{eq:BHP_IVP}
		\begin{split}
			\frac{\partial \rho(t,z) }{\partial t} 
			&= 
			\frac{\partial}{ \partial z} 
			\left[ 
			\frac{1 - \rho( t,z)}{\rho(t,z)} 
			\int_z^{\infty} \rho(t, w) dw 
			\right];
			\\
		\rho(0,z) &= \rho_0(z).
		\end{split}
	\end{equation}
\end{conjecture}

\begin{remark}
	In \Cref{conj:hydrodynamics}, it is unclear to the authors
	if there is a unique solution to the initial value problem
	\eqref{eq:BHP_IVP}. In particular, it is unclear what type of
	solution the limiting density profile $\rho(t,z)$ should be.
\end{remark}

\begin{remark}
	The differential equation \eqref{eq:BHP_IVP}
	can be informally obtained 
	by looking at the local version of the BHP.
	That is, locally we expect the configuration
	to be close to the independent Bernoulli
	random configuration on the whole line $\mathbb{Z}$
	with the density $\rho(t,z)$.
	Then the expression under $\partial/\partial z$
	in the right-hand side of 
	\eqref{eq:BHP_IVP} is 
	the (negative) flux. 
	Indeed, $\int_z^{\infty} \rho(t, w) dw$
	means the inhomogeneous rate in the BHP, 
	while $-(1-\rho(t,z))/\rho(t,z)$ is the 
	local flux of the homogeneous BHP 
	with left jumps and speed $1$.
	See 
	\Cref{prop:martingales_in_BHP}
	below for more discussion.
\end{remark}

Let us check that \Cref{conj:hydrodynamics}
holds for the initial data associated with the TASEP
distributions $\upmu_t$.

\begin{proposition}
	Fix some $t_0 \in \mathbb{R}$
	and let $\eta_0^{\epsilon} \sim
	\upmu_{\epsilon^{-1} e^{t_0}}$ be the TASEP random 
	configuration 
	at time $\epsilon^{-1} e^{t_0}$. 
	Then,
	the sequence $\{ \eta_0^{\epsilon}\}_{\epsilon \in \mathbb{R}_{>0}}$
	is associated to the density profile 
	\begin{equation*}
		\rho_0 (z)
		=
		\begin{cases} 
			1, & z < e^{ t_0};
			\\ 
			\frac{e^{ t_0}-z}{2e^{ t_0}},&
			 -e^{ t_0} \leq z \leq e^{ t_0} ;
			\\ 
			0,& z> e^{t_0},
		\end{cases}
	\end{equation*}
	and \Cref{conj:hydrodynamics} is true for the 
	measures $\nu^{\epsilon}=\upmu_{\epsilon^{-1}e^{t_0}}$.
\end{proposition}

\begin{proof}
	By results for the TASEP, we know that the sequence
	$\eta_0^{\epsilon}$ is associated to the density profile $\rho_0$ given
	in the statement. 
	Also, by \Cref{thm:intro_action_on_TASEP}, we know
	that the random configuration $\eta_{t}^{\epsilon}$
	obtained from $\nu^\epsilon=\upmu_{\epsilon^{-1}e^{t_0}}$ 
	by the BHP evolution as in \Cref{conj:hydrodynamics},
	is distributed according to
	$\upmu_{\epsilon^{-1} e^{t_0 - t}}$. 

	So, again by results for the
	TASEP, we know that the sequence of random measures
	$\pi_t^{\epsilon}$ converges to an absolutely continuous measure
	$\pi_t(dz) = \rho(z, t) dz$ with the density given by 
	\begin{equation*}
		\rho (t,z) 
		=
		\begin{cases} 
			1, & z < e^{ t_0 -t};
			\\
			\frac{e^{ t_0-t}-z}{2e^{ t_0-t}};
			&
			-e^{ t_0-t} \leq z \leq e^{ t_0-t} 
			\\
			0, & z> e^{ t_0-t}.
		\end{cases}
	\end{equation*}
	One can then check directly that the above 
	$\rho(t, z)$ solves the initial
	value problem \eqref{eq:BHP_IVP}.
	This completes the proof.
\end{proof}

We base \Cref{conj:hydrodynamics} on the random evolution
of the empirical measure $\pi_t^{\epsilon}$ given by the
infinitesimal generator for the BHP.

\begin{proposition}
	\label{prop:martingales_in_BHP}
	Let $f: \mathbb{R} \rightarrow \mathbb{R}$ be a twice differentiable
	compactly supported function and let $\eta_t \in
	\mathcal{C}$ be the random configuration given by the BHP. Here the time $t\ge0$ 
	and the initial configuration 
	$\eta_0\in \mathcal{C}$ are fixed.
	Then, there
	are martingales $M_t^{\epsilon, f}$ with respect to the natural
	filtration $\sigma(\eta_s^{\epsilon},  s \leq t)$ so that  
	\begin{equation*}
		\langle \pi_t^{\epsilon} , f \rangle
		=
		\langle \pi_0^{\epsilon} , f\rangle 
		+
		\int_0^t \langle \pi_s^\epsilon, g^{\epsilon}  f' \rangle ds
		+
		M_t^{\epsilon, f} + \mathcal{O}(\epsilon^{2}),
	\end{equation*}
	for $\pi_t^{\epsilon}$ the random empirical measure of $\eta_t$ and the function 
	\begin{equation*}
		g^{\epsilon}(x)
		:=
		- \Biggl( \sum_{y=\lfloor \epsilon^{-1}x \rfloor }^{\infty} 
			\epsilon\, \eta_s(y)
			\Biggr) 
			\Biggl(
			\sum_{m=1}^{\infty} 
			m 
			\prod_{k=1}^{m} 
			\bigl(
				1- \eta_s(\lfloor \epsilon^{-1}x \rfloor- k)
			\bigr) 
			\Biggr).
	\end{equation*}
\end{proposition}

\begin{proof}
	We have 
	\begin{equation*}
		\frac{\partial}{ \partial t} 
		\,
		\mathbb{E}\,
		\langle \pi_t^{\epsilon}, f\rangle 
		=
		\mathbb{E}\,
		A^L
		\langle \pi_t^{\epsilon}, f\rangle,
	\end{equation*}
	where we regard 
	$
	\langle \pi_t^{\epsilon}, f\rangle
	$
	as a function of the configuration
	$\eta_t$.
	We can compute
	\begin{equation*}
	\begin{split}
		A^L\langle \pi_t^{\epsilon}, f\rangle
		&= \sum_{x \in \mathbb{Z}} \left[\sum_{m=1}^{\infty}
		\left(\frac{f(\epsilon x -\epsilon m) - f(\epsilon x)}{\epsilon m}
		\right) m \prod_{k=1}^{m}(1- \eta(x-k)) \right]
		\left(\sum_{y=x}^{\infty} \epsilon \eta(y) \right) \eta(x).
	\end{split}
	\end{equation*}
	With the help of the approximation 
	\begin{equation*}
		f(\epsilon x -\epsilon m) = f(\epsilon x) -f'(\epsilon x) (\epsilon m) + \mathcal{O}(\epsilon^2),
	\end{equation*}
	the statement follows from standard results on Markov chains.
\end{proof}

\subsection{Limit shape for TASEP with step initial condition}

Let us present an alternate derivation for the limit shape of the TASEP
with the step initial configuration
assuming \Cref{conj:hydrodynamics} but independent of the
similar result for the TASEP. We only assume that the TASEP
empirical measure
converges to $\rho$ satisfying the following system of
equations:
\begin{equation}\label{eq:hydrodynamic_system}
	\begin{split}
	\frac{\partial \rho(t,z)}{\partial t}& + \frac{\partial}{\partial z} \left[\frac{1- \rho(t,z)}{t \rho(t, z)} \int_{z}^{\infty}\rho(t,w) dw\right] = 0;
	\\
\frac{\partial \rho(t,z)}{\partial t} & + \frac{\partial}{\partial z} [\rho(t,z)(1- \rho(t,z))] =0.
	\end{split}
	\end{equation}
In particular, we show that this system of partial differential equations determines a unique solution under some general assumptions.
	
First, eliminate the time derivative so that 
\begin{equation*}
	\frac{\partial}{ \partial z} \left[\frac{1- \rho(t,z)}{\rho(t,z)} \left( \rho(t,z)^2 - \frac{1}{t} \int_{z}^{\infty} \rho(t,w)dw\right) \right] =0.
\end{equation*}
Then,
	\begin{equation*}
		\frac{1- \rho(t,z)}{\rho(t,z)} \left( \rho(t,z)^2 - \frac{1}{t} \int_{z}^{\infty} \rho(t,w)dw\right)  = c(t).
	\end{equation*}
	Note that, for all $t \in \mathbb{R}_{\geq 0}$, there is a $z \in \mathbb{Z}$ small enough so that $\rho(t, z) = 1$.
	This implies that the constant $c(t)$ is in fact zero.
	Thus, we have
	\begin{equation*}
	\rho^2(t,z) = \frac{1}{t} \int_{z}^{\infty} \rho(t,w) d w.
	\end{equation*}
Taking the space derivative, we have
	\begin{equation}\label{eq:hydro_partial_space}
	\frac{\partial \rho(t,z) }{\partial z}  = - \frac{1}{2 t}.
	\end{equation}
Revisiting the system of equations \eqref{eq:hydrodynamic_system}, we may now write the second equation as follows
	\begin{equation*}
	\frac{\partial \rho(t,z)}{ \partial t}  = \frac{1}{2 t} (1 - 2 \rho(t,z)).
	\end{equation*}
By separation of variables, we may solve the equation above up to a constant of integration, but this constant of integration may be determined by 
\eqref{eq:hydro_partial_space}. Thus, we get the well-known 
hydrodynamic density function
	\begin{equation*}
		\rho(t,z) = \frac{1}{2} - \frac{z}{2 t},
		\qquad z\in[-t,t].
	\end{equation*}
	
We have thus shown that the comparability of the TASEP and the BHP
uniquely picks out the entropic solution to the Burgers' equation for
the limiting density function with the step initial condition.

\section{Extensions and open questions}
\label{sec:extensions}

In this section we describe a number of 
modifications and 
extensions of the 
constructions presented earlier, and 
outline a number of open questions.

\subsection{More general $q$-Gibbs measures}
\label{sub:general_schur}

The Markov map $\mathbb{L}^{(q)}$ 
from \Cref{def:qGibbs_map}
acts nicely on Schur processes 
$\mathbb{P}[(1,q,q^2,\ldots )\mid \rho]$
with general specializations $\rho$.
Even more generally, we can consider 
\emph{two-sided} Schur processes which live 
on interlacing arrays of signatures.
Signatures are analogues of partitions in which parts are allowed
to be negative. Interlacing arrays 
of signatures are simply the 
collections $\{\lambda^{(k)}_j\}_{1\le j\le k}$, 
satisfying the interlacing inequalities as in 
\Cref{fig:array},
and with $\lambda^{(k)}_j\in \mathbb{Z}$.
(Note that we consider arrays of infinite depth.)

For a specialization $\rho$ parametrized as
\begin{equation}
	\label{eq:alpha_pm_specialization}
	\begin{split}
		&
		\rho=(\alpha^{\pm};\beta^{\pm};\gamma^{\pm}),
		\quad 
		\alpha^{\pm}_1\ge \alpha^{\pm}_2\ge \ldots\ge0 ,
		\quad 
		\beta^{\pm}_1\ge \beta^{\pm}_2\ge \ldots\ge0 ,
		\quad 
		\gamma^{\pm}\ge 0,
		\\
		&\hspace{130pt}
		\sum_{i=1}^\infty(\alpha_{i}^{\pm}+\beta_i^{\pm})<\infty,
		\quad 
		\beta_1^++\beta_1^-\le 1,
	\end{split}
\end{equation}
and a signature $\lambda=(\lambda_1\ge \ldots\ge \lambda_N )$, 
$\lambda_i\in \mathbb{Z}$, define 
\begin{equation}
	\label{eq:Jacobi_Trudi}
	s_\lambda(\rho)
	:=
	\det\left[ 
		\psi_{\lambda_i+j-i}(\rho) 
	\right]_{i,j=1}^{N},
\end{equation}
where $\psi_n(\rho)$, $n\in \mathbb{Z}$, are the coefficients 
of the expansion 
\begin{equation}
	\label{eq:Voiculescu_functions}
	\sum_{n\in \mathbb{Z}}\psi_n(\rho) u^n=
	e^{\gamma^+(u-1)+\gamma^{-}(u^{-1}-1)}
	\prod_{i\ge 1}
	\frac{1+\beta_i^+(u-1)}{1-\alpha_i^+(u-1)}
	\frac{1+\beta_i^-(u^{-1}-1)}{1-\alpha_i^-(u^{-1}-1)},
\end{equation}
and $|u|=1$.
One of the equivalent forms of
the Edrei--Voiculescu theorem 
(e.g., see \cite{BorodinOlsh2011GT})
states that \eqref{eq:Voiculescu_functions}
parametrizes the space of all
totally nonnegative two-sided sequences.

\begin{remark}
	In particular, taking $\alpha_i^\pm=\beta_i^\pm=0$ for all $i$,
	$\gamma^-=0$, and $\gamma^+=t$ turns the just defined specialization
	$\rho$ into $\rho_t$ defined in \Cref{sub:specializations}
\end{remark}

Define the two-sided ascending Schur process
$\mathbb{P}[(1,q,q^2,\ldots )\mid \rho]$
as the unique $q$-Gibbs measure on
interlacing arrays of signatures 
such that for any $N$, 
\begin{equation}
	\label{eq:two_sided_q_Schur}
	\mathbb{P}[(1,q,q^2,\ldots )\mid \rho]
	(\lambda^{(1)},\ldots,\lambda^{(N)} )=
	\frac{1}{Z}\,
	s_{\lambda^{(1)}}(1)
	s_{\lambda^{(2)}/\lambda^{(1)}}(q)
	\ldots
	s_{\lambda^{(N)}/\lambda^{(N-1)}}(q^{N-1})
	\,
	s_{\lambda^{(N)}}(\rho),
\end{equation}
where the skew Schur functions for signatures can be defined 
by \eqref{eq:skew_one_var}.
Define the Schur process
$\mathbb{P}[\vec 1\mid \rho]$ as the $q\to 1$
degeneration of \eqref{eq:two_sided_q_Schur} (here and below
we denote by $\vec 1$ the sequence of spectral
parameters which are all equal to $1$).
Another equivalent form of the Edrei--Voiculescu theorem
states that $\mathbb{P}[\vec 1\mid \rho]$
are all possible extreme Gibbs measures on interlacing arrays of signatures 
(a Gibbs measure is called extreme if it cannot be represented as a convex combination of 
other Gibbs measures).
We refer to 
\cite{Borodin2010Schur} for further details on
the definition of the two-sided Schur processes.

\begin{theorem}
\label{thm:action_on_q_Gibbs_Schur}
	Let $\rho$ be a specialization 
	with parameters
	\eqref{eq:alpha_pm_specialization}
	such that $\alpha_i^{-}=0$ for all $i$.
	Then we have for all $0<q<1$:
	\begin{equation*}
		\mathbb{P}[(1,q,q^2,\ldots )\mid \rho]
		\,
		\mathbb{L}^{(q)}
		=
		\mathbb{P}[(1,q,q^2,\ldots )\mid \rho^{(q)}],
	\end{equation*}
	where $\rho^{(q)}$ is the specialization
	corresponding to the parameters
	\begin{equation}
		\label{eq:q_modif_params}
		\begin{split}
			&
			\hat\alpha_i^+=\frac{\alpha_i^+ q}{1+\alpha_i^+-\alpha_i^+ q}
			,\qquad 
			\hat\alpha_i^-=0
			\\&
			\hat \beta_i^+=\frac{\beta_i^+ q}{1-\beta_i^++\beta_i^+ q}
			,\qquad 
			\hat\beta^-_i=
			\frac{\beta_i^- q^{-1}}{1-\beta_i^-+\beta_i^-q^{-1}}
			\\&
			\hat\gamma^+=q\gamma^+
			,\qquad 
			\hat\gamma^-=q^{-1}\gamma^{-}.
		\end{split}
	\end{equation}
\end{theorem}
Note that $\hat \alpha_i^+,\hat \beta_i^{\pm}\ge0$,
and $\hat \beta_1^++\hat \beta_1^-\le 1$.
\begin{proof}[Proof of \Cref{thm:action_on_q_Gibbs_Schur}]
	This follows from \Cref{thm:main_prop_qGibbs}
	similarly to the computation in the beginning of \Cref{sub:q_schur_TASEP}.
	Namely, denote \eqref{eq:Voiculescu_functions}
	by
	$\Psi(u;\rho)$.
	The $q$-Gibbs measure
	\eqref{eq:two_sided_q_Schur} 
	corresponds to the $q$-Gibbs harmonic family
	\begin{equation*}
		\varphi_N(\lambda)=
		\frac{s_\lambda(\rho)}
		{\Psi(1;\rho)\Psi(q;\rho)\Psi(q^2;\rho)\ldots\Psi(q^{N-1};\rho)}.
	\end{equation*}
	Note that $\Psi(1;\rho)=1$ but it is convenient to 
	include this factor here. 
	Note also that the condition $\alpha_i^{-}\equiv 0$ ensures that the 
	series $\Psi(q^m;\rho)$ converge for all $m\in \mathbb{Z}_{\ge1}$.
	The action of $\mathbb{L}^{(q)}$
	turns the $q$-Gibbs harmonic family $\{\varphi_N\}$ into
	\begin{equation*}
		\hat \varphi_N(\lambda)=
		\frac{q^{|\lambda|}s_\lambda(\rho)}
		{\Psi(q;\rho)\Psi(q^2;\rho)\ldots\Psi(q^{N};\rho)}
		=
		\frac{\det[\psi_{\lambda_i+j-i}(\rho)\,q^{\lambda_i+j-i}]_{i,j=1}^{N}}
		{\Psi(q;\rho)\Psi(q^2;\rho)\ldots\Psi(q^{N};\rho)}
		.
	\end{equation*}
	In particular, for $N=1$ we have
	\begin{equation*}
		\psi_n(\rho^{(q)})=\frac{\psi_n(\rho)\,q^{n}}{\Psi(q;\rho)},
		\qquad n\in \mathbb{Z},
	\end{equation*}
	which readily translates into
	the modification of the parameters \eqref{eq:q_modif_params} in the claim.
\end{proof}

Measures on interlacing arrays
of the form
$\mathbb{P}[(1,q,q^2,\ldots )\mid \rho]$
are not extreme $q$-Gibbs.
A classification of extreme $q$-Gibbs measures
is obtained in \cite{Gorin2010q} (note that our
$q$ corresponds to $1/q$ in that paper, so the description
of the boundary needs to be reversed).
Extreme $q$-Gibbs measures $\mathbb{P}_{\mathbf{n}}^{(q)}$
are parametrized by infinite sequences
\begin{equation*}
	\mathbf{n}=(n_1\ge n_2\ge \ldots ), \qquad  n_i\in \mathbb{Z}.
\end{equation*}
Moreover, $\lim_{N\to+\infty}\lambda^{(N)}_j=n_j$ for each 
fixed $j=1,2,\ldots$, where
$\lambda^{(N)}_j$ come from the random configuration 
distributed according to $\mathbb{P}_{\mathbf{n}}^{(q)}$.
It is not hard to show the following.
\begin{proposition}
	\label{prop:action_L_on_n}
	The action of the Markov map $\mathbb{L}^{(q)}$
	on extreme $q$-Gibbs measures corresponds to 
	the left shift in the space of parameters:
	\begin{equation*}
		\mathbb{P}^{(q)}_{(n_1,n_2,n_3,\ldots )}
		\,
		\mathbb{L}^{(q)}
		=
		\mathbb{P}^{(q)}_{(n_2,n_3,n_4,\ldots )}.
	\end{equation*}
\end{proposition}

In \cite{BG2011non} a decomposition
of the non-extreme $q$-Gibbs
measures
$\mathbb{P}[(1,q,q^2,\ldots )\mid \rho]$
onto the extreme ones $\mathbb{P}^{(q)}_{\mathbf{n}}$ 
is given in terms of a determinantal point
process on the set of shifted
labels. The shifted labels
in our notation are
$n_1-1>n_2-2>\ldots $, and they form a random 
point configuration on $\mathbb{Z}$ whose correlation
functions have a determinantal form. The action of
$\mathbb{L}^{(q)}$ 
on $\mathbf{n}$
from \Cref{prop:action_L_on_n}
removes the largest 
point in this determinantal process on $\mathbb{Z}$, and 
shifts all its other points 
by one to the right.
\begin{question}
	How to explicitly link
	the action of $\mathbb{L}^{(q)}$
	on $\mathbf{n}$
	with the modification of the parameters
	\eqref{eq:q_modif_params}
	of the 
	determinantal point process 
	describing 
	$\mathbb{P}[(1,q,q^2,\ldots )\mid \rho]$?
	Does this correspondence
	(between the action on the parameters of the kernel
	and the action on the underlying 
	random point configuration)
	survive any limit transition 
	to more familiar determinantal point processes
	(e.g., random matrix spectra or Airy$_2$)?
\end{question}

\subsection{Limit \texorpdfstring{$q\to1$}{q->1} and action on Gibbs measures}
\label{sub:q1_arrays}

The $q\to1$ limit of \Cref{thm:action_on_q_Gibbs_Schur}
can be obtained similarly to the argument in \Cref{sec:limit_q_1}.
Define by $\mathbb{L}_{\tau}$, $\tau\in \mathbb{R}_{\ge0}$, 
the continuous-time
Markov semigroup under which each 
particle $\lambda^{(k)}_j$ at each $k$-th level of
the interlacing array
independently jumps to the left into one of the possible locations $m$,
where
\begin{equation*}
	\max\{\lambda^{(k+1)}_{j+1},\lambda^{(k-1)}_j\}\le m\le \lambda^{(k)}_j-1,
\end{equation*}
at rate $k$ per each of these possible locations. 

However, this definition presents an issue since
in a generic interlacing array, under $\mathbb{L}_\tau$
infinitely many particles jump in finite time.
Moreover, because for any $k\in \mathbb{Z}_{\ge1}$
jumps of $\lambda^{(k)}_j$ depend on the $(k+1)$-st level,
one cannot simply restrict $\mathbb{L}_\tau$ to 
the first several levels. 
Therefore, we have to consider a smaller space 
of interlacing arrays:

\begin{definition}
	\label{def:issue_with_many_jumps_ok}
	Let the subset $\mathcal{S}^c\subset\mathcal{S}$ 
	consist of interlacing arrays
	$\{\lambda^{(N)}_j\}_{1\le j\le N}$
	satisfying $\lambda^{(N)}_j=0$ for all $N$ and all $J(N)\le j\le N$, 
	where $N-J(N)\to+\infty$ as $N\to+\infty$.
\end{definition}

For each fixed $K$, the restriction of $\mathbb{L}_\tau$ to 
\begin{equation*}
	\{\lambda^{(N)}_j\colon N\in \mathbb{Z}_{\ge1}, \ N-K+1\le j\le N\}
\end{equation*}
(that is, to the $K$ leftmost diagonals)
is a Markov process, in which only finitely many particles 
jump in finite time.
For different $K$, these Markov processes are compatible.
Therefore, $\mathbb{L}_\tau$ makes sense on the state 
space $\mathcal{S}^c$.
Below we denote by $\mathbb{L}_{\tau}$ the Markov semigroup constructed in this manner.

\begin{theorem}\label{thm:L_action_q_to_1_gibbs_general}
	The action of the semigroup $\mathbb{L}_\tau$ on extreme
	Gibbs measures $\mathbb{P}[\vec 1\mid \rho]$,
	where $\rho$ is a specialization as in 
	\eqref{eq:alpha_pm_specialization}--\eqref{eq:Voiculescu_functions}
	with $\alpha_i^-=\beta_i^-=0$ for all $i$, $\gamma^-=0$, and $\beta_1^+<1$, transforms the parameters of 
	$\rho$ exactly as in \eqref{eq:q_modif_params},
	but with $q$ replaced by $e^{-\tau}$.
\end{theorem}
\begin{proof}[Idea of proof]
	One can check that the Schur process $\mathbb{P}[\vec 1\mid \rho]$ 
	with
	$\alpha_i^-=\beta_i^-=0$ for all $i$, $\gamma^-=0$, and $\beta_1^+<1$
	is supported on the subset $\mathcal{S}^c$ described
	in \Cref{def:issue_with_many_jumps_ok}.
	Similarly to \Cref{sec:limit_q_1}, we see that
	under the scaling
	$q=e^{-\varepsilon}$, 
	$T=\lfloor \tau/\varepsilon \rfloor$, $\varepsilon\to0$, 
	we have $(\mathbb{L}^{(q)})^T\to \mathbb{L}_{\tau}$.
	Next, the modification of the parameters
	\eqref{eq:q_modif_params}
	is a one-parameter semigroup. That is, applying $\mathbb{L}^{(q)}$
	one more time replaces $q$ everywhere in 
	\eqref{eq:q_modif_params} by $q^2$. Because
	$q^T\sim e^{-\tau}$, we get the result.
\end{proof}

In particular, $\mathbb{L}_{\tau}$ maps the 
push-block process of \cite{BorFerr2008DF} (see \Cref{def:kpz_growth}
below)
backwards in time in the same sense as \Cref{thm:intro_action_on_TASEP}.

\subsection{Iterated R maps}
\label{sub:iterated_R}

Consider the maps $R^{(j)}_\alpha$ defined in 
\Cref{sec:Markov_maps}. Similarly to
\Cref{sub:iterated_L_maps},
we can define the iterated R map $\mathbb{R}^{(q)}$ by
\begin{equation*}
	\mathbb{R}^{(q)}
	:=
	R^{(1)}_{q}
	R^{(2)}_{q^2}
	R^{(3)}_{q^3}
	\ldots 
\end{equation*}
(this definition has the same formal meaning as for the map $\mathbb{L}^{(q)}$, see
\Cref{sub:iterated_L_maps}).
The map $\mathbb{R}^{(q)}$ acts nicely on 
$q^{-1}$-Gibbs measures
(i.e., corresponding to $\vec c=(1,q^{-1},q^{-2},\ldots )$).
Namely, one can check that an analogue of 
\Cref{thm:main_prop_qGibbs}
holds, with $q$ replaced by $q^{-1}$
in the definition of the harmonic functions and in \eqref{eq:action_q_Gibbs_main_thm}.
The $q\to1$ continuous-time limit $\mathbb{R}_{\tau}$
of $\mathbb{R}^{(q)}$ is also readily defined 
with the help of \Cref{def:issue_with_many_jumps_ok}
--- this is just the mirroring 
of $\mathbb{L}_{\tau}$ from \Cref{sub:general_schur},
in which all particles jump to the right.
One can obtain the following analogue of \Cref{thm:action_on_q_Gibbs_Schur,thm:L_action_q_to_1_gibbs_general}
for the action of $\mathbb{R}^{(q)}$ and $\mathbb{R}_{\tau}$ on $q^{-1}$-Gibbs Schur processes:
\begin{theorem}
	\label{thm:R_action_general}
	Let $\rho$ be a specialization 
	as in 
	\eqref{eq:alpha_pm_specialization}--\eqref{eq:Voiculescu_functions}
	such that $\alpha_i^{+}=0$ for all $i$.
	We have
	for all $0<q<1$:
	\begin{equation*}
		\mathbb{P}[(1,q^{-1},q^{-2},\ldots )\mid \rho]
		\,
		\mathbb{R}^{(q)}
		=
		\mathbb{P}[(1,q^{-1},q^{-2},\ldots )\mid \rho^{(1/q)}],
	\end{equation*}
	where $\rho^{(1/q)}$ 
	has modified parameters as in \eqref{eq:q_modif_params},
	but with $q$ replaced
	by $1/q$.
	Moreover, if $\alpha_i^+=\beta_i^+=0$ for all $i$, $\gamma^+=0$, and $\beta_1^+<1$, 
	then 
	$\mathbb{P}[\vec 1\mid \rho]\,
	\mathbb{R}_{\tau}
	=
	\mathbb{P}[\vec 1\mid \rho^{(e^\tau)}]$,
	where $\rho^{(e^\tau)}$ is defined in a similar way.
\end{theorem}

\begin{question}
	\label{q:forward_map}
	Is it possible to extend the definition of 
	$\mathbb{R}_\tau$ to Schur processes with $\gamma^+>0$?
	(This is equivalent to extending $\mathbb{L}_\tau$ to 
	the case $\gamma^->0$.)
\end{question}

If such an extension is possible, then $\mathbb{R}_\tau$ would turn the 
time $t$ in the 
Schur process $\mathbb{P}[\vec 1\mid \rho_t]$
(with the Plancherel specialization $\gamma^+=t$ and all other parameters zero)
into $e^\tau t$,
that is, \emph{forward}.
Note that this process 
would move infinitely many particles in finite time
and move individual particles very far, too.

Recall that $\mathbb{P}[\vec 1\mid \rho_t]$
can be generated by the push-block dynamics
(\Cref{def:kpz_growth} below).
Under this dynamics, 
the rightmost components 
$\{\lambda^{(N)}_1\}$
of the interlacing array evolve as a PushTASEP, a
close relative of TASEP, but with a pushing mechanism
\cite{BorFerr08push}, \cite{BorFerr2008DF}.
Therefore, a positive answer to \Cref{q:forward_map}
would lead to a continuous-time semigroup
which maps PushTASEP forward in time.

\subsection{Arrays of finite depth}
\label{sub:c_Gibbs_finite}

Fix $N\in \mathbb{Z}_{\ge1}$
and let $\mathcal{S}^{\lambda,N}$ be the space
of interlacing arrays
$\lambda^{(1)}\prec \ldots\prec\lambda^{(N-1)}\prec\lambda^{(N)} $
with fixed top row $\lambda^{(N)}=\lambda$,
where $\lambda=(\lambda_1\ge \ldots\ge \lambda_N\ge0 )$, 
$\lambda_i\in \mathbb{Z}$.
Fix pairwise distinct spectral parameters 
$c_1,\ldots,c_N>0 $.

Recall the single level Markov maps $L^{(j)}_{\alpha}$,
$R^{(j)}_{\alpha}$, $j=1,\ldots,N-1$,
defined in 
\Cref{sec:Markov_maps}. 
Consider the product space $\widetilde{\mathcal{S}}^{\lambda,N}:=
\mathcal{S}^{\lambda,N}\times \mathfrak{S}_N$,
where $\mathfrak{S}_N$ is the symmetric group.
For each elementary permutation $s_i=(i,i+1)$, $1\le i\le N-1$, 
define the Markov map $T(s_i)$ on 
$\widetilde{\mathcal{S}}^{\lambda,N}$ 
as follows. On 
the $\mathfrak{S}_N$ part it 
deterministically acts by $\sigma\mapsto s_i \sigma$.
On each fiber $\mathcal{S}^{\lambda,N}\times \{\sigma\}$
it acts as the Markov map 
\begin{equation}
	\label{eq:L_R_choice_in_permutation}
	T(s_i):=
	\begin{cases}
		L^{(i)}_{c_{\sigma(i+1)}/c_{\sigma(i)}},& 
		\textnormal{if $c_{\sigma(i)}>c_{\sigma(i+1)}$};
		\\[7pt]
		R^{(i)}_{c_{\sigma(i)}/c_{\sigma(i+1)}},& 
		\textnormal{otherwise}.
	\end{cases}
\end{equation}
Note that $T(s_i)$ do not satisfy the symmetric group
relations when acting on 
$\widetilde{\mathcal{S}}^{\lambda,N}$
(in particular, $T(s_i)^2$ is not identity).

Let $\mathbb{M}^{\lambda}_{\vec c}$
denote the $\vec c$-Gibbs measure
on $\mathcal{S}^{\lambda,N}$:
\begin{equation*}
	\mathbb{M}^{\lambda}_{\vec c}(\lambda^{(1)},\ldots,\lambda^{(N-1)} )
	=
	\frac{s_{\lambda^{(1)}}(c_1)
	s_{\lambda^{(2)}/\lambda^{(1)}}(c_2)\ldots 
	s_{\lambda/\lambda^{(N-1)}}(c_N)}{s_{\lambda}(c_1,\ldots,c_{N} )}.
\end{equation*}
Note that in contrast with arrays of infinite
depth (cf. \Cref{sub:c_Gibbs_def}), here the $\vec c$-Gibbs property 
determines the measure $\mathbb{M}_{\vec c}^{\lambda}$ uniquely.

Let $w_N=(N,N-1,\ldots,2,1 )$ be 
the longest element in the symmetric group $\mathfrak{S}_N$,
and $w_N=s_{i_1}s_{i_2}\ldots s_{i_{N(N-1)/2}} $,
$1\le i_k\le N-1$,
be its reduced word decomposition which is also assumed fixed.
Define 
\begin{equation}
	\label{eq:T_map}
	\mathbb{T}:=
	T(s_{i_1}) 
	T(s_{i_2})
	\ldots 
	T(s_{i_{N(N-1)/2}})
\end{equation}
(in this notation we do not indicate the dependence 
on the choice of a particular reduced word).
Clearly, $\mathbb{T}^2$ acts as the identity on the 
$\mathfrak{S}_N$ part of 
$\widetilde{\mathcal{S}}^{\lambda,N}$.
Moreover, $\mathbb{T}^2$ preserves the measure
$\mathbb{M}_{\vec c}^{\lambda}$ viewed as the measure on 
$\mathcal{S}^{\lambda,N}\times\{ e \}$.
Indeed, this is because by \Cref{prop:main_lemma_level_general}
each $T(s_i)$ maps $\mathbb{M}_{\sigma \vec c}^{\lambda}$
to $\mathbb{M}_{s_i \sigma \vec c}^{\lambda}$. The map
$\mathbb{T}^2$ can be viewed as a sampling algorithm for the 
measure~$\mathbb{M}_{\vec c}^{\lambda}$:
\begin{proposition}
	\label{prop:sampling_algo}
	Start with any (nonrandom) interlacing array 
	$(\lambda^{(1)}\prec\ldots\prec\lambda^{(N-1)}\prec\lambda)$
	and apply the Markov map $\mathbb{T}^{2k}$ to it. 
	The distribution of the resulting
	random interlacing array 
	converges, as $k\to+\infty$, to
	$\mathbb{M}_{\vec c}^{\lambda}$ in the total variation norm.
\end{proposition}
\begin{proof}
	This follows from the standard convergence 
	theorem for Markov chains on finite spaces.
	Indeed, 
	the Markov chain corresponding to $\mathbb{T}^{2k}$ is 
	\begin{itemize}
		\item 
			aperiodic since $\mathbb{T}^2$ assigns positive probability
			to the trivial move;
		\item 
			irreducible
			because $\mathbb{T}^2$ assigns positive probability 
			to changing only one 
			entry $\lambda^{(k)}_j$, $1\le j\le k\le N-1$
			in the interlacing array 
			(the set $\mathcal{S}^{\lambda,N}$
			is connected by such individual changes).
	\end{itemize}
	This completes the proof.
\end{proof}

\begin{question}
	How fast is the convergence in \Cref{prop:sampling_algo},
	depending on the system size (which is $\sim N \lambda_1$)?
	What is the mixing time of $\mathbb{T}$?
\end{question}

\subsection{$q$-distributed lozenge tilings}
\label{sub:q_lozenge}

Let us now consider a concrete case of the 
setup outlined in 
the previous
\Cref{sub:c_Gibbs_finite}.
Fix $N$ and the top row $\lambda=(b,b,\ldots,b,0,0\ldots,0)$, where $b$ repeats $a$ times and $0$ repeats $c$ times, 
with $a+c=N$.
Then interlacing arrays
of depth $N$ and top row $\lambda$ are in bijection with 
lozenge tilings of a hexagon with sides $a,b,c,a,b,c$,
or, equivalently, with boxed plane partitions
(see \Cref{fig:tiling} for an illustration and, e.g., \cite{BorodinPetrov2013Lect} for more details).

\begin{figure}[htpb]
	\centering
	\includegraphics[width=.45\textwidth]{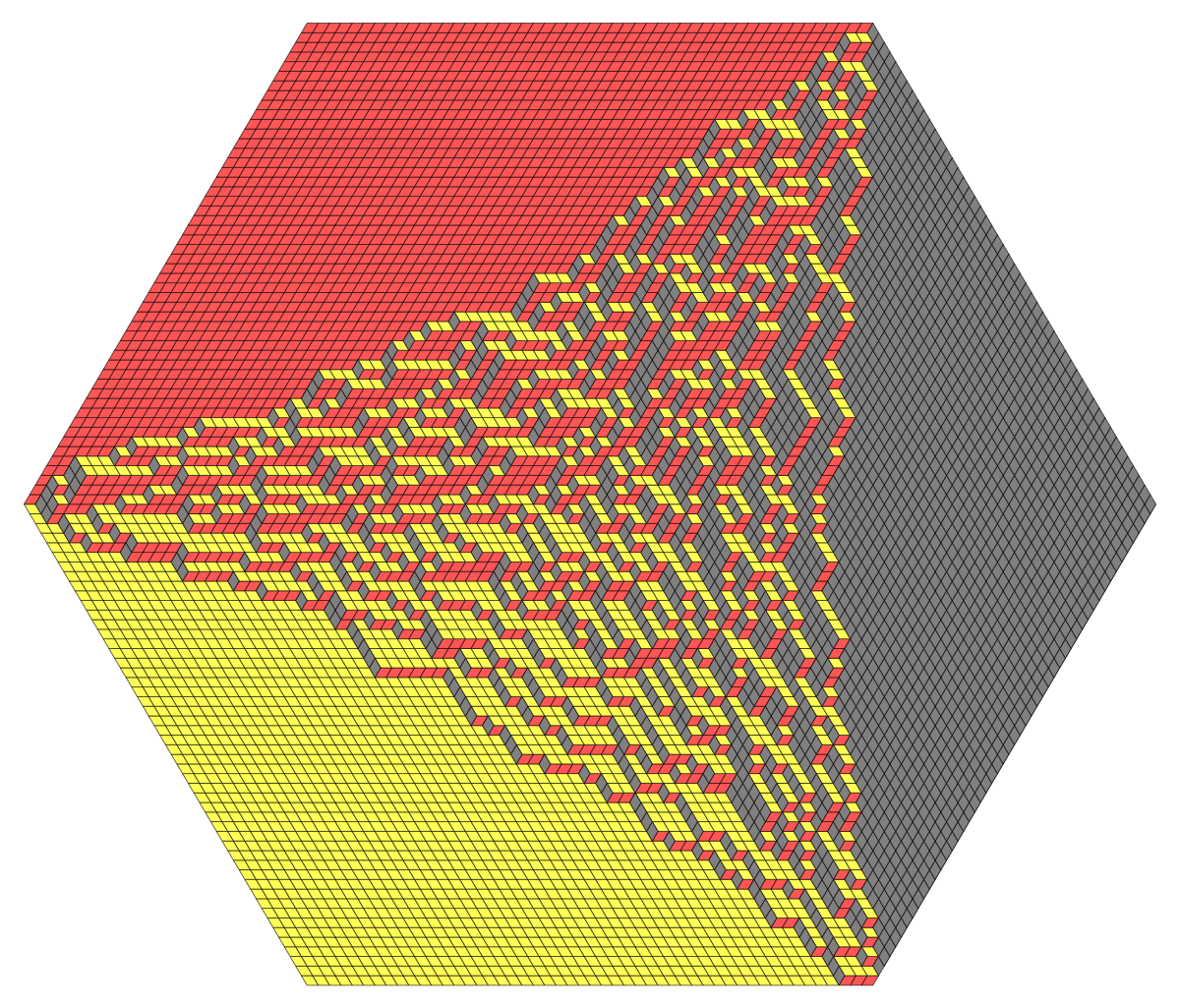}
	\hspace{.07\textwidth}
	\includegraphics[width=.45\textwidth]{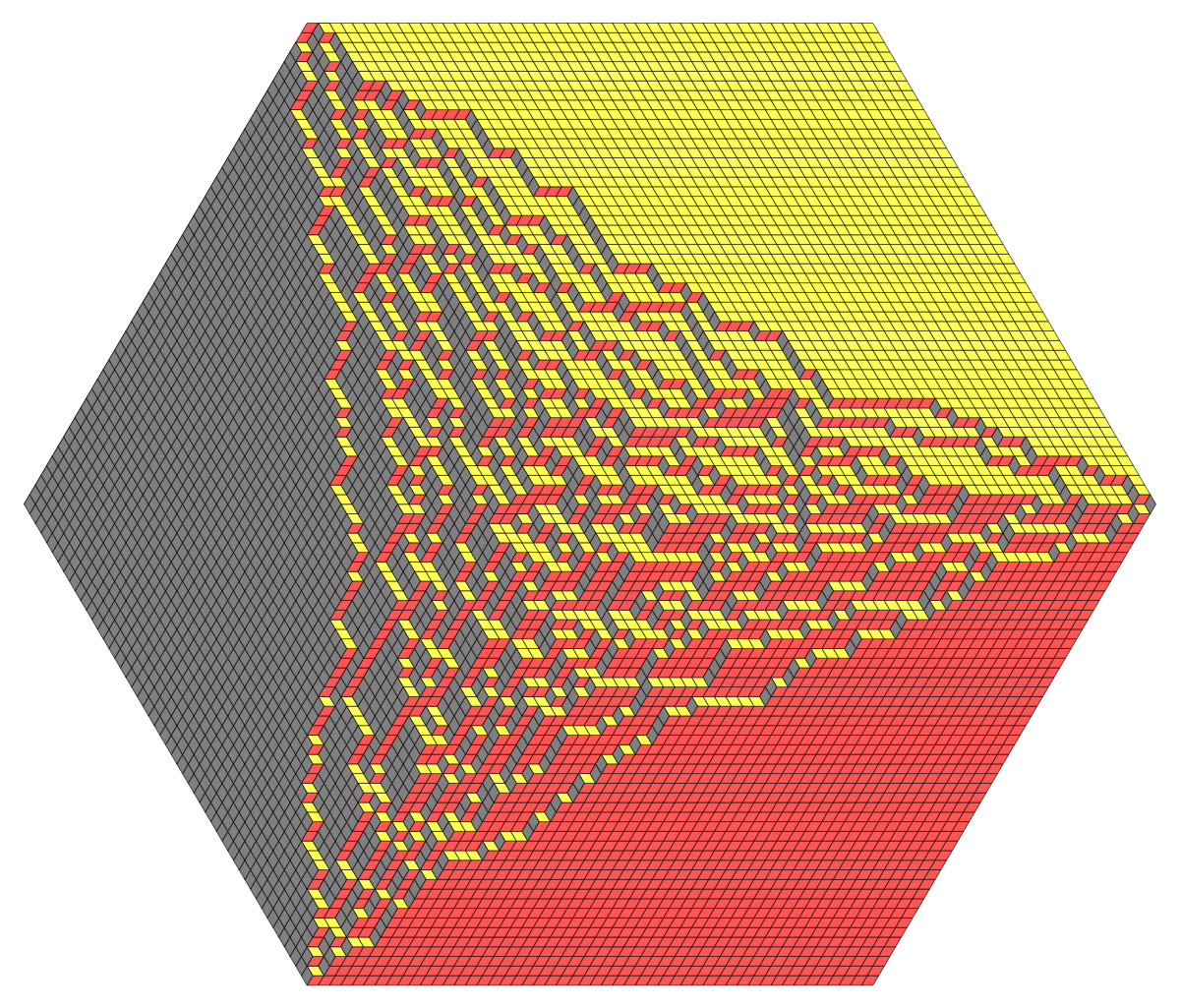}
	\caption{Samples from the measures $\mathbb{M}_{q^{-1}}$ (left) and
	$\mathbb{M}_{q}$ (right) with $a=b=c=50$ and $q=0.95$.
	The sample on the left is generated by the shuffling
	algorithm of \cite{borodin-gr2009q}, and the picture on 
	the right is the result of applying the map $\mathbb{T}$.}
	\label{fig:tiling}
\end{figure}

Let $\mathbb{M}_{q^{-1}}$ and $\mathbb{M}_{q}$ denote the 
measures under which the probability weight
of a lozenge tiling is proportional to 
$q^{-\mathrm{vol}}$ or $q^{\mathrm{vol}}$, respectively, 
where the volume is defined in \eqref{eq:volume}.
These two measures are 
$\vec c$-Gibbs with $\vec c=(1,q,q^2,\ldots,q^{N-1} )$ and 
$\vec c=(q^{N-1},\ldots,q,1 )$, respectively (recall that multiplying $\vec c$ by a scalar
does not change the $\vec c$-Gibbs property).

Take the reduced word 
\begin{equation*}
	w_N=
	(s_{1}s_2\ldots s_{N-1})
	(s_{1}s_2\ldots s_{N-2})
	\ldots 
	(s_1s_2)
	(s_1),
\end{equation*}
and let $\mathbb{T}$ be the corresponding Markov map \eqref{eq:T_map}.
One readily sees that the action of $\mathbb{T}$ on $\mathbb{M}_{q^{-1}}$:
\begin{itemize}
	\item 
		Turn $\mathbb{M}_{q^{-1}}$ into $\mathbb{M}_q$;
	\item 
		Almost surely moves vertical lozenges (see \Cref{fig:tiling})
		to the left because in \eqref{eq:L_R_choice_in_permutation}
		we always choose the option L;
	\item Changes the $(N-1)$-st row of the tiling only once, 
		the $(N-2)$-nd only twice, and so on.
\end{itemize}

An exact sampling algorithm for $\mathbb{M}_{q^{-1}}$
was presented in 
\cite{borodin-gr2009q}.
Starting with the exact sample of $\mathbb{M}_{q^{-1}}$ (\Cref{fig:tiling}, left)
and applying $\mathbb{T}$, 
we obtain an exact sample of $\mathbb{M}_{q}$ (\Cref{fig:tiling}, right),
while randomly moving the vertical lozenges to the left.
An implementation of this mapping 
$\mathbb{M}_{q^{-1}}\,\mathbb{T}=\mathbb{M}_q$ with all the intermediate steps
can be found online 
\cite{PetrovZhang2019simul}.

The map $\mathbb{T}$ works in the same way for an arbitrary
top row $\lambda$ (when the polygon being tiled
is not necessarily a hexagon, but can be a general sawtooth domain
as in, e.g., \cite{Petrov2012GT}).
The advantage of the hexagon case is the presence of the exact sampling
algorithm \cite{borodin-gr2009q}.

\begin{question}
	Consider lozenge tilings of growing sawtooth domains with top rows
	$\lambda=\lambda(N)$ which depend on $N$ in some way.
	Can the symmetry
	of the $q^{\pm\mathrm{vol}}$ measures
	manifested by the map $\mathbb{T}$
	be utilized to obtain the limit shape and fluctuations of the
	leftmost piece of the
	frozen boundary as $N\to+\infty$? 
	
	Here by the leftmost piece we mean the part of the curve separating
	the leftmost region occupied by only vertical lozenges,
	and the liquid region.
	Existence and characterization of limit shapes for $q^{\pm\mathrm{vol}}$ 
	is due to \cite{CohnKenyonPropp2000},
	\cite{OkounkovKenyon2007Limit},
	and some explicit formulas were obtained recently in
	\cite{DiFran2019qvol}.
\end{question}

\subsection{Dynamics in the bulk}
\label{sub:bulk_limits}

Consider the
Schur process $\mathbb{P}[\vec 1\mid \rho_t]$ (also sometimes known as the 
Plancherel measure for the infinite-dimensional unitary group).
It is convenient to use lozenge tiling interpretation
of interlacing arrays as in the previous \Cref{sub:q_lozenge}.
From 
\cite{BorodinKuan2007U},
\cite{BorFerr2008DF}
it is known that as $N$, $k$, and $t$ go to infinity proportionally to each other, 
the local lattice configuration 
of lozenges
around 
each
$\lambda^{(N)}_k$
converges to the ergodic translation invariant Gibbs measure
on lozenges tilings of the whole plane (see \Cref{fig:Gibbs_tilings} for an illustration).
Such ergodic measures form a two-parameter family \cite{Sheffield2008}.
As parameters one can take the densities of two of the three types of lozenges.
We remark that the ergodic Gibbs measures are far from being independent
Bernoulli ones. In particular, 
the joint correlations of lozenges possess a determinantal structure
\cite{okounkov2003correlation}.

\begin{figure}[htpb]
	\centering
	\includegraphics[width=.65\textwidth]{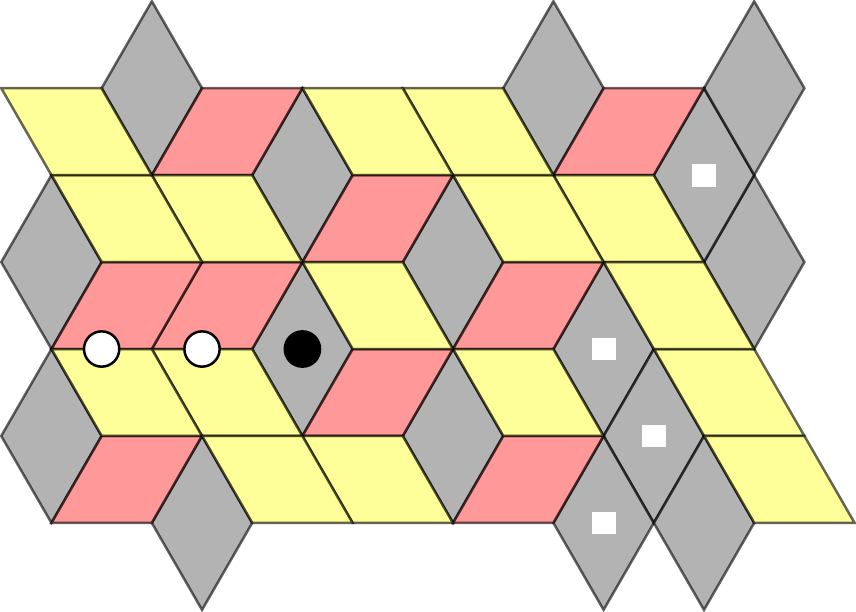}

	\caption{A lozenge configuration in the bulk and a possible move 
	under the bulk limit of $\mathbb{L}_{\tau}$: the vertical lozenge with a black dot
	can move to one of the white dotted locations, at rate
	$1$ per white dot. Square marks indicate lozenges 
	which are blocked in the 
	push-block dynamics.}
	\label{fig:Gibbs_tilings}
\end{figure}

We say that $(N,k,t)$ correspond to the \emph{bulk} of the system
if the limiting density of each of the types of lozenges around $\lambda^{(N)}_{k}(t)$
is positive.
One can also consider the bulk limit of the 
dynamics $\mathbb{L}_{\tau}$. 
Because $\mathbb{L}_{\tau}$ maps the Schur process
$\mathbb{P}[\vec 1\mid \rho_t]$ to
$\mathbb{P}[\vec 1\mid \rho_{e^{-\tau}t}]$
and $t\to+\infty$, we need to scale $\tau$ 
as 
$\tau=\uptau/t$
(here $\uptau\in \mathbb{R}_{>0}$ is the new scaled time which stays fixed).
Then $e^{-\uptau/t}\,t\sim \left( 1-\frac{\uptau}t \right)t=t-\uptau$.
Considering $\mathbb{L}_{\uptau/t}$ is equivalent to slowing down 
all the jump rates in $\mathbb{L}$ by the factor of $t$.
Since we are looking around level $N$ and $N$ grows proportionally
to $t$, the slowed down dynamics in the bulk
will have equal jump rates on all levels
at finite distance from the $N$-th one.

Therefore, under the bulk limit of $\mathbb{L}_{\tau}$, each vertical lozenge can move into 
one of the holes to the left of it (with the requirement that the interlacing is preserved),
at a constant rate \emph{per hole} (for simplicity, we can assume that this rate is equal to $1$).
See \Cref{fig:Gibbs_tilings} for an illustration.

Consider the combination of the dynamics
$\mathbb{L}_{\uptau l/t}$ and $\mathbb{R}_{\uptau r/t}$ 
running in parallel,\footnote{Here we are ignoring the issues with 
definitions of the continuous-time dynamics
outlined in \Cref{sub:q1_arrays,sub:iterated_R}.}
where $l,r>0$ are parameters.
In the bulk limit of this combination, 
we 
readily obtain the Hammersley-type process in the bulk with two-sided jumps.
This two-sided dynamics was introduced and studied in 
\cite{Toninelli2015-Gibbs},
where it was shown that this dynamics
preserves the ergodic Gibbs measures on tilings of the whole plane.
We see that our Markov maps $\mathbb{L}_{\tau}$ and $\mathbb{R}_\tau$
can be viewed as 
the pre bulk limit versions 
of the two-sided Hammersley-type
processes of 
\cite{Toninelli2015-Gibbs}.

\medskip

Let us now discuss connections to the push-block
dynamics of \cite{BorFerr2008DF}. For completeness, 
let us recall its definition:
\begin{definition}[Push-block dynamics]
	\label{def:kpz_growth}
	Each vertical lozenge has an independent exponential
	clock of rate 1. When the clock rings, the lozenge
	tries to move to the right by one. 
	If it is blocked by a vertical lozenge from below
	(see the square mark in \Cref{fig:Gibbs_tilings}),
	then the jump is suppressed.
	If there are vertical lozenges above the one moving, then they also get pushed 
	to the right by one.
\end{definition}

The one-sided particular case of the Hammersley-type processes
is the push-block dynamics, up to rotating the picture 
by $\pi/3$ and focusing on the yellow lozenges in \Cref{fig:Gibbs_tilings} instead of the 
vertical (gray) ones.

Thus, in the bulk limit \Cref{thm:L_action_q_to_1_gibbs_general}
informally turns into the
statement that one can run the 
one-sided Hammersley-type process and the push-block
dynamics (both in terms of the vertical lozenges),
and the resulting process preserves ergodic Gibbs measures.
This statement follows from 
\cite{Toninelli2015-Gibbs}, as well as 
its rather straightforward generalization given next:

\begin{proposition}
	Running six one-sided Hammersley-type processes in parallel,
	where each individual process moves one type of lozenges in one of the 
	directions $e^{\mathbf{i}\pi k/3}$, $0\le k\le 5$, 
	at a specified rate $\upalpha_k \ge 0$,
	preserves ergodic Gibbs measures on 
	tilings of the whole plane.
\end{proposition}

\subsection{Branching graph perspective}
\label{sub:back_TASEP_and_branching_graph}

Recall that by $\mathcal{S}^c$ we denote the 
set of all interlacing arrays of infinite depth
which have many zeroes along the left border
(\Cref{def:issue_with_many_jumps_ok}).
Let us explain how the Markov maps $\mathbb{L}_{\tau}$
can be utilized to equip $\mathcal{S}^c\times \mathbb{R}$ with
a structure of an 
$\mathbb{R}$-graded
\emph{projective system}
in the sense of 
\cite{BorodinOlsh2011Bouquet}.
Projective systems
generalize branching graphs such as the Young graph, and the latter 
play a fundamental role in 
Asymptotic Representation Theory
\cite{VK81AsymptoticTheory},
\cite{borodin2016representations}.
The definitions and questions in this subsection 
are motivated by the connection to branching
graphs. 

\begin{remark}
	The set $\mathcal{S}^c\times \mathbb{R}$
	is ``larger'' than the more well-studied branching graphs.
	Namely, in the Young and Gelfand--Tsetlin graphs
	the vertices are indexed by Young diagrams and 
	signatures, respectively (a signature is a tuple
	$(\nu_1\ge \ldots\ge\nu_N )$, $\nu_i\in \mathbb{Z}$),
	while in $\mathcal{S}^c\times \mathbb{R}$ the 
	vertices are whole infinite collections of
	interlacing diagrams $\lambda^{(1)}\prec\lambda^{(2)}\prec\ldots $.
	This makes it hard to predict which properties
	of the Young and Gelfand--Tsetlin graphs
	could translate to $\mathcal{S}^c\times \mathbb{R}$.
\end{remark}

Let $M_s$, $s\in \mathbb{R}$, 
be probability measures on $\mathcal{S}$ supported by $\mathcal{S}^c$
(examples include the one-sided Schur measures as in \Cref{thm:L_action_q_to_1_gibbs_general}).
We call 
the family 
$\{M_s\}_{s\in \mathbb{R}}$ 
\emph{coherent} 
if for any $\tau\ge 0$ and $s\in \mathbb{R}$ we have
\begin{equation*}
	M_s\,
	\mathbb{L}_\tau 
	=M_{s-\tau}.
\end{equation*}
Coherent families are sometimes known as
\emph{entrance laws}, cf.
\cite{Dynkin1978entrance_laws}.
Clearly, coherent families form a 
convex set. Its extreme elements are, by definition,
those which cannot be represented as 
nontrivial convex combinations of other coherent families.

\begin{question}
	How to characterize extreme coherent
	families?
	Can every coherent family be represented
	in a unique way as a (continual) convex combination
	of the extremes?
\end{question}

Let us present an example of 
a coherent family based on 
Schur processes.
Take $M_s^{\mathrm{Schur}}=\mathbb{P}[\vec 1\mid \rho(s)]$, 
where $\rho(s)$ is a specialization with 
$\alpha_i^{\pm}(s)=\beta_i^-(s)=0$ for all $i$, $\gamma^-(s)=0$,
and other parameters
given by
\begin{equation*}
	\beta_i^+(s)=
	\frac{\beta_i^+ e^{s}}{1-\beta_i^++\beta_i^+e^{s}},\qquad 
	\gamma^+(s)=e^{s}\gamma^+,
\end{equation*}
where $\beta_i^{+}$ and $\gamma^{+}$ are fixed and satisfy 
\eqref{eq:alpha_pm_specialization}.
The fact that the family $\{ M_s^{\mathrm{Schur}} \}$ 
is indeed coherent follows
from \Cref{thm:L_action_q_to_1_gibbs_general}.

Let us discuss two particular examples.
\begin{itemize}
	\item 
		When $\gamma^+=1$ and all other parameters
		are zero, $M_s^{\mathrm{Schur}}$ is the family of single-time 
		distributions of the push-block dynamics
		under the 
		logarithmic time change
		$s=\log t$.
	\item 
		When $\beta_1^+=\beta\in(0,1)$ and
		all other parameters are zero, 
		the random interlacing array corresponding to $M_s^{\mathrm{Schur}}$
		has the form $\lambda^{(N)}=(1^{X_N}0^{N-X_N})$, 
		where $(X_1,X_2,\ldots )$ is the trajectory of the 
		simple random walk with steps $0,1$ taken with probabilities
		$1-\beta_1^+(s)$ and $\beta_1^+(s)$, respectively.
		The parameter $\beta_1^+(s)$ interpolates between $0$ and $1$ 
		at $s=-\infty$ and $s=+\infty$, respectively.
		The map $\mathbb{L}_\tau$ thus provides a coupling
		between these simple random walk trajectories
		with varying probability of up step.
		The concrete action of $\mathbb{L}_\tau$
		in this example
		leads to 
		\Cref{prop:Bernoulli_random_walk} formulated in the 
		Introduction.
\end{itemize}

\begin{question}
	Are the coherent families $\{ M_s^{\mathrm{Schur}}\}$
	extreme? Are there other interesting (extreme or 
	non-extreme) coherent families?
\end{question}

Let us focus on the case 
$\{M_s^{\mathrm{Schur}}\}$ with
$\gamma^+=1$ 
and all other parameters zero. 
The structure of a projective family
allows to define
for each $s\in \mathbb{R}$ the 
\emph{up-down Markov process}
on $\mathcal{S}^c$ which preserves each $M_s^{\mathrm{Schur}}$
(see \cite{Borodin2007}).
In more detail, the forward Markov generator is defined as
\begin{equation*}
	\mathbb{L}^{up}_{s,s+ds}(\boldsymbol \mu\to\boldsymbol\lambda)
	=
	\frac
	{M^{\mathrm{Schur}}_{s+ds}(\boldsymbol\lambda)}
	{M^{\mathrm{Schur}}_s(\boldsymbol\mu)}\;
	\mathbb{L}_{ds}(\boldsymbol\lambda\to \boldsymbol\mu),
	\qquad 
	\boldsymbol\lambda, \boldsymbol\mu\in \mathcal{S}^c.
\end{equation*}
One can check that this is not the same forward
evolution as the push-block 
generator (under any time change).
In particular, $\mathbb{L}_{s,s+ds}^{up}$
is time-inhomogeneous.
Therefore, the up-down Markov process 
arising from the branching graphs formalism
does not reduce 
(in restriction to the leftmost particles $\lambda^{(N)}_{N}$)
to the stationary
dynamics from \Cref{sec:equil_dyn}.

\begin{question}
	The up-down Markov chains 
	associated with distinguished non-extreme coherent families
	on well-studied branching
	graphs 
	converge to infinite-dimensional diffusions
	on the boundary
	(e.g., 
	\cite{Borodin2007},
	\cite{Petrov2007}). 
	Is there such a limit procedure 
	for the up-down processes associated
	with $\{M_s^{\mathrm{Schur}}\}$
	or other coherent families on $\mathcal{S}^c\times \mathbb{R}$?
\end{question}

Viewing $\mathcal{C}$ as a subset of $\mathcal{S}^c$,
one can similarly define the projective
system structure on $\mathcal{C}\times \mathbb{R}$
associated with the Markov maps $\mathbf{L}_\tau$
(\Cref{def:text_def_L_tau}). 
The restrictions of $\{M_s^{\mathrm{Schur}}\}$
form coherent families on $\mathcal{C}\times \mathbb{R}$,
and all the problems formulated in this subsection
also make sense for the 
smaller object $\mathcal{C}\times \mathbb{R}$.
Note that the up-down Markov chain on each floor $\mathcal{C}\times\{s\}$ 
with $\gamma^+=1$ (and all other parameters zero)
preserves the TASEP distribution $\upmu_{e^s}$, but is it not the same as the 
stationary dynamics
discussed in \Cref{sec:equil_dyn}.

\subsection{Lifting to additional parameters}
\label{sub:lifting_sHL}

The definition of the local Markov maps
$L^{(j)}_{\alpha}$ and $R^{(j)}_{\alpha}$
which randomly change a single level of an interlacing
array is inspired by the bijectivization of 
a degenerate case of
the Yang-Baxter
equation. Beyond this degenerate case 
associated with the Schur symmetric polynomials, 
the bijectivization can be developed to include
models associated with spin Hall-Littlewood 
or spin $q$-Whittaker symmetric functions
\cite{BufetovPetrovYB2017}, 
\cite{BufetovMucciconiPetrov2018}.
A scheme of symmetric functions 
is given in \Cref{fig:symm_functions_scheme}.

\begin{figure}[htpb]
	\centering
	\includegraphics{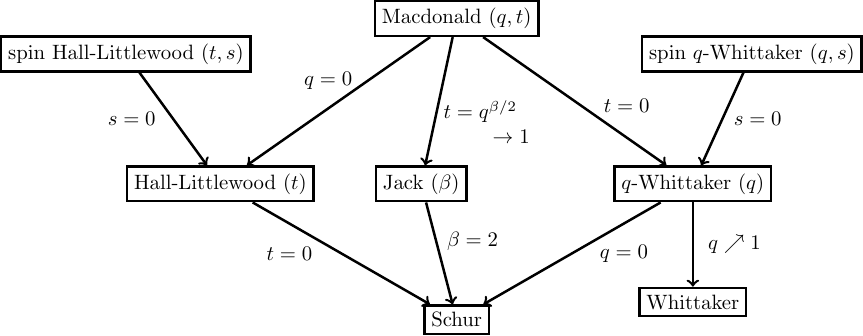}
	\caption{An hierarchy of symmetric functions.} 
	\label{fig:symm_functions_scheme}
\end{figure}

Let us consider three setups.
First, in the spin Hall-Littlewood case, the 
maps $L^{(j)}_\alpha$ and $R^{(j)}_\alpha$
can be obtained by considering sequences of 
local transitions given in
Figures 4 and 5 
in \cite{BufetovPetrovYB2017}
(see \Cref{sub:bijectivisation_mention} for more details).
Therefore, one can potentially define Markov 
maps preserving the class of probability
measures on interlacing arrays satisfying 
a version of the Gibbs property associated
with the spin Hall-Littlewood functions.
These Gibbs measures include the subclass of 
spin Hall-Littlewood processes. 
The Markov maps on the spin Hall-Littlewood processes could project 
(in a way similar to how $\mathbb{L}_\tau$ leads to $\mathbf{L}_\tau$)
into
maps acting nicely on distributions of the 
stochastic six-vertex model 
and the ASEP with step initial data.

\medskip

Second, on the spin $q$-Whittaker side the TASEP
is generalized 
to the $q$-TASEP
\cite{SasamotoWadati1998}, 
\cite{BorodinCorwin2011Macdonald}
and further to the $q$-Hahn TASEP
\cite{Povolotsky2013}, \cite{Corwin2014qmunu}.
A continuous-time
version of the $q$-Hahn TASEP 
can be found in 
\cite{Takeyama2014}, \cite{barraquand2015q}.

\begin{question}
	Do there exist Markov maps on 
	(spin) $q$-Whittaker processes
	mapping the time parameter in the $q$-TASEP or the 
	(continuous-time)
	$q$-Hahn TASEP backwards?
\end{question}

Finally, let us discuss
a setting which does not immediately fit into the scheme
of \Cref{fig:symm_functions_scheme}
but is also of interest. 
Configurations of the 
(not necessarily stochastic)
six-vertex model 
with the domain wall boundary conditions
(e.g., see \cite{reshetikhin2010lectures})
can be encoded as finite depth
interlacing arrays of strict partitions with 
fixed top row. 
The Yang-Baxter equation swapping spectral
parameters in this model can potentially be bijectivised
in the same way as in \cite{BufetovPetrovYB2017},
which should lead to 
Markov maps acting nicely on the distribution
of the six-vertex model. (In the Schur case this
is described in \Cref{sub:c_Gibbs_finite}.)

\begin{question}
	Can these Markov maps be taken to the continuous-time
	limit similarly to the $q\to1$ limit described in
	\Cref{sec:limit_q_1}? If this is possible, this would lead to 
	a new non-local sampling algorithm for the 
	distribution of the 
	homogeneous (i.e., with equal spectral parameters)
	six-vertex model with domain wall boundary
	conditions. The bulk limit of this latter algorithm should 
	presumably coincide with the Markov process
	from 
	\cite{BorodinBufetov2015}
	preserving the distribution of 
	the six-vertex model on a torus.
\end{question}

\bibliographystyle{alpha}
\bibliography{bib}

\begin{thebibliography}{BPSS96}

\bibitem[ABB19]{ABB2018stochasticization}
A.~Aggarwal, A.~Borodin, and A.~Bufetov.
\newblock Stochasticization of solutions to the yang-baxter equation.
\newblock {\em Annales Henri Poincare}, 20(8):2495--2554, 2019.
\newblock arXiv:1810.04299 [math.PR].

\bibitem[AD95]{aldous1995hammersley}
D.~Aldous and P.~Diaconis.
\newblock {Hammersley's interacting particle process and longest increasing
  subsequences}.
\newblock {\em Probab. Theory Relat. Fields}, 103(2):199--213, 1995.

\bibitem[AG05]{andjel2005long}
E.~Andjel and H.~Guiol.
\newblock Long-range exclusion processes, generator and invariant measures.
\newblock {\em Ann. Probab.}, 33(6):2314--2354, 2005.
\newblock arXiv:math/0411655 [math.PR].

\bibitem[BB17]{BorodinBufetov2015}
A.~Borodin and A.~Bufetov.
\newblock {An irreversible local Markov chain that preserves the six vertex
  model on a torus}.
\newblock {\em Ann. Inst. H. Poincar\'e B}, 53(1):451--463, 2017.
\newblock arXiv:1509.05070 [math-ph].

\bibitem[BB18]{balazs2018product}
M.~Balazs and R.~Bowen.
\newblock Product blocking measures and a particle system proof of the jacobi
  triple product.
\newblock {\em Ann. Inst. H. Poincar\'e B}, 54(1):514--528, 2018.
\newblock arXiv:1606.00639 [math.PR].

\bibitem[BC14]{BorodinCorwin2011Macdonald}
A.~Borodin and I.~Corwin.
\newblock Macdonald processes.
\newblock {\em Probab. Theory Relat. Fields}, 158:225--400, 2014.
\newblock arXiv:1111.4408 [math.PR].

\bibitem[BC16]{barraquand2015q}
G.~Barraquand and I.~Corwin.
\newblock {The $q$-Hahn asymmetric exclusion process}.
\newblock {\em Annals of Applied Probability}, 26(4):2304--2356, 2016.
\newblock arXiv:1501.03445 [math.PR].

\bibitem[BF87]{BenassiFouque1987}
Albert Benassi and Jean-Pierre Fouque.
\newblock Hydrodynamical limit for the asymmetric simple exclusion process.
\newblock {\em The Annals of Probability}, 15(2):546--560, 1987.

\bibitem[BF08]{BorFerr08push}
A.~Borodin and P.~Ferrari.
\newblock {Large time asymptotics of growth models on space-like paths I:
  PushASEP}.
\newblock {\em Electron. J. Probab.}, 13:1380--1418, 2008.
\newblock arXiv:0707.2813 [math-ph].

\bibitem[BF14]{BorFerr2008DF}
A.~Borodin and P.~Ferrari.
\newblock Anisotropic growth of random surfaces in 2+1 dimensions.
\newblock {\em Commun. Math. Phys.}, 325:603--684, 2014.
\newblock arXiv:0804.3035 [math-ph].

\bibitem[BG13]{BG2011non}
A.~Borodin and V.~Gorin.
\newblock {Markov processes of infinitely many nonintersecting random walks}.
\newblock {\em Probab. Theory Relat. Fields}, 155(3-4):935--997, 2013.
\newblock arXiv:1106.1299 [math.PR].

\bibitem[BG16]{BorodinGorinSPB12}
A.~Borodin and V.~Gorin.
\newblock Lectures on integrable probability.
\newblock In {\em Probability and Statistical Physics in St.\ Petersburg},
  volume~91 of {\em Proceedings of Symposia in Pure Mathematics}, pages
  155--214. AMS, 2016.
\newblock arXiv:1212.3351 [math.PR].

\bibitem[BG18]{basu2018time}
R.~Basu and S.~Ganguly.
\newblock Time correlation exponents in last passage percolation.
\newblock {\em arXiv preprint}, 2018.
\newblock arXiv:1807.09260 [math.PR].

\bibitem[BGH19]{basu_ganguly_hammond_2019}
R.~Basu, S.~Ganguly, and A.~Hammond.
\newblock Fractal geometry of airy$_2$ processes coupled via the airy sheet.
\newblock {\em arXiv preprint}, 2019.
\newblock arXiv:1904.01717 [math.PR].

\bibitem[BGR10]{borodin-gr2009q}
A.~Borodin, V.~Gorin, and E.~Rains.
\newblock {q-Distributions on boxed plane partitions}.
\newblock {\em Selecta Math.}, 16(4):731--789, 2010.
\newblock arXiv:0905.0679 [math-ph].

\bibitem[BK08]{BorodinKuan2007U}
A.~Borodin and J.~Kuan.
\newblock {Asymptotics of Plancherel measures for the infinite-dimensional
  unitary group}.
\newblock {\em Adv. Math.}, 219(3):894--931, 2008.
\newblock arXiv:0712.1848 [math.RT].

\bibitem[BL19]{baik2019multipoint}
J.~Baik and Z.~Liu.
\newblock Multipoint distribution of periodic tasep.
\newblock {\em Jour. AMS}, 2019.
\newblock arXiv:1710.03284 [math.PR].

\bibitem[BMP19]{BufetovMucciconiPetrov2018}
A.~Bufetov, M.~Mucciconi, and L.~Petrov.
\newblock Yang-baxter random fields and stochastic vertex models.
\newblock {\em arXiv preprint}, 2019.
\newblock arXiv:1905.06815 [math.PR]. To appear in Adv. Math.

\bibitem[BO09]{Borodin2007}
A.~Borodin and G.~Olshanski.
\newblock Infinite-dimensional diffusions as limits of random walks on
  partitions.
\newblock {\em Probab. Theory Relat. Fields}, 144(1):281--318, 2009.
\newblock arXiv:0706.1034 [math.PR].

\bibitem[BO12]{BorodinOlsh2011GT}
A.~Borodin and G.~Olshanski.
\newblock {The boundary of the Gelfand-Tsetlin graph: A new approach}.
\newblock {\em Adv. Math.}, 230:1738--1779, 2012.
\newblock arXiv:1109.1412 [math.CO].

\bibitem[BO13]{BorodinOlsh2011Bouquet}
A.~Borodin and G.~Olshanski.
\newblock {The Young bouquet and its boundary}.
\newblock {\em Moscow Mathematical Journal}, 13(2):193--232, 2013.
\newblock arXiv:1110.4458 [math.RT].

\bibitem[BO16]{borodin2016representations}
A.~Borodin and G.~Olshanski.
\newblock {\em Representations of the Infinite Symmetric Group}, volume 160.
\newblock Cambridge University Press, 2016.

\bibitem[Bor11]{Borodin2010Schur}
A.~Borodin.
\newblock {Schur dynamics of the Schur processes}.
\newblock {\em Adv. Math.}, 228(4):2268--2291, 2011.
\newblock arXiv:1001.3442 [math.CO].

\bibitem[Bor17]{Borodin2014vertex}
A.~Borodin.
\newblock On a family of symmetric rational functions.
\newblock {\em Adv. Math.}, 306:973--1018, 2017.
\newblock arXiv:1410.0976 [math.CO].

\bibitem[BP14]{BorodinPetrov2013Lect}
A.~Borodin and L.~Petrov.
\newblock {Integrable probability: From representation theory to Macdonald
  processes}.
\newblock {\em Probab. Surv.}, 11:1--58, 2014.
\newblock arXiv:1310.8007 [math.PR].

\bibitem[BP16]{BorodinPetrov2013NN}
A.~Borodin and L.~Petrov.
\newblock {Nearest neighbor Markov dynamics on Macdonald processes}.
\newblock {\em Adv. Math.}, 300:71--155, 2016.
\newblock arXiv:1305.5501 [math.PR].

\bibitem[BP19]{BufetovPetrovYB2017}
A.~Bufetov and L.~Petrov.
\newblock {Yang-Baxter field for spin Hall-Littlewood symmetric functions}.
\newblock {\em Forum Math. Sigma}, 7:e39, 2019.
\newblock arXiv:1712.04584 [math.PR].

\bibitem[BPSS96]{Brankov_1996}
J.~Brankov, V.~Priezzhev, A.~Schadschneider, and M.~Schreckenberg.
\newblock The kasteleyn model and a cellular automaton approach to traffic
  flow.
\newblock {\em Journal of Physics A: Mathematical and General},
  29(10):L229--L235, may 1996.
\newblock arXiv:cond-mat/9512062.

\bibitem[BS18]{belitsky2016self}
V.~Belitsky and G.~Sch{\"u}tz.
\newblock Self-duality and shock dynamics in the $ n $-component priority asep.
\newblock {\em Stochast. Proc. Appl.}, 128(4):1165--1207, 2018.
\newblock arXiv:1606.04587 [math.PR].

\bibitem[BW17]{BorodinWheelerSpinq}
A.~Borodin and M.~Wheeler.
\newblock Spin $q$-whittaker polynomials.
\newblock {\em arXiv preprint}, 2017.
\newblock arXiv:1701.06292 [math.CO].

\bibitem[CFS18]{chhita2018limit}
S.~Chhita, P.~Ferrari, and H.~Spohn.
\newblock Limit distributions for kpz growth models with spatially homogeneous
  random initial conditions.
\newblock {\em Ann. Appl. Probab.}, 28(3):1573--1603, 2018.
\newblock arXiv:1611.06690 [math.PR].

\bibitem[CH14]{corwin2014brownian}
I.~Corwin and A.~Hammond.
\newblock Brownian gibbs property for airy line ensembles.
\newblock {\em Inventiones mathematicae}, 195(2):441--508, 2014.
\newblock arXiv:1108.2291 [math.PR].

\bibitem[CKP01]{CohnKenyonPropp2000}
H.~Cohn, R.~Kenyon, and J.~Propp.
\newblock A variational principle for domino tilings.
\newblock {\em Jour. AMS}, 14(2):297--346, 2001.
\newblock arXiv:math/0008220 [math.CO].

\bibitem[Cor12]{CorwinKPZ}
I.~Corwin.
\newblock {The Kardar-Parisi-Zhang equation and universality class}.
\newblock {\em Random Matrices Theory Appl.}, 1:1130001, 2012.
\newblock arXiv:1106.1596 [math.PR].

\bibitem[Cor14]{Corwin2014qmunu}
I.~Corwin.
\newblock {The $q$-Hahn Boson process and $q$-Hahn TASEP}.
\newblock {\em Int. Math. Res. Notices}, (rnu094), 2014.
\newblock arXiv:1401.3321 [math.PR].

\bibitem[DFG19]{DiFran2019qvol}
P.~Di~Francesco and E.~Guitter.
\newblock A tangent method derivation of the arctic curve for q-weighted paths
  with arbitrary starting points.
\newblock {\em Jour. Phys. A}, 52(11):115205, 2019.
\newblock arXiv:1810.07936 [math-ph].

\bibitem[DOV18]{directed_landscape}
D.~Dauvergne, J.~Ortmann, and B.~Virag.
\newblock The directed landscape.
\newblock {\em arXiv preprint}, 2018.
\newblock arXiv:1812.00309 [math.PR].

\bibitem[Dyn78]{Dynkin1978entrance_laws}
E.~Dynkin.
\newblock Sufficient statistics and extreme points.
\newblock {\em {Ann. Probab.}}, 6(705-730), 1978.

\bibitem[Fer96]{ferrari1996}
P.A. Ferrari.
\newblock Limit theorems for tagged particles.
\newblock {\em Markov Process. Related Fields}, 2(1):17--40, 1996.

\bibitem[Fer08]{Ferrari_Airy_Survey}
P.~Ferrari.
\newblock {The universal Airy$_1$ and Airy$_2$ processes in the Totally
  Asymmetric Simple Exclusion Process}.
\newblock In J.~Baik, T.~Kriecherbauer, L.-C. Li, K.~T.-R. McLaughlin, and
  C.~Tomei, editors, {\em {Integrable Systems and Random Matrices: In Honor of
  Percy Deift}}, Contemporary Math., pages 321--332. AMS, 2008.
\newblock arXiv:math-ph/0701021.

\bibitem[Fer18]{ferrari_PA2018tasep}
P.A. Ferrari.
\newblock Tasep hydrodynamics using microscopic characteristics.
\newblock {\em Probability Surveys}, 15:1--27, 2018.

\bibitem[FM05]{FerrariMartin2005}
P.A. Ferrari and J.~Martin.
\newblock Multiclass processes, dual points and m/m/1 queues.
\newblock {\em arXiv preprint}, 2005.
\newblock arXiv:math-ph/0509045.

\bibitem[FO19]{ferrari2019time}
P.~Ferrari and A.~Occelli.
\newblock Time-time covariance for last passage percolation with generic
  initial profile.
\newblock {\em Mathematical Physics, Analysis and Geometry}, 22(22):1, 2019.
\newblock arXiv:1807.02982 [math-ph].

\bibitem[FS16]{ferrari2016time}
P.~Ferrari and H.~Spohn.
\newblock On time correlations for kpz growth in one dimension.
\newblock {\em SIGMA}, 12(074), 2016.
\newblock arXiv:1602.00486 [math-ph].

\bibitem[Ful97]{fulton1997young}
W.~Fulton.
\newblock {\em {Young Tableaux with Applications to Representation Theory and
  Geometry}}.
\newblock Cambridge University Press, 1997.

\bibitem[GO16]{gorin2016quantization}
V.~Gorin and G.~Olshanski.
\newblock A quantization of the harmonic analysis on the infinite-dimensional
  unitary group.
\newblock {\em J. Funct. Anal.}, 270(1):375--418, 2016.
\newblock arXiv:1504.06832 [math.RT].

\bibitem[Gor12]{Gorin2010q}
V.~Gorin.
\newblock {The q-Gelfand-Tsetlin graph, Gibbs measures and q-Toeplitz
  matrices}.
\newblock {\em Adv. Math.}, 229(1):201--266, 2012.
\newblock arXiv:1011.1769 [math.RT].

\bibitem[Gui97]{guiol1997resultat}
H.~Guiol.
\newblock Un r\'esultat pour le processus d'exclusion \`a longue port\'ee [a
  result for the long-range exclusion process].
\newblock {\em Annales de l'Institut Henri Poincare (B) Probability and
  Statistics}, 33(4):387--405, 1997.

\bibitem[Ham72]{hammersley1972few}
JM~Hammersley.
\newblock {A few seedlings of research}.
\newblock In {\em Proc. Sixth Berkeley Symp. Math. Statist. and Probability},
  volume~1, pages 345--394, 1972.

\bibitem[Joh00]{johansson2000shape}
K.~Johansson.
\newblock {Shape fluctuations and random matrices}.
\newblock {\em Commun. Math. Phys.}, 209(2):437--476, 2000.
\newblock arXiv:math/9903134 [math.CO].

\bibitem[Joh18]{johansson2018two}
K.~Johansson.
\newblock The two-time distribution in geometric last-passage percolation.
\newblock {\em arXiv preprint}, 2018.
\newblock arXiv:1802.00729 [math.PR].

\bibitem[Joh19]{Johansson_two_time_shorter}
K.~Johansson.
\newblock The long ans short time asymptotics of the two-time distribution in
  local random growth.
\newblock {\em arXiv preprint}, 2019.
\newblock arXiv:1904.08195 [math.PR].

\bibitem[JR20]{JohanssonRahman2019}
K.~Johansson and M.~Rahman.
\newblock Multi-time distribution in discrete polynuclear growth.
\newblock {\em Comm. Pure Appl. Math., to appear}, 2020.
\newblock arXiv:1906.01053 [math.PR].

\bibitem[KO07]{OkounkovKenyon2007Limit}
R.~Kenyon and A.~Okounkov.
\newblock Limit shapes and the complex {B}urgers equation.
\newblock {\em Acta Math.}, 199(2):263--302, 2007.
\newblock arXiv:math-ph/0507007.

\bibitem[KOO98]{Kerov1998}
S.~Kerov, A.~Okounkov, and G.~Olshanski.
\newblock {The boundary of Young graph with Jack edge multiplicities}.
\newblock {\em Int. Math. Res. Notices}, 4:173--199, 1998.
\newblock arXiv:q-alg/9703037.

\bibitem[Lig75]{Liggett1975}
Thomas~M Liggett.
\newblock Ergodic theorems for the asymmetric simple exclusion process.
\newblock {\em Transactions of the American Mathematical Society},
  213:237--261, 1975.

\bibitem[Lig05]{Liggett1985}
T.~Liggett.
\newblock {\em {Interacting Particle Systems}}.
\newblock Springer-Verlag, Berlin, 2005.

\bibitem[LP19]{PetrovLi2019simul}
H.~Li and L.~Petrov.
\newblock Computer simulation of the backwards tasep evolution, 2019.

\bibitem[Mac95]{Macdonald1995}
I.G. Macdonald.
\newblock {\em Symmetric functions and {H}all polynomials}.
\newblock Oxford University Press, 2nd edition, 1995.

\bibitem[MG69]{MacdonaldGibbsASEP1969}
C.~MacDonald and J.~Gibbs.
\newblock Concerning the kinetics of polypeptide synthesis on polyribosomes.
\newblock {\em Biopolymers}, 7(5):707--725, 1969.

\bibitem[MGP68]{macdonald1968bioASEP}
C.~MacDonald, J.~Gibbs, and A.~Pipkin.
\newblock Kinetics of biopolymerization on nucleic acid templates.
\newblock {\em Biopolymers}, 6(1):1--25, 1968.

\bibitem[MP17]{mkrtchyan2017gue}
S.~Mkrtchyan and L.~Petrov.
\newblock Gue corners limit of q-distributed lozenge tilings.
\newblock {\em Electron. J. Probab.}, 22(101):24 pp., 2017.
\newblock arXiv:1703.07503 [math.PR].

\bibitem[MQR17]{matetski2017kpz}
K.~Matetski, J.~Quastel, and D.~Remenik.
\newblock The kpz fixed point.
\newblock {\em arXiv preprint}, 2017.
\newblock arXiv:1701.00018 [math.PR].

\bibitem[O'C03a]{OConnell2003Trans}
N.~O'Connell.
\newblock {A path-transformation for random walks and the Robinson-Schensted
  correspondence}.
\newblock {\em Trans. AMS}, 355(9):3669--3697, 2003.

\bibitem[O'C03b]{OConnell2003}
N.~O'Connell.
\newblock {Conditioned random walks and the RSK correspondence}.
\newblock {\em J. Phys. A}, 36(12):3049--3066, 2003.

\bibitem[Oko01]{okounkov2001infinite}
A.~Okounkov.
\newblock {Infinite wedge and random partitions}.
\newblock {\em Selecta Math.}, 7(1):57--81, 2001.
\newblock arXiv:math/9907127 [math.RT].

\bibitem[OR03]{okounkov2003correlation}
A.~Okounkov and N.~Reshetikhin.
\newblock {Correlation function of Schur process with application to local
  geometry of a random 3-dimensional Young diagram}.
\newblock {\em Jour. AMS}, 16(3):581--603, 2003.
\newblock arXiv:math/0107056 [math.CO].

\bibitem[Pet09]{Petrov2007}
L.~Petrov.
\newblock A two-parameter family of infinite-dimensional diffusions in the
  {K}ingman simplex.
\newblock {\em Functional Analysis and Its Applications}, 43(4):279--296, 2009.
\newblock arXiv:0708.1930 [math.PR].

\bibitem[Pet14]{Petrov2012GT}
L.~Petrov.
\newblock {The Boundary of the Gelfand-Tsetlin Graph: New Proof of
  Borodin-Olshanski's Formula, and its q-analogue}.
\newblock {\em Mosc. Math. J.}, 14(1):121--160, 2014.
\newblock arXiv:1208.3443 [math.CO].

\bibitem[Pov13]{Povolotsky2013}
A.~Povolotsky.
\newblock {On integrability of zero-range chipping models with factorized
  steady state}.
\newblock {\em J. Phys. A}, 46:465205, 2013.
\newblock arXiv:1308.3250 [math-ph].

\bibitem[PZ19]{PetrovZhang2019simul}
L.~Petrov and E.~Zhang.
\newblock Computer simulations of dynamics on q-vol lozenge tilings inverting
  the parameter q, 2019.

\bibitem[QS15]{QuastelSpohnKPZ2015}
J.~Quastel and H.~Spohn.
\newblock {The one-dimensional KPZ equation and its universality class}.
\newblock {\em J. Stat. Phys}, 160(4):965--984, 2015.
\newblock arXiv:1503.06185 [math-ph].

\bibitem[Res10]{reshetikhin2010lectures}
N.~Reshetikhin.
\newblock Lectures on the integrability of the 6-vertex model.
\newblock In {\em {Exact Methods in Low-dimensional Statistical Physics and
  Quantum Computing}}, pages 197--266. Oxford Univ. Press, 2010.
\newblock arXiv:1010.5031 [math-ph].

\bibitem[Rom15]{romik2015surprising}
D.~Romik.
\newblock {\em The surprising mathematics of longest increasing subsequences}.
\newblock Cambridge University Press, 2015.

\bibitem[Ros81]{Rost1981}
H.~Rost.
\newblock Nonequilibrium behaviour of a many particle process: density profile
  and local equilibria.
\newblock {\em Z. Wahrsch. Verw. Gebiete}, 58(1):41--53, 1981.

\bibitem[RS05]{rakos2005current}
A.~R{\'a}kos and G.~Sch{\"u}tz.
\newblock {Current distribution and random matrix ensembles for an integrable
  asymmetric fragmentation process}.
\newblock {\em J. Stat. Phys}, 118(3-4):511--530, 2005.
\newblock arXiv:cond-mat/0405464 [cond-mat.stat-mech].

\bibitem[Sep99]{seppalainen1999existence}
T.~Sepp\"{a}l\"{a}inen.
\newblock Existence of hydrodynamics for the totally asymmetric simple
  k-exclusion process.
\newblock {\em {Ann. Probab.}}, 27(1):361--415, 1999.

\bibitem[She05]{Sheffield2008}
S.~Sheffield.
\newblock Random surfaces.
\newblock {\em Ast\'erisque}, 304, 2005.
\newblock arXiv:math/0304049 [math.PR].

\bibitem[Spi70]{Spitzer1970}
F.~Spitzer.
\newblock {Interaction of Markov processes}.
\newblock {\em Adv. Math.}, 5(2):246--290, 1970.

\bibitem[Spo]{Spohn2012}
H.~Spohn.
\newblock {KPZ Scaling Theory and the Semi-discrete Directed Polymer Model}.
\newblock arXiv:1201.0645 [cond-mat.stat-mech].

\bibitem[SW98]{SasamotoWadati1998}
T.~Sasamoto and M.~Wadati.
\newblock {Exact results for one-dimensional totally asymmetric diffusion
  models}.
\newblock {\em J. Phys. A}, 31:6057--6071, 1998.

\bibitem[Tak14]{Takeyama2014}
Y.~Takeyama.
\newblock {A deformation of affine Hecke algebra and integrable stochastic
  particle system}.
\newblock {\em J. Phys. A}, 47(46):465203, 2014.
\newblock arXiv:1407.1960 [math-ph].

\bibitem[Ton17]{Toninelli2015-Gibbs}
F.~Toninelli.
\newblock A $(2+ 1) $-dimensional growth process with explicit stationary
  measures.
\newblock {\em {Ann. Probab.}}, 45(5):2899--2940, 2017.
\newblock arXiv:1503.05339 [math.PR].

\bibitem[VK81]{VK81AsymptoticTheory}
A.~Vershik and S.~Kerov.
\newblock Asymptotic theory of the characters of the symmetric group.
\newblock {\em Funktsional. Anal. i Prilozhen.}, 15(4):15--27, 96, 1981.

\bibitem[VK86]{Vershik1986}
A.~Vershik and S.~Kerov.
\newblock The characters of the infinite symmetric group and probability
  properties of the {R}obinson-{S}hensted-{K}nuth algorithm.
\newblock {\em SIAM J. Alg. Disc. Math.}, 7(1):116--124, 1986.

\end{thebibliography}

\end{document}